\documentclass[11pt,a4paper]{article}
\usepackage[utf8]{inputenc}
\usepackage{amsmath}
\usepackage{amsfonts}
\usepackage[colorlinks,linkcolor=blue,anchorcolor=blue,citecolor=blue]{hyperref}
\usepackage{amssymb}
\usepackage{extarrows}
\usepackage{amsthm}
\usepackage{comment}
\usepackage{hyperref}
\usepackage{enumerate}
\PassOptionsToPackage{normalem}{ulem}
\usepackage{ulem}
\usepackage[margin=1in]{geometry} 
\usepackage{array}
\usepackage{cases}
\usepackage{mathrsfs}
\usepackage{longtable}
\allowdisplaybreaks
\newcolumntype{L}[1]{>{\raggedright\let\newline\\\arraybackslash\hspace{0pt}}m{#1}}
\newcolumntype{C}[1]{>{\centering\let\newline\\\arraybackslash\hspace{0pt}}m{#1}}
\newcolumntype{R}[1]{>{\raggedleft\let\newline\\\arraybackslash\hspace{0pt}}m{#1}}

\usepackage{natbib}
\usepackage{color}
\usepackage{dsfont}
\usepackage{marginnote}
\usepackage{enumitem}
\usepackage{graphicx}
\theoremstyle{plain}
\newtheorem{theorem}{\protect Theorem}[section]
\newtheorem{prop}[theorem]{\protect Proposition}
\newtheorem{definition}[theorem]{\protect Definition}

\newtheorem{lemma}[theorem]{\protect Lemma}
\newtheorem{remark}[theorem]{\protect Remark}
\newtheorem{ass}{\protect Assumption}
\newtheorem{corollary}[theorem]{\protect Corollary}

\def\d{\mathrm{d}}
\def\ie{\mathrm{i.e.}}
\def\as{\mathrm{a.s.}}
\def\RV{\mathrm{r.v.}}

\newcommand{\comp}{\mathrm{comp}}

\newcommand{\R}{\mathbb{R}}
\newcommand{\E}{\mathbb{E}}

\newcommand{\T}{\top}
\newcommand{\C}{\mathcal{C}}
\newcommand{\F}{\mathcal{F}}
\newcommand{\Fb}{\mathbb{F}}
\newcommand{\Pb}{\mathbb{P}}
\newcommand{\Pc}{\mathcal{P}}

\newcommand{\tr}{\mathrm{tr}}
\newcommand{\Law}{\mathcal{L}}

\newcommand{\dist}{\mathrm{dist}}
\newcommand{\Gr}{\mathrm{Gr}}

\newcommand{\pa}{\partial}

\newcommand{\Mc}{\mathcal{M}}

\newcommand{\U}{\mathscr{U}}

\newcommand{\tb}{\tilde b}
\newcommand{\ts}{\tilde\sigma}
\newcommand{\tf}{\tilde f}

\newcommand{\assref}[1]{\hyperref[#1]{Assumption \ref*{#1}}}
\newcommand{\thmref}[1]{\hyperref[#1]{Theorem \ref*{#1}}}
\newcommand{\propref}[1]{\hyperref[#1]{Proposition \ref*{#1}}}
\newcommand{\remref}[1]{\hyperref[#1]{Remark \ref*{#1}}}
\newcommand{\lemref}[1]{\hyperref[#1]{Lemma \ref*{#1}}}
\newcommand{\defref}[1]{\hyperref[#1]{Definition \ref*{#1}}}
\newcommand{\corref}[1]{\hyperref[#1]{Corollary \ref*{#1}}}
\newcommand{\equref}[1]{\hyperref[#1]{(\ref*{#1})}}
\newcommand{\exaref}[1]{\hyperref[#1]{Example \ref*{#1}}}

\title{Mean Field Game of Controls with State Reflections: Existence and Limit Theory}

\author{Lijun Bo \thanks{Email: lijunbo@ustc.edu.cn, School of Mathematics and Statistics, Xidian University, Xi'an, 710126, China.}
	\and
	Jingfei Wang \thanks{Email:wjf2104296@mail.ustc.edu.cn, School of Mathematical Sciences, University of Science and Technology of China, Hefei, 230026, China.}
	\and
	Xiang Yu \thanks{Email: xiang.yu@polyu.edu.hk, Department of Applied Mathematics, The Hong Kong Polytechnic University, Kowloon, Hong Kong, China.}
}
\date{ }

\begin{document}
	
	\maketitle
	
	\vspace{-0.2in}
	\begin{abstract}
		This paper studies mean field game (MFG) of controls by featuring the state-control joint distribution and the reflected state process at an exogenous stochastic reflection boundary. We contribute to the literature with a customized relaxed formulation and some new compactification arguments to establish the existence of a Markovian mean field equilibrium (MFE) in the weak sense. We consider an enlarged canonical space, utilizing the dynamic Skorokhod mapping, to accommodate the stochastic reflection boundary process. A fixed-point argument on the extended space using an extension transformation technique is developed to tackle challenges from the joint measure flow of the state and the relaxed control that may not be continuous. Furthermore, the bidirectional connections between the MFG and the $N$-player game are  established. We first show that any weak limit of empirical measures induced by $\boldsymbol{\epsilon}$-Nash equilibria in $N$-player games must be supported exclusively on the set of relaxed mean field equilibria, analogous to the propagation of chaos in mean field control problems. We then prove the convergence result that a Markovian MFE in the weak sense can be approximated by a sequence of constructed $\boldsymbol{\epsilon}$-Nash equilibria in the weak sense in $N$-player games when $N$ tends to infinity.\\		
		\ \\
		\textbf{Keywords}: Mean field game of controls, reflected state process, compactification method,  mean field equilibrium in the weak sense, limit theory, approximate Nash equilibrium in the weak sense.\\
		\ \\
		\textbf{MSC2020}: 49N80, 60H30, 93E20, 60G07
	\end{abstract}
	
\vspace{0.2in}	
\section{Introduction}
Decision making for a large population of interacting agents has wide applications in finance, energy management, epidemic control, traffic flow, and among others. To overcome the coupled influence in decision making by all agents and the curse of dimensionality, the mean field game (MFG) theory has been proposed to effectively approximate the rational behavior of agents by assuming that each agent interacts with the aggregated effect of the population. The pioneer studies in Lasry and Lions~\cite{Lasry07} and Huang et al.~\cite{Huang06} provided insightful ways to study the approximate Nash equilibria for large stochastic systems using the more tractable mean field equilibrium solution, which spurred fast-growing developments in MFG theory and applications in the past decade. 
	
In classical MFG problems, the mean field interaction is typically encoded by the distribution of the state process. The problem then consists of solving an optimal control problem for a representative agent and finding a fixed point as the consistency condition of the state distribution flow. The first core question in MFG theory is to investigate the existence of mean field equilibrium (MFE). Several approaches have been proposed and developed in different context in the literature. In early groundbreaking works initiated by Lasry and Lions \cite{Lions,Lasry07}, the HJB equation is employed to describe the evolution of the value function of the representative agent, while the Fokker-Planck equation models the evolution of the distribution of the state variables across the population. This method of coupled PDE system also allows for various local mean field interactions, see among \cite{Cardaliaguet, Gomes, Gomes1}. Later, a probabilistic approach based on Pontryagin maximum principle is proposed for MFG problems by examining the forward-backward stochastic differential equations (FBSDEs), where the forward equation governs the evolution of the state process and the backward equation corresponds to the adjoint process coupled via the law of the forward state process. For example, Bensoussan et al. \cite{Bensoussan} studied this approach in the linear-quadratic case.  Carmona and Delarue \cite{Carmona0} advanced this method to a more general setting. Apart from the previous two mainstream methods, Lacker \cite{Lacker} proposed a probabilistic compactification method, originating from the pioneer work of relaxed control formulation in El Karoui et al. \cite{Karoui1} and Haussmann and Lepeltier \cite{Haussmann}, to establish the existence of the Markovian MFE in a general setup. The compactification arguments tackle the law of the controlled system directly and allow for non-unique optimal controls by utilizing a set-valued fixed-point theorem (such as Kakutani’s fixed-point theorem). Recently, the compactification method has been refined as a powerful tool to analyze the existence of MFE for MFG problems in various contexts, including MFG with common noise in Carmona et al. \cite{Carmona}; MFG with controlled jumps in Benazzoli et al. \cite{Benazzoli}; MFG with absorption in Campi and Fisher \cite{Campi}; MFG with past absorptions and surviving players in Campi et al. \cite{Campi1}; MFG with finite states in Cecchin and Fisher \cite{Cecchin}; MFG with singular controls in Fu and Horst \cite{Fu}, and among others.
	
Another important question in MFG theory is to discuss the connections between the $N$-player game and the MFG problem via the so-called limit theory such as can one identify all possible limits from the $N$-player approximate Nash equilibria as $N$ tends to infinity, and can one use a MFE to construct an approximate $N$-player Nash equilibria? An increasing number of efforts can be found in the literature concerning the limit theory. For instance, the forward direction of convergence has been studied by \cite{Feleqi, Lasry07} in restrictive cases and was later investigated comprehensively in \cite{Fischer,Lakcer3}, primarily focusing on the convergence of open-loop Nash equilibria. Cardaliaguet et al. \cite{Cardaliaguet} made an important breakthrough to discuss the convergence of closed-loop equilibria on the strength of the master equation, which requires the uniqueness of the MFE. Lacker \cite{Lacker4} established the limit theory for closed-loop equilibria in a weak formulation using the Girsanov transform. Djete \cite{Djete1} introduced the concept of measure-valued MFG solutions and proved the convergence of open-loop Nash equilibria to a measure-valued MFG solution in the framework of MFG of control with common noise. It is worth noting that the framework considered in Jackson and Tangpi \cite{JTangpi24} concerns open-loop Nash equilibria. They adopted the maximum principle to recast the convergence problem as a question of forward-backward propagation of chaos and obtained the convergence result for MFG with common noise and controlled volatility under the displacement monotonicity condition. The reverse direction of connection has also been extensively studied in the literature. To name a few, Carmona and Delarue \cite{Carmona0} established the approximating procedure by employing an FBSDE approach.  Benazzoli et al. \cite{Benazzoli1} constructed the approximate open-loop Nash equilibria for jump diffusion dynamics. Carmona and Lacker \cite{Carmona1} examined the approximate Nash-equilibira in the framework of MFG of controls, where the law of the control appears in the running cost function. Djete \cite{Djete1} constructed the approximate Nash equilibria using a measure-valued MFG solution in the presence of common noise.
	
	In the present paper, we aim to contribute some new theoretical advances to the compactification method and limit theory for
	MFG problems when mean field interactions appear via the joint distribution of the state and the control as well as the state process exhibits reflections along some stochastic boundary.

\subsection{Related literature}
	
Compared with classical MFG problems, MFG of controls with mean-field interactions depending on the joint distribution of the state and the control are relatively under-explored. Existing studies along this direction mainly focus on the existence and uniqueness results. Some restrictive model assumptions such as convexity and continuity are often imposed to facilitate the proofs. For instance, \cite{Cardaliaguet2, Kobeissi, Bonnans} handled the existence of MFEs by employing the PDE approach as in classical MFG problems. Achdou and Kobeissi \cite{Achdou} developed some numerical approximations via finite difference methods for the PDE system arising from the MFG of controls. The mean field FBSDE approach was developed in Carmona and Lacker \cite{Carmona1} for uncontrolled and non-degenerate volatility when the mean field term is depicted via two marginal laws of the state and control. A similar approach was used by Graber \cite{Graber} in a linear-quadratic framework. Djete \cite{Djete1} focused on the limit theory for MFG of controls with common noise using their tailor-made measure-valued MFG solution. It is worth noting that some techniques in previous studies in the presence of marginal laws of the state and the control under the separated structure are not applicable to the general joint law of the state and the control in the present paper (see \eqref{strict_SDE} and \eqref{cost_func_strict}). In particular, how to utilize the compactification method to cope with the joint law dependence calls for new and careful investigations.  

Another new feature in the present work is to allow the state process to reflect along an exogenous stochastic boundary described by a continuous adapted process. Stochastic control and stochastic differential games with reflected diffusions have attracted increasing attention in recent studies. To name a few,  Borkar and Budhiraja \cite{Borkar05} examined the ergodic drift control of reflected diffusions; Han et al.~\cite{Han16} concerned the optimal pricing barrier control in a regulated financial market system modelled by a reflected diffusion process;  Ferrari \cite{Fer19} handled a class of singular control problems for a reflected diffusion process at boundary zero; Bo  et al. \cite{BoLiaoYu21, BHY23b, BHY25} studied the stochastic control problems of reflected diffusion with applications in optimal benchmark tracking portfolio using the capital injection; Bo et al. \cite{BHYEJP24} proposed a decomposition-homogenization technique for Robin boundary PDE problems arising from stochastic control of state process in nonnegative orthant; Pradhan \cite{Prad21} solved a risk-sensitive ergodic control problem for reflected diffusions in nonnegative orthant; Ghosh and Kumar \cite{Gh02} tackled a zero-sum stochastic differential game in nonnegative orthant. Recently, Bayraktar et al. \cite{Bayraktar2019} formulated a MFG problem with state reflections where only the drift is under control in the context of large symmetric queuing systems in heavy traffic with strategic servers. For the existence of MFE, the HJB equation with Neumann boundary conditions, together with some fixed point arguments, plays the key role in  \cite{Bayraktar2019}. Ricciardi \cite{Ricciardi2023} discussed the convergence of Nash equilibrium in $N$-player games towards the MFE in the MFG with drift control and state reflections using the approach of the master equation with Neumann boundary condition. For general MFG of controls with joint law dependence in the context with reflected state process, the existence of MFE and the limit theory remain interesting open problems.

\subsection{Our contributions}
	
The first theoretical contribution of the present paper is to prove, for the first time, the existence of a Markovian MFE for MFG of controls  using our customized compactification method (see \thmref{existence_RMFE} and \corref{existence_Markovian}). A well-known challenge in the compactification method arises from the relaxed formulation when the mean-field term is lifted to the joint distribution of the state $X$ and the relaxed control $\Lambda$. Mathematically, we need to cope with $\Pc_2(\R\times\Pc(U))$ instead of $\Pc_2(\R\times U)$ in compactification arguments. Another challenge is the discontinuity of the joint law measure flow of the state and the control. Indeed, recall that in the classical MFG problem, for any $\boldsymbol{\mu}=(\mu_t)_{t\in [0,T]}\in C([0,T];\Pc_2(\R))$, the state process is given by
\begin{align*}
\d X_t^{\boldsymbol{\mu}}=b(t,X_t^{\boldsymbol{\mu}},\mu_t,\alpha_t)\d t+\sigma(t,X_t^{\boldsymbol{\mu}},\mu_t,\alpha_t)\d W_t.
\end{align*}
Note that $\boldsymbol{\mu}^n\to\boldsymbol{\mu}$ in $C([0,T];{\cal P}_2(\R))$  as $n\to\infty$ is equivalent to $\sup_{t\in [0,T]}\mathcal{W}_{2,\R}(\mu_t^n,\mu_t)\to 0$ as $n\to\infty$, which indicates that, under some regular conditions on $b$ and $\sigma$, 
\begin{align*}
\lim_{n\to\infty}\sup_{t\in [0,T]}\E\left[\left|X_t^{\boldsymbol{\mu}^n}-X_t^{\boldsymbol{\mu}}\right|^2\right]=0.
\end{align*}
However, in the MFG of controls, the state process is governed by
\begin{align*}
\d X_t^{\boldsymbol{\rho}}=b(t,X_t^{\boldsymbol{\rho}},\rho_t,\alpha_t)\d t+\sigma(t,X_t^{\boldsymbol{\rho}},\rho_t,\alpha_t)\d W_t,
\end{align*}
where $\boldsymbol{\rho}$ is the joint law of $(X,\alpha)=(X_t,\alpha_t)_{t\in[0,T]}$ with $\rho_t$ being its $t$-marginal distribution. Due to the possible discontinuity of the control process $\alpha=(\alpha_t)_{t\in[0,T]}$, the convergence $\boldsymbol{\rho}^n\to\boldsymbol{\rho}$ in $\Pc_2(\C\times\mathcal{B})$ (see \equref{CB_space} for the definition of $\C\times\mathcal{B}$) is strictly weaker than $\sup_{t\in [0,T]}\mathcal{W}_{2,\R\times \R^l}(\rho_t^n,\rho_t)\to 0$ as $n\to\infty$, which in turn poses a challenge in showing 
\begin{align*}
\lim_{n\to\infty}\sup_{t\in [0,T]}\E\left[\left|X_t^{\boldsymbol{\rho}^n}-X_t^{\boldsymbol{\rho}}\right|^2\right]=0.
\end{align*}
This issue caused by joint law dependence also poses additional complexities in establishing the limit theory within the relaxed control framework. In response to this challenge, we adopt the extension transformation (see \equref{eq:extension}) introduced in \cite{BWWY} together with a separation condition (see (A1) in \assref{ass1} and \remref{ass1_remark}).  We consider the joint distribution of entire trajectories of the state and control processes and establish the fixed point result directly on the extended space. As a price to pay, our method brings new difficulties in compactification arguments using relaxed control formulation in the weak sense. Some additional efforts are needed to conclude the fixed point result on the extended space, and the proof of \thmref{existence_RMFE} differs substantially from the one in Lacker \cite{Lacker} for the classical MFG (see \lemref{closed} and \lemref{compact_set}).
	
The second theoretical contribution is our incorporation of state reflections in the MFG problem (see \eqref{strict_SDE} and \eqref{eq:reflectedterm}). By using a proper Skorokhod mapping, we reveal that the problem can be reformulated to allow us to apply the techniques for MFG of controls without state reflections. For ease the presentation, we focus on the simple case where the state process $X=(X_t)_{t\in[0,T]}$ reflects upon reaching a given random boundary $A=(A_t)_{t\in[0,T]}$ as a continuous stochastic process (see \equref{strict_SDE}). In particular, to address the additional reflection behavior and stochastic boundary, we introduce a flexible formulation that expands the canonical space by incorporating the canonical space of the process $A$. With the help of the Skorokhod mapping, reflections can be encoded in the additional dimension in the canonical process, leaving us to primarily deal with the joint law issues in the compactification method and the limit theory. Our approach in the present paper provides a convenient and novel way to handle general state reflections in the MFG framework, which may motivate some related future studies.		
	
As the third theoretical contribution, we introduce for the first time the notion of Nash equilibrium in the weak sense (\defref{NE_weak}) to study the limit theory for MFG of controls, in which the equilibrium is defined through distributions on the canonical space, without specifying the underlying probability space. The main challenge in this definition is to characterize, in terms of distributions, how each decision-maker optimally chooses her strategy when all other agents' strategies remain fixed (\defref{RsiN}). We then establish two key limit theorems. Firstly, we prove that any weak limit of the empirical measures induced by $\boldsymbol{\epsilon}$-Nash equilibria in the weak sense must be supported exclusively on the set of relaxed mean field equilibria (\thmref{propagation}). This is an analogue of the propagation of chaos result in the mean field control theory, akin to Theorem 2.11 in Lacker \cite{Lacker1}. In the proof of \thmref{propagation}, we propose the notion of the extended relaxed control (see \defref{extended_relaxed_control}), in which the probability measure space $\Pc_2(\C\times\mathcal{Q})$ is incorporated in the canonical space. We additionally establish the relationship between such extended relaxed control and the (classical) relaxed control introduced in \defref{relaxed_control} (see \lemref{relation_value}), which, to the best of our knowledge, is novel in the literature.  Secondly, we prove that  if a Markovian MFE exists, we can always construct a sequence of $\boldsymbol{\epsilon}$-Nash equilibria in the weak sense in $N$-player games that converges to this MFE (\thmref{convergence_theorem}), thereby fulfilling the converse direction of the limit theory. Our results establish a bidirectional relationship between the $N$-player Nash equilibria in the weak sense and their mean field counterpart in the framework of MFG of controls, which complements the existing limit theory in the strong sense by Djete \cite{Djete1} using the measure-valued MFE.  Furthermore, we also highlight the difference between our results and the limit theory in the uncontrolled diffusion framework discussed in Lacker \cite{Lakcer3} where the mean of $N$-player laws $\frac1N\sum_{i=1}^N \Pb_n\circ (\xi^i,B,W^i,\hat{\mu}[\Lambda^n],\Lambda^{n,i},X^i[\Lambda^n])^{-1}$
is considered. In contrast, we focus on the convergence of the corresponding empirical measure $\frac1N\sum_{i=1}^N\delta_{(X^i,\Lambda^i,W^i,A^i)}$ within the framework of controlled diffusion with reflection,  
which requires some distinct proofs. Compared with the mean of laws, considering the empirical mean offers several advantages: (i) it is naturally observable and implementable within the $N$-player framework, as each agent interacts through the realized average of the population rather than through its abstract distribution; (ii) it is computationally tractable, providing a practical foundation for numerical approximation via particle system simulations; (iii) it allows a more intuitive representation of the aggregate population behavior, facilitating the interpretation and derivation of the mean field limit equations.

The rest of the paper is organized as follows. Section \ref{sec:formulation} introduces the problem formulation of MFG of controls with a stochastic reflection boundary and the dynamic Skorokhod representation. In Section \ref{sec:formulation-MFG}, we first consider the extension transformation and introduce the resulting relaxed control formulation and definitions of relaxed MFE and strict MFE in the weak sense in the mean field model. We then introduce the problem formulation of the $N$-player game and the $\epsilon$-Nash equilibrium in the weak sense. Section \ref{sec:existence} presents the main results of this paper including (i) the existence of a relaxed MFE and a Markovian strict MFE in the weak sense; (ii) the limit theory as the bidirectional convergent results between the approximate $N$-player $\epsilon$-Nash equilibria in the weak sense and mean field equilibria in the weak sense. Section \ref{sec:proofs} collects technical proofs of the main theorems and auxiliary results in previous sections.

	\vspace{0.3in}
	\noindent{\bf Notations.}\quad We list below some notations that will be used frequently throughout the paper:
\vspace{-0.2in}
\begin{center}
\begin{longtable}{l l}
$|\cdot|$ & Euclidean norm on $\R^n$\\
$\delta_x$ & The Dirac measure at $x$.\\
$L^p((A,\mathscr{B}(A),\lambda_A);E)$ & Set of $L^p$-integrable $E$-valued mapping defined on $(A,\mathscr{B}(A))$\\ 
& we write $L^p(A;E)$ for short\\
${\tt m}$ & The Lebesgue measure over $[0,T]$.\\
$\Law^P(\eta)$ & Law of r.v. $\eta$ under probability measure $P$\\
$\Pc_p(E)$ & Set of probability measures on $E$ with finite $p$-order moments\\
$\mathcal{W}_{p,E}$ & The $p$-Wasserstein metric on $\Pc_p(E)$\\
$\E^{P}[\cdot]$ & Expectation operator under probability measure $P$\\
$\Mc(U)$ & Set of signed Radon measures on $U$\\
$\C=C([0,T];E)$ & Set of $E$-valued continuous functions on $[0,T]$\\
$\|\cdot\|_{\infty}$ & The supremum norm of $\C$ \\
$R(\boldsymbol{\xi})$ ($R^{\rm s}(\boldsymbol{\xi})$) & Set of admissible relaxed (strict) controls (see Def. \ref{relaxed_control})\\
$R_N$ ($R_N^{\rm s}$) & Set of $N$-player relaxed (strict) controls (see Def. \ref{relaxed_strategy})\\
$R_{\rm e}(\Theta)$ ($R_{\rm e}^{\rm s}(\boldsymbol{\xi})$) & Set of extended admissible relaxed (strict) controls (see Def. \ref{extended_relaxed_control})
\end{longtable}
\end{center}

\section{Problem Formulation of MFG with State Reflections}\label{sec:formulation}
	
We formulate the problem of MFG of controls with state reflections, where the mean-field interaction is encoded via the joint law of the state-control. 
    
Let $T>0$ be a finite horizon and $(\Omega,\F,\Fb,\Pb)$ be a filtered  probability space with the filtration $\Fb=(\F_t)_{t\in[0,T]}$ satisfying the usual conditions. For $l\in\mathbb{N}$ and $p>2$, let $W=(W_t)_{t\in[0,T]}$ be a standard scalar $\Pb$-Brownian motion and $U\subset\R^l$ be a compact subset as the policy space.  
Let $A=(A_t)_{t\in [0,T]}$ be a continuous and adapted process which will be viewed as our reflection boundary (or the benchmark process) satisfying $\E[\sup_{t\in[0,T]}|A_t|^p]<\infty$ and  $A_0\leq \eta$ with $\eta\in L^2((\Omega,\F_0,\Pb);\R)$. We also define 
\begin{align}\label{hatP}
    \hat{P}:=\Pb\circ (W,A)^{-1}
\end{align}
 by the joint law of $(W,A)$. 
    
Let $\Fb=(\F_t)_{t\in [0,T]}$ be the enlarged natural filtration generated by $(W,A)$ so that $\Fb$ satisfies the usual condition. Denote by $\U[0,T]$ the set of admissible controls which are $\Fb$-progressively measurable processes. We then set $\mathcal{B}:=
L^2([0,T];U)$ which is also equivalent to $L^{\infty}([0,T];U)$ due to the compactness of the policy space $U$. Moreover, we endow $\mathcal{B}$ with the $L^2$-metric defined by $d_{\mathcal{B}}(\beta^1,\beta^2):=(\int_0^T|\beta^1_t-\beta^2_t|^{2}\d t)^{\frac12}$ for $\beta^1,\beta^2\in\mathcal{B}$.
Then, we equip $\C\times\mathcal{B}$ with the product metric given by, for $(C^1,\beta^1),(C^2,\beta^2)\in\C\times\mathcal{B}$, 
\begin{align}\label{CB_space}
d_{\C\times\mathcal{B}}((C^1,\beta^1),(C^2,\beta^2)):=\left\|C^1-C^2\right\|_{\infty}+d_{\mathcal{B}}(\beta^1,\beta^2),\quad\forall (C^1,\beta^1).
\end{align}
Consequently, we have the following lemma whose proof is reported in subsection \ref{appendix}.
\begin{lemma}\label{CB_property} 
For $\boldsymbol{\rho}^1,\boldsymbol{\rho}^2\in\Pc_2(\C\times\mathcal{B})$ with $\rho_t^1,\rho_t^2$ being their $t$-marginal distributions, there exists a constant $C>0$ only depending on $T$ such that\begin{align}
			\int_0^T\mathcal{W}_{2,\R\times U}(\rho_t^1,\rho_t^2)^2\d t\leq C\mathcal{W}_{2,\C\times\mathcal{B}}(\boldsymbol{\rho}^1,\boldsymbol{\rho}^2)^2.
		\end{align}
	\end{lemma}
	
Let $(b,\sigma):[0,T]\times\R\times\Pc_2(\R\times U)\times U\mapsto\R^2$ be measurable mappings. For any $\boldsymbol{\rho}\in\Pc_2(\C\times\mathcal{B})$ with the induced measure flow $(\rho_t)_{t\in [0,T]}$ such that $\rho_t$ stands for the marginal distribution at time $t$ and an admissible control $\alpha=(\alpha_t)_{t\in [0,T]}\in\U[0,T]$, we consider the following controlled state process with reflection given by
\begin{align}\label{strict_SDE}
\d X_t^{\alpha,\boldsymbol{\rho}}=b(t,X_t^{\alpha,\boldsymbol{\rho}},\rho_t,\alpha_t)\d t+\sigma(t,X_t^{\alpha,\boldsymbol{\rho}},\rho_t,\alpha_t)\d W_t+\d R^{A}_t,\quad X_0^{\alpha,\rho}=\eta,
\end{align}
where $R^A=(R_t^A)_{t\in [0,T]}$ is an $\Fb$-adapted non-decreasing continuous process with $R_0^A=0$ such that $\Pb$-$\as$, $X_t^{\alpha,\boldsymbol{\rho}}\geq A_t$ for all $t\in [0,T]$ and $\mathbb{P}$-a.s. 
\begin{align}\label{eq:reflectedterm}
\int_0^T \mathbf{1}_{\{X_t^{\alpha,\boldsymbol{\rho}}>A_t\}}\d R_t^{A}=0.  
\end{align}
	
Let $\mu_T$ be the first marginal distribution of $\rho_T$. We consider the following cost functional that 
\begin{align}\label{cost_func_strict}
J(\alpha,\boldsymbol{\rho}):=\E\left[\int_0^Tf(t,X_t^{\alpha,\boldsymbol{\rho}},\rho_t,\alpha_t)\d t+\int_0^T c(t,X_t^{\alpha,\boldsymbol{\rho}})\d R_t^A+g(X_T,\mu_T)\right],
	\end{align}
where $f:[0,T]\times\R\times\Pc_2(\R\times U)\times U\mapsto\R$ is the running cost function, $c:[0,T]\times \R\mapsto\R$ is the cost function associated to the reflection, and $g:\R\times\Pc_2(\R)\mapsto\R$ is the terminal cost function.
	
The strict mean field equilibrium (MFE) in the strong sense is defined as below.
\begin{definition}[Strict MFE in the strong sense]\label{strong_MFE}
A couple $(\boldsymbol{\rho}^*,\alpha^*)\in\mathcal{P}_2(\C\times\mathcal{B})\times\U[0,T]$ is said to be a strict MFE in the strong sense ((S)-MFE) for the MFG of controls with state reflection in \equref{strict_SDE}-\equref{cost_func_strict} if $J(\alpha^*,\boldsymbol{\rho}^*)\leq J(\alpha,\boldsymbol{\rho}^*)$ for all $\alpha\in\U[0,T]$, and the consistency condition $\boldsymbol{\rho}^*=\Law^{\Pb}(X^{\alpha^*,\boldsymbol{\rho}^*},\alpha^*)$ holds.
\end{definition}

We impose the following assumptions on model coefficients throughout the paper.
\begin{ass}\label{ass1} 
\begin{itemize}
\item[{\rm(A1)}] {\rm(Separation condition)} The coefficient $(b,\sigma^2,f):[0,T]\times\R\times\Pc_2(\R\times U)\times U\to \R^3$ has the following decomposition, for $(t,x,\rho,u)\in [0,T]\times\R\times\Pc_2(\R \times U)\times U$,
\begin{align*}
(b,\sigma^2,f)(t,x,\rho,u)&=(b_1,\sigma_1^2,f_1)(t,x,\mu,u)+(k_1b_2,k_2\sigma_2^2,k_3f_2)(t,x,\rho,u),\\
(b_2,\sigma_2^2,f_2)(t,x,\rho,u)&=\int_{\R\times U}(b_3,\sigma_3^2,f_3)(t,x,\mu,u,x',u')\rho(\d x',\d u')
\end{align*}
with some $k_1,k_2,k_3\in\R$ and $\mu\in\Pc_2(\R)$ being the first marginal distribution of $\rho$. Here, $(b_1,\sigma_1,f_1):[0,T]\times\R\times\Pc_2(\R)\times U\to\R^3$ is Borel measurable and jointly continuous in $(x,\mu,u)\in\R\times\Pc_2(\R)\times U$,  $(b_3,\sigma_3,f_3):[0,T]\times\R\times\Pc_2(\R)\times U\times\R\times U\to\R^3$ is Borel measurable and $(b_2,\sigma_2,f_2):[0,T]\times\R\times\Pc_2(\R\times U)\times U\to \R^3$ is jointly continuous in $(x,\rho,u)\in \R\times\Pc_2(\R\times U)\times U$.
Furthermore, $(b_1,\sigma_1,f_1)$ and $(b_3,\sigma_3,f_3)$ are uniformly continuous in $\mu\in\Pc_2(\R)$ with respect to $(t,x,u,x',u')\in [0,T]\times\R\times U\times\R\times U$.
			
\item[{\rm(A2)}] The coefficients $b(t,x,\rho,u)$ and $\sigma(t,x,\rho,u)$ are uniformly Lipschitz continuous in $x\in\R$ in the sense that, there is a constant  $M>0$ independent of $(t,\rho,u)\in [0,T]\times\Pc_2(\R\times U)\times U$ such that, for all $x,x'\in\R$,
\begin{align*}
\left|b(t,x',\rho,u)-b(t,x,\rho,u)\right|+\left|\sigma(t,x',\rho,u)-\sigma(t,x,\rho,u)\right|\leq M|x-x'|.
\end{align*}
\item[\rm(A3)] There exists a constant $M>0$ such that, for all $(t,x,\rho,u)\in[0,T]\times\R\times\Pc_2(\R\times U)\times U$, 
\begin{align*}
|b(t,x,\rho,u)|+|\sigma(t,x,\rho,u)|\leq M(1+|x|+M_2(\rho)+|u|)
\end{align*}
with $M_2(\rho):=(\int_{\R\times U}(|x|^2+|u|^2)\rho(\d x,\d u))^{\frac12}$ for $\rho\in\Pc_2(\R\times U)$.
\item[\rm(A4)] The mappings $f$ and $g$ are jointly continuous and there exists a constant $M>0$ such that, for all $(t,x,\mu,\rho,u)\in [0,T]\times\R\times\Pc_2(\R)\times\Pc_2(\R\times U)\times U$, 
\begin{align*}
|f(t,x,\rho,u)|+|g(x,\mu)|\leq M\left(1+|x|^2+M_2(\mu)^2+M_2(\rho)^2\right)
\end{align*}
with $M_2(\mu):=(\int_{\R}|x|^2\mu(\d x))^{\frac12}$. 
\item[\rm(A5)] There is a constant $M>0$ independent of $(t,x,\mu,u)\in [0,T]\times\R\times\Pc_2(\R)\times U$ such that, for all $(x',u'),(y',v')\in\R\times U$,
\begin{align*}
|(b_3,\sigma_3,f_3)(t,x,\mu,u,x',u')-(b_3,\sigma_3,f_3)(t,x,\mu,u,y',v')|\leq M(|x'-y'|+|u'-v'|).   
\end{align*}
\item [{\rm (A6)}] The mapping $c:[0,T]\times \R\to\R$ is uniformly continuous and has at most linear growth in $x\in \R$, i.e., there is a constant $M$ independent of $(t,x)\in [0,T]\times \R$ such that $|c(t,x)|\leq M(1+|x|)$.
\end{itemize}
\end{ass}
	
We have the following discussions on  \assref{ass1}:   
\begin{remark}\label{ass1_remark}
The condition~{\rm (A1)} in \assref{ass1} is mandated to overcome some technical issues arising from the joint law of the state and the control. In fact, the mean field term $\rho_t\in\Pc_2(\R\times U)$ in the MFG of controls differs significantly from the state measure flow $\boldsymbol{\mu}=(\mu_t)_{t\in [0,T]}$ in classical MFG problems in the sense that it can be taken for granted that $\boldsymbol{\mu}=(\mu_t)_{t\in [0,T]}$ is continuous thanks to the continuity of the state process $X=(X_t)_{t\in [0,T]}$. In contrast, the measure flow $\boldsymbol{\rho}=(\rho_t)_{t\in [0,T]}$ is not necessarily  continuous as the control process $\alpha=(\alpha_t)_{t\in [0,T]}$ may not be continuous. The condition {\rm (A1)}, together with the condition {\rm (A5)}, ensures that the state process $X$ depends continuously on the mean field term (see \lemref{closed}). The separation conditions {\rm (A1)} and {\rm (A5)} can alternatively be replaced by \assref{ass3} (see \remref{existence_Lip} and \remref{closed_Lip}). However, we adhere to the conventional assumptions as in the classical MFG framework and do not impose the Lipschitz condition on the measure $\rho\in\Pc_2(\R\times U)$ here. We stress that the separation conditions {\rm (A1)} and {\rm (A5)} and the joint Lipschitz condition in \assref{ass3} play the same role in proving our main result \thmref{existence_RMFE}. 
\end{remark}
        
Under \assref{ass1}, SDE~\equref{strict_SDE} has a unique  strong solution $(X^{\alpha,\boldsymbol{\rho}},R^A)=(X_t^{\alpha,\boldsymbol{\rho}},R_t^A)_{t\in[0,T]}$, and hence the cost functional $J(\alpha,\boldsymbol{\rho})$ in \equref{cost_func_strict} is well defined. We next introduce the so-called dynamic Skorokhod problem.
\begin{definition}[Dynamic Skorokhod problem]\label{DSP}
Let $f,a\in\C$ satisfy $f(0)\geq a(0)$. A pair of functions $(g,\ell)\in \C\times\C$ is called a solution to the dynamic Skorokhod problem for $(f,a)$ {\rm (DSP($f,a$) for short)} if {\rm(i)} $g(t)=f(t)+\ell(t)$, $\forall t\in [0,T]$; {\rm(ii)} $\ell(0)=0$ and $t\mapsto\ell(t)$ is continuous and non-decreasing; {\rm(iii)} $g(t)\geq a(t)$ for all $t\in [0,T]$; {\rm(iv)} $\int_0^T\mathbf{1}_{\{g(t)>a(t)\}}\d\ell(t)=0$.
\end{definition}
Moreover, the next result follows from PiliPenko \cite{PiliPenko}.
\begin{lemma}\label{Gamma}
Let $(f,a)\in \C\times\C$ with $f(0)\geq a(0)$. Then, there is a unique solution $(g,\ell)$ to {\rm DSP($f,a$)}. Furthermore, it holds that, for all $t\in[0,T]$,
\begin{equation}\label{representation}
g(t)=f(t)+\sup_{s\in[0,t]}\{-(f(s)-a(s))\vee 0\},~~\ell(t)=\sup_{s\in[0,t]}\{-(f(s)-a(s))\vee 0\}.
\end{equation}
\end{lemma}
	
Let us introduce the set of functions $\mathcal{D}:=\{(a,f)\in \C\times\C;~f(0)\geq a(0)\}$ and define the dynamic Skorokhod mapping $\Gamma:\mathcal{D}\to\C$ by
\begin{align}\label{eq:Skorohodmap}
\Gamma(a,f)=g,
\end{align}
where $(g,\ell)$ is a solution to DSP($f,a$). Lemma~\ref{Gamma} yields that, the mapping $\Gamma:\mathcal{D}\to\C$ is Lipschitz continuous in the sense that, for any $(a_i,f_i)\in\mathcal{D}$ with $i=1,2$,
\begin{align}\label{Lipschitz_Gamma}
\left\|\Gamma(a_1,f_1)-\Gamma(a_2,f_2)\right\|_{\infty}\leq 2\left\|f_1-f_2\right\|_{\infty}+\left\| a_1-a_2\right\|_{\infty}.
\end{align}
Let us define the $\Fb$-adapted process $Y^{\alpha,\boldsymbol{\rho}}=(Y_t^{\alpha,\boldsymbol{\rho}})_{t\in [0,T]}$ by 
\begin{align}\label{strict_Y}
Y_t^{\alpha,\boldsymbol{\rho}}:=\eta+\int_0^tb(s,X_s^{\alpha,\boldsymbol{\rho}},\rho_s,\alpha_s)\d s+\int_0^t\sigma(s,X_s^{\alpha,\boldsymbol{\rho}},\rho_s,\alpha_s)\d W_s,
\end{align}
where $(X^{\alpha,\boldsymbol{\rho}},R^A)$ is the unique  strong solution to SDE~\equref{strict_SDE}. Then, one can easily verify that the pair $(X^{\alpha,\boldsymbol{\rho}},R^A)$ is the unique solution to DSP($Y^{\alpha,\boldsymbol{\rho}},A)$ in the sense of \defref{DSP}.
	
\section{Relaxed Control Formulation}\label{sec:formulation-MFG}
	
\subsection{Mean field model with extension transformation }\label{MFE}
	
In this subsection, we first consider the relaxed control formulation for MFG of controls under the reflected state dynamics. As a preparation, let us first introduce some basic spaces:
\begin{itemize}
\item 	The spaces $(\C^Y,\C^W,\tilde\Omega):=(\C,\C,\C)$ are endowed with the same norm $\|\cdot\|_{\infty}$\ and the corresponding Borel $\sigma$-algebras denoted by $\F^Y$, $\F^W$ and $\tilde\F$.  Let $\F_t^Y$ $\F_t^W$ and $\tilde\F_t$ be the respective Borel $\sigma$-algebras up to time $t\in[0,T]$. 
		
\item 	The space $\mathcal{Q}$ of relaxed controls is defined as the set of measures $q$ in $[0,T]\times U$ with the first marginal equal to the Lebesgue measure and $\int_{[0,T]\times U}|u|^pq(\d t,\d u)<\infty$. We endow the space $\mathcal{Q}$ with the $2$-Wasserstein metric on $\Pc_2([0,T]\times U)$ given by $d_{\mathcal{Q}}(q^1,q^2)=\mathcal{W}_{2,[0,T]\times U}\left(\frac{q^1}{T},\frac{q^2}{T}\right)$, where the metric on $[0,T]\times U$ is given by $((t_1,u_1),(t_2,u_2))\to|t_2-t_1|+|u_2-u_1|$. Note that, each $q\in\mathcal{Q}$ can be identified with a measurable function $[0,T]\in t\to q_t\in\Pc_2(U)$, defined uniquely up to $\as$ by $q(\d t,\d u)=q_t(\d u)\d t$. In the sequel, we will always refer to the measurable mapping $q=(q_t)_{t\in [0,T]}$ to a relaxed control in $\mathcal{Q}$. Let $\F^{\mathcal{Q}}$ be the Borel $\sigma$-algebra of $\mathcal{Q}$ and $\F_t^{\mathcal{Q}}$ be the $\sigma$-algebra generated by the maps $q\mapsto q([0,s]\times V)$ with $s\in [0,t]$ and Borel measurable $V\subset U$. Because $U$ is compact and Polish, $\mathcal{Q}$ as a closed subset of $\Pc_2([0,T]\times U)$ is also compact and Polish.
\end{itemize}
Define the canonical space $\Omega:=\C^Y\times\mathcal{Q}\times\C^W\times\tilde{\Omega}$, and equip it with the product $\sigma$-algebra $\F=\F^{Y}\otimes\F^{\mathcal{Q}}\otimes\F^W\otimes\tilde{\F}$ and the product filtration $\F_t=\F_t^Y\otimes\F_t^{\mathcal{Q}}\otimes\F_t^W\otimes\tilde{\F}_t$ for $t\in[0,T]$. Note that $\Omega$ endowed with the metric  $d_{\Omega}(\omega^1,\omega^2)=\|y^1-y^2\|_{\infty}+d_{\mathcal{Q}}(q^1,q^2)+\|w^1-w^2\|_{\infty}+\|\tilde\omega^1-\tilde\omega^2\|_{\infty}$ for $\omega^i=(y^i,q^i,w^i,\tilde\omega^i)\in\Omega$ with $i=1,2$ is a Polish space and so is $\Pc_2(\Omega)$ endowed with the $2$-Wasserstein metric $\mathcal{W}_{2,\Omega}$.
	We identify the coordinate mapping as $(Y,\Lambda,W,A)=(Y_t,\Lambda_t,W_t,A_t)_{t\in[0,T]}$, i.e., for $\omega=(y,q,w,\tilde\omega)\in\Omega$. The coordinate processes are defined by $Y_t(\omega)=y_t$, $\Lambda_t(\omega)=q_t$, $W_t(\omega)=w_t$ and  $A_t(\omega)=\tilde{\omega}_t$. 
	For simplicity, denote by $\F_t^Y$, $\F_t^{\mathcal{Q}}$, $\F_t^W$ and $\tilde\F_t$ the natural extensions of these filtrations to $\Omega$. In the sequel, when talking about the filtrations $\F_t^Y$, $\F_t^{\mathcal{Q}}$ and $\tilde\F_t$, there should be no confusion of which space the filtrations are defined on.
	\begin{remark}
		We note the difference between the processes $(W,A)$ in the strong sense and the coordinate processes $(W,A)$ in the weak sense defined on $(\Omega,\F, P)$. By abuse of notation, we shall use the same notation to ease the presentation. 
	\end{remark}
	
Next, we introduce the \textit{extension transformation} in \cite{BWWY} to cope with the joint law of the state and the relaxed control. More precisely, for $h:\Pc_2(\R\times U)\mapsto\R$, the extension $\tilde h:\Pc_{2}(\R\times\Pc(U))\to\R$ of the mapping $h$ is defined by
\begin{align}\label{eq:extension}
\tilde h(\xi):=h(\mathscr{P}(\xi)),\quad\forall \xi\in\Pc_2(\R\times\Pc(U))    
\end{align}
with $\mathscr{P}(\xi)(\d x,\d u):=\int_{\mathcal{M}(U)}q(\d u)\xi(\d x,\d q)$ for $\xi\in\Pc_2(\R\times\Pc(U))$. For the mapping $\mathscr{P}:\Pc_2(\R\times\Pc(U))\mapsto\Pc_2(\R\times U)$, we have the following continuous property whose proof is postponed to subsection~\ref{appendix}.
	\begin{lemma}\label{Lipschitz_P}
		It holds that $\mathcal{W}_{2,\R\times U}(\mathscr{P}(\xi_1),\mathscr{P}(\xi_2))\leq\mathcal{W}_{2,\R\times\Pc(U)}(\xi_1,\xi_2)$ for all $\xi_1,\xi_2\in\Pc_2(\R\times\Pc(U))$. 
		Here, the space $\R\times\Pc(U)$ is endowed with the product metric $d_{\R\times\Pc(U)}((x_1,q_1),(x_2,q_2))=|x_1-x_2|+\mathcal{W}_{2,U}(q_1,q_2)$ for $(x_i,q_i)\in\R\times{\cal P}(U)$ with $i=1,2$.
	\end{lemma} 
	
	\begin{remark}
		From the compactness of the policy space $U$, it follows that $\Pc(U)=\Pc_2(U)$. As a result, we can adopt the $2$-Wasserstein metric for $\Pc(U)$ in \lemref{Lipschitz_P}.
	\end{remark}
	
We then consider the extension transformation \eqref{eq:extension} to the coefficients $b,\sigma,f$,  and denote by $\tb$, $\ts$, $\tf$ their corresponding extensions. The next result follows directly from \cite{BWWY}.
	\begin{lemma}\label{extension_ass}
The extensions $\tb,\ts,\tf$ satisfy the following properties:
\begin{itemize}
\item[{\rm(B1)}] The mapping $(\tb,\ts^2,\tf):[0,T]\times\R\times\Pc_2(\R\times \Pc(U))\times U\to \R^3$ has the  decomposition:
\begin{align*}
(\tb,\ts^2,\tf)(t,x,\xi,u)&=(b_1,\sigma_1^2,f_1)(t,x,\mu,u)+(k_1\tb_2,k_2\ts_2^2,k_3\tf_2)(t,x,\xi,u),
\end{align*}
where the extension $(\tb_2,\ts_2^2,\tf_2):[0,T]\times\R\times\Pc_2(\R\times \Pc(U))\times U\to\R^3$ admits the representation given by
\begin{align*}
(\tb_2,\ts_2^2,\tf_2)(t,x,\xi,u)=\int_{\R\times \mathcal{M}(U)}\int_U(b_3,\sigma_3^2,f_3)(t,x,\mu,u,x',u')q'(\d u')\xi(\d x',\d q').    
\end{align*}
Here, by convention, $\mu$ denotes the first marginal law of $\xi$. In particular, $\tb,\ts,\tf$ are Borel measurable and jointly continuous in $(x,\xi,u)\in \R\times\Pc_2(\R\times\Pc(U))\times U$.
\item[{\rm(B2)}] The mappings $\tb(t,x,\xi,u)$ and $\ts(t,x,\xi,u)$ are uniformly Lipschitz continuous in $x\in\R$ in the sense that, there is a constant $M>0$ independent of $(t,\xi,u)\in [0,T]\times\Pc_2(\R\times\Pc(U))\times U$ such that, for all $x,x'\in\R$,
\begin{align*}
\left|\tb(t,x',\xi,u)-\tb(t,x,\xi,u)\right|+\left|\ts(t,x',\xi,u)-\ts(t,x,\xi,u)\right|\leq M|x-x'|.
\end{align*}
\item[\rm(B3)] There is a constant $M>0$ such that, for all $(t,x,\xi,u)\in[0,T]\times\R\times\Pc_2(\R\times\Pc(U))\times  U$, 
\begin{align*}
\left|\tb(t,x,\xi,u)|+|\ts(t,x,\xi,u)\right|\leq M\left(1+|x|+M_2(\xi)+|u|\right)
\end{align*}
with $M_2(\xi):=(\int_{\R\times\Pc(U)}\left(|x|+\mathcal{W}_{2,U}(q,\delta_{u_0})\right)^2\xi(\d x,\d q))^{\frac12}$ for $\xi\in\Pc_2(\R\times {\cal P}(U))$ and some $u_0\in U$.
\item[\rm(B4)] There is $K>0$ such that $|\tf(t,x,\xi,u)|\leq K(1+|x|^2+M_2(\xi)^2)$ for all $(t,x,\xi,u)\in [0,T]\times\R\times\Pc_2(\R\times\Pc(U))\times U$.
\end{itemize}
\end{lemma}
	
Next, we give the definition of admissible relaxed control rules for the MFG problem. 
\begin{definition}[Relaxed control]\label{relaxed_control}
Let	$\boldsymbol{\xi}\in\Pc_2(\C\times\mathcal{Q})$ and define the measure flow $(\xi_t)_{t\in [0,T]}$ by setting $\xi_t$ as the marginal distribution at time $t$. We call a probability measure $P\in\Pc_2(\Omega)$ on $(\Omega,\F)$ an  admissible relaxed control rule
(denoted by $P\in R(\boldsymbol{\xi})$) if it holds that {\rm(i)}  $P\circ Y_0^{-1}=\Law^{\mathbb{P}}(\eta)$; {\rm(ii)} $P(Y_0\geq A_0)=1$; {\rm(iii)} the restriction of $P$ to $\C^W\times\tilde\Omega$,  $P|_{\C^W\times\tilde\Omega}$, agrees with $\Pb\circ (W,A)^{-1}$, $\ie$ $P\circ (W,A)^{-1}=\Pb\circ (W,A)^{-1}=\hat P$ (recall \equref{hatP}); {\rm(iv)} for all $\phi\in C^2_b(\R\times\R)$, the process
\begin{align*}
{\tt M}^{\boldsymbol{\xi}}\phi(t):=\phi(Y_t,W_t)-\int_0^t\int_U\bar{\mathbb{L}}\phi(s,X_s,Y_s,W_s,\xi_s,u)\Lambda_s(\d u)\d s,\quad t\in[0,T]
\end{align*}
is a $(P,\Fb)$-martingale.  Here, the infinitesimal generator acting on $\phi\in C_b^2(\R\times\R)$ is defined by, for $(t,x,\xi,u)\in[0,T]\times\R\times\Pc_2(\R\times\Pc(U))\times U$, 
\begin{align*}
\bar{\mathbb{L}}\phi(t,x,y,w,\xi,u):=\bar b(t,x,\xi,u)^{\T}\nabla\phi(y,w)+\frac12\tr\left(\bar\sigma\bar\sigma^{\T}(t,x,\xi,u)\nabla^2\phi(y,w)\right),
\end{align*}
where $\nabla$ stands for the full gradient with respect to $(y,w)$, the coefficients are defined by
\begin{align*}
            \bar b(t,x,\xi,u)=\begin{pmatrix}
                \tb(t,x,\xi,u)\\
                0
            \end{pmatrix},\qquad \bar\sigma(t,x,\xi,u)=\begin{pmatrix}
                \ts(t,x,\xi,u)\\
                1
            \end{pmatrix},
        \end{align*}
and $X(\omega)=\Gamma(A(\omega),Y(\omega))$  with the mapping $\Gamma:\mathcal{D}\mapsto\C$ defined in \eqref{eq:Skorohodmap}. Furthermore, if there exists an $\mathbb{F}$-progressively measurable $U$-valued process $\alpha=(\alpha_t)_{t\in [0,T]}$ on $\Omega$ such that $P(\Lambda_t(\d u)\d t=\delta_{\alpha_t}(\d u)\d t)=1$, we say that $P$ corresponds to a strict control $\alpha$ or we call it a strict control rule. The set of all strict control rules is denoted by $R^{\rm s}(\boldsymbol{\xi})$.
	\end{definition}
	
		In contrast to strict controls in the strong sense in \defref{strong_MFE}, the relaxed control is defined only through the probability measure on the canonical space without specifying the underlying filtered probability space, which is tailor-made for the weak formulation. From this point onwards, we will focus on this weak formulation in the rest of the paper. Moreover, for strict control and relaxed control, we will not distinguish the underlying probability space and we will only focus on the difference between $U$-valued process and $\Pc(U)$-valued process.
        
We have the following equivalent characterization of the set $R(\boldsymbol{\xi})$ in the weak formulation, whose proof is classical and we omit it.
\begin{lemma}\label{moment_p}
Let	$\boldsymbol{\xi}\in\Pc_2(\C\times\mathcal{Q})$ and $(\xi_t)_{t\in [0,T]}$ be the induced measure flow in \defref{relaxed_control}. Then,	$P\in\mathcal{R}(\boldsymbol{\xi})$ iff there exists a filtered probability space $(\Omega',\F',\Fb'=(\F'_t)_{t\in[0,T]},P')$ supporting a $\Pc(U)$-valued $\Fb'$-progressively measurable process $\Lambda=(\Lambda_t)_{t\in [0,T]}$, a scalar $\Fb'$-adapted process $Y^{\Lambda,\boldsymbol{\xi}}=(Y_t^{\Lambda,\boldsymbol{\xi}})_{t\in[0,T]}$, a standard scalar $\Fb'$-Brownian motion $W=(W_t)_{t\in [0,T]}$, an $\Fb'$-martingale measure ${\cal M}$ on $U\times[0,T]$  with intensity $\Lambda_t(\d u)\d t$ and a real-valued $\mathbb{F}'$-adapted continuous process $A=(A_t)_{t\in[0,T]}$,  and it holds that {\rm(i)} $P'\circ(Y_0^{\Lambda,\boldsymbol{\xi}})^{-1}=\Law^{\mathbb{P}}(\eta)$; {\rm(ii)} $P'(Y_0^{\Lambda,\boldsymbol{\xi}}\geq A_0)=1$; 
{\rm(iii)} $W_t=\int_0^t\int_U\mathcal{M}(\d u,\d t),~\forall t\in [0,T]$ $P'$-$\as$.
{\rm(iv)} the dynamics of state process obeys that, $P'$-a.s.,
\begin{align*}
\d Y_t^{\Lambda,\boldsymbol{\xi}}=\int_U\tb(t,X_t^{\Lambda,\boldsymbol{\xi}},\xi_t,u)\Lambda_t(\d u)\d t+\int_U\ts(t,X_t^{\Lambda,\boldsymbol{\xi}},\xi_t,u){\cal M}(\d u,\d t),~ X_t^{\Lambda,\boldsymbol{\xi}}=\Gamma(A,Y^{\Lambda,\boldsymbol{\xi}})_t.
\end{align*}         
There exists a constant $C>0$ depending on $M,\Law^{\mathbb{P}}(\eta),T$ and $\boldsymbol{}{\xi}$ such that 
\begin{align}\label{eq:momentp00}
\E^{P'}\left[\sup_{t\in[0,T]}\left|Y_t^{\Lambda,\boldsymbol{\xi}}\right|^p\right]\leq C,
\end{align}
with $M$ being stated in \assref{ass1}. Moreover, if $P\circ (X,\Lambda)^{-1}=\boldsymbol{\xi}$, we may find a constant $C_1>0$ depending only on $(M,\Law^{\Pb}(\eta)),T)$ such that
\begin{align}\label{eq:C1Lemma36}
\int_{\C\times\mathcal{Q}}\|\boldsymbol{x}\|^p_{\infty}~\boldsymbol{\xi}(\d\boldsymbol{x},\d q)=\E^{P'}\left[\sup_{t\in[0,T]}\left|X_t^{\Lambda,\boldsymbol{\xi}}\right|^p\right]\leq C_1.
\end{align} 
\end{lemma}

For an arbitrary  probability measure $\boldsymbol{\xi}\in\Pc_2(\C\times\mathcal{Q})$, let us define the cost functional on $\Pc_2(\Omega)$ by
\begin{align}\label{cost_func_relaxed}
\mathcal{J}(P;\boldsymbol{\xi}):=\E^P[\Delta(\boldsymbol{\xi})],\quad \forall P\in \Pc_2(\Omega),
\end{align}
where $\Delta(\boldsymbol{\xi})$ is defined by
\begin{align}\label{costDel}
\Delta(\boldsymbol{\xi}):=\int_0^T\int_U\tf(t,X_t,\xi_t,u)\Lambda_t(\d u)\d t+\int_0^Tc(t,X_t)\d R^A_t+g(X_T,\mu_T).
\end{align}
Here, $\tilde{f}$ is the corresponding extension of the running cost function $f$, $\mu_T$ is the first marginal distribution of $\xi_T$ and $R^A:=X-Y$ with $X$ described in \defref{relaxed_control}. It follows from \assref{ass1} that $v^*(\boldsymbol{\xi}):=\inf_{P\in R(\boldsymbol{\xi})}\mathcal{J}(P;\boldsymbol{\xi})<\infty$. Thus, we can define the set of optimal control rules by
\begin{align}\label{optimal_rule}
R_{\rm opt}(\boldsymbol{\xi}):=\{P\in R(\boldsymbol{\xi});~\mathcal{J}(P;\boldsymbol{\xi})=v^*(\boldsymbol{\xi})\}.
\end{align}
We next give the definition of relaxed MFE based on the relaxed control rules.
\begin{definition}[Relaxed MFE]\label{relaxed_MFE}
The couple $(\boldsymbol{\xi}^*,P^*)\in{\cal P}_2(\C\times\mathcal{Q})\times {\cal P}_2(\Omega)$ is said to be a relaxed MFE {\rm((R)-MFE)} for the MFG of controls with state reflection in \equref{cost_func_relaxed} and \defref{relaxed_control} if $\mathcal{J}(P^*;\boldsymbol{\xi}^*)\leq \mathcal{J}(P;\boldsymbol{\xi}^*)$ for all $P\in R(\boldsymbol{\xi}^*)$ and the consistency condition $\boldsymbol{\xi}^*=P^*\circ (X,\Lambda)^{-1}$ holds.
\end{definition}
Then, the component $\boldsymbol{\xi}^*$ of the (R)-MFE in Definition~\ref{relaxed_MFE} is in fact the fixed point of the set-valued mapping:
\begin{align}\label{eq:Kstar}
\textsf{R}^*: \Pc_2(\C\times\mathcal{Q})&\to2^{\Pc_2(\C\times\mathcal{Q})},\quad\boldsymbol{\xi}\mapsto R_{\rm opt}(\boldsymbol{\xi})\circ (X,\Lambda)^{-1},
\end{align}
where $2^{\Pc_2(\C\times\mathcal{Q})}$ denotes the power set of $\Pc_2(\C\times\mathcal{Q})$. Hence, to find an (R)-MFE, it suffices to find a fixed point of the set-valued mapping $\textsf{R}^*$.
	
As a result of the weak formulation that we adopt in the present paper, the previous definition of (S)-MFE in the strong sense in \defref{strong_MFE} is not suitable to work with. Instead, we need the following definition of the (S)-MFE in the weak sense and show its existence in the current framework. Let us first introduce a continuous transformation mapping that
	\begin{align}\label{Psi}
		\Psi:\C\times\mathcal{B}\to \C\times\mathcal{Q},\quad\C\times\mathcal{B}\ni (X,\alpha)\mapsto (X,\delta_{\alpha_t}(\d u)\d t)\in\C\times\mathcal{Q}.
	\end{align}
	The corresponding push forward measure transformation is denoted by
	\begin{align}\label{eq:transscrR}
		\mathscr{R}:\Pc_2(\C\times\mathcal{B})\to \Pc_2(\C\times\mathcal{Q}),\quad\Pc_2(\C\times\mathcal{B})\ni\boldsymbol{\rho}\to \boldsymbol{\rho}\circ\Psi^{-1}.   
	\end{align}
The next lemma ensures the continuity of the transformation $\mathscr{R}$ defined in \eqref{eq:transscrR}, whose proof is given in subsection \ref{appendix}.
\begin{lemma}\label{extension_continuous}
For any $\boldsymbol{\rho}^1,\boldsymbol{\rho}^2\in\Pc_2(\C\times\mathcal{B})$, it holds that $\mathcal{W}_{2,\C\times\mathcal{Q}}(\mathscr{R}(\boldsymbol{\rho}^1),\mathscr{R}(\boldsymbol{\rho}^2))\leq\mathcal{W}_{2,\C\times\mathcal{B}}(\boldsymbol{\rho}^1,\boldsymbol{\rho}^2)$.
\end{lemma}
	
Let us now define the (S)-MFE in the weak sense.
\begin{definition}[Strict MFE in the weak sense]\label{strict_MFE}
The couple $(\boldsymbol{\rho}^*,P^*)\in\Pc_2(\C\times\mathcal{B})\times\Pc_2(\Omega)$ is said to be a strict MFE in the weak sense {\rm((SW)-MFE)} for the MFG of controls with state reflection in \equref{cost_func_relaxed} and \defref{relaxed_control}  if {\rm(i)} $P^*$ is a strict control rule, i.e., $P^*\in R^{\rm s}(\boldsymbol{\xi}^*)$ with $\boldsymbol{\xi}^*=\mathscr{R}(\boldsymbol{\rho^*})$; {\rm(ii)} $P^*$ is optimal among strict control rules, i.e., $\mathcal{J}(P^*;\boldsymbol{\xi}^*)\leq \mathcal{J}(P;\boldsymbol{\xi}^*)$ for all $P\in R^{\rm s}(\boldsymbol{\xi}^*)$; {\rm(iii)} the consistency condition holds, i.e. $\boldsymbol{\rho}^*=P^*\circ (X,\alpha)^{-1}$ with $\alpha=(\alpha_t)_{t\in [0,T]}$ being the corresponding strict control in \defref{relaxed_control}.
\end{definition} 
	
\begin{definition}[Markovian strict MFE]\label{markovian}
The couple $(\boldsymbol{\rho}^*,P^*)\in\Pc_2(\C\times\mathcal{B})\times\Pc_2(\Omega)$ is said to be a Markovian strict MFE in the weak sense (Markovian {\rm (SW)-MFE}) if $(\boldsymbol{\rho}^*,P^*)$ is an (S)-MFE and there exists a Borel measurable mapping $\psi:[0,T]\times\R\to U$ such that $P^*(\Lambda_t(\d u)\d t=\delta_{\psi(t,X_t)}(\d u)\d t)=1$.
\end{definition}
\begin{remark}
When the mean field term only depends on the law of state variable and there is no reflection, \defref{relaxed_MFE} and \defref{strict_MFE} reduce respectively to the definitions of relaxed MFG solution and strict MFG solution given in Lacker \cite{Lacker}.
\end{remark}
	
\subsection{$N$-player game with state reflections}
This subsection is devoted to introducing the $N$-player game with state reflections in both strict and relaxed formulations.
	
We first give a strict formulation in the strong sense. Let the filtered probability space $(\Omega,\F,\Pb)$ support (i) $N$ i.i.d. $\R^2$-valued continuous processes $(W^1,A^1),\ldots,(W^N,A^N)$ with the same law of $\hat P$ (recall the probability measure $\hat P$ defined on $\C^W\times\tilde\Omega$ before \equref{CB_space}) such that $\E\left[\sup_{t\in [0,T]}|A^i_t|^p\right]<\infty$ for $i=1,\ldots,N$ and $W^1,\ldots,W^N$ are standard Brownian motions  (ii) $N$ i.i.d. $\F_0$-measurable scalar r.v.s $\eta^1,\ldots,\eta^N$ with the same law of $\eta$ such that $A_0^i\leq \eta^i$ for $i=1,\ldots,N$. Denote by $\U^N[0,T]$ the $N$-Cartesian product of $\U[0,T]$. Let $\Fb$ be the natural filtration generated by $(W^1,A^1),\dots,(W^N,A^N)$. Assume that $\Fb$ satisfies the usual conditions. A generic element in the $N$-player control set $\U^N[0,T]$ shall be denoted by $\boldsymbol{\alpha}=(\alpha^1,\ldots,\alpha^N)$.
	
The following assumption is imposed.
\begin{ass}\label{ass3}
The mappings $b(t,x,\rho,u)$ and $\sigma(t,x,\rho,u)$ are uniformly Lipschitz continuous in $(x,\rho)\in\R\times\Pc_2(\R\times U)$ in the sense that, there is $M>0$ independent of $(t,u)\in [0,T]\times U$ such that, for all $(x,\rho),(x',\rho')\in\R\times\Pc_2(\R\times U)$,
\begin{align*}
\left|b(t,x',\rho',u)-b(t,x,\rho,u)\right|+\left|\sigma(t,x',\rho',u)-\sigma(t,x,\rho,u)\right|\leq M\left(|x-x'|+\mathcal{W}_{2,\R\times U}(\rho,\rho')\right).
\end{align*}
\end{ass}
For any $\boldsymbol{\alpha}\in\U^N[0,T]$, by the classical SDE theory, there exists a unique $L^p$-integrable solution $Y^{\boldsymbol{\alpha}}=(Y^{\boldsymbol{\alpha},1},\ldots,Y^{\boldsymbol{\alpha},N})$ to the SDE system that, for $i=1,\ldots,N$,
\begin{align}\label{N_SDE}
\begin{cases}
\displaystyle\d Y_t^{\boldsymbol{\alpha},i}=b(t,X_t^{\boldsymbol{\alpha},i},\rho_t^N,\alpha_t^i)\d t+\sigma(t,X_t^{\boldsymbol{\alpha},i},\rho_t^N,\alpha_t^i)\d W_t^i,\quad Y_0^i=\eta^i,\\[0.6em]
\displaystyle X_t^{\boldsymbol{\alpha},i}=\Gamma(A^i,Y^{\boldsymbol{\alpha},i})_t,\quad\rho_t^N=\frac{1}{N}\sum_{i=1}^N\delta_{(X_t^{\boldsymbol{\alpha},i},\alpha_t^i)}.
\end{cases}
\end{align}
For any $i=1,\ldots,N$, we define the cost functional by
\begin{align}\label{N_cost}
J^i(\boldsymbol{\alpha}):=\E\left[\int_0^Tf\left(t,X_t^{\boldsymbol{\alpha},i},\rho_t^N,\alpha_t^i\right)\d t+\int_0^Tc(t,X_t^{\boldsymbol{\alpha},i})\d R_t^i+g\left(X^{\boldsymbol{\alpha},i}_{T},\mu_T^N\right)\right]
\end{align}
with $R^i:=X^{\boldsymbol{\alpha},i}-Y^{\boldsymbol{\alpha},i}$ and $\mu_T^N:=\frac1N\sum_{i=1}^N\delta_{X^{\boldsymbol{\alpha},i}_T}$. For any $\beta\in\U[0,T]$, we introduce the $N$-player control $(\boldsymbol{\alpha}^{-i},\beta)$ as usual by $(\boldsymbol{\alpha}^{-i},\beta):=(\alpha^1,\ldots,\alpha^{i-1},\beta,\alpha^{i+1},\ldots,\alpha^N)$.
	
Now, let us provide the definition of $\boldsymbol{\epsilon}$-Nash equilibrium in the  strong sense for the $N$-player game. 
\begin{definition}[$\boldsymbol{\epsilon}$-Nash equilibrium in the strong sense]\label{NE}
Let $\boldsymbol{\epsilon}=(\epsilon^1,\ldots,\epsilon^N)\in\R_+^N$. A control  $\boldsymbol{\alpha}\in \U^N[0,T]$ is called an $\boldsymbol{\epsilon}$-(open loop) Nash equilibrium of the $N$-player game \equref{N_SDE}-\equref{N_cost} if it holds that
\begin{align}\label{N_NE_strict}
J^i(\boldsymbol{\alpha})\leq \inf_{\beta\in\U[0,T]}J^i((\boldsymbol{\alpha}^{-i},\beta))+\epsilon^i,\quad \forall i=1,\ldots,N.
\end{align}
\end{definition}
A control $\boldsymbol{\alpha}\in\U^N[0,T]$ can induce a probability measure $\Pb\circ ((Y^{\boldsymbol{\alpha},i},\delta_{\alpha_t^i}(\d u)\d t,W^i,A^i)_{i=1}^N)^{-1}$ on $\Omega^N$. The definition of such probability measure on $\Omega^N$ is also invariant to the choice of the probability space $(\Omega,\F,\Fb,\Pb)$. Thus, in the sequel, we will consider the $N$-player game in the weak formulation. We adopt the notation $(\Omega^N,\F^N,\Fb^N)$ for the $N$-Cartesian product of $(\Omega,\F,\Fb)$ respectively. Similarly, the coordinate mappings are denoted by $(Y^i,\Lambda^i,W^i,A^i)_{i=1}^N$. By \lemref{Lipschitz_P}, we readily have the next result. 
\begin{lemma}\label{extension_ass_N}
Let \assref{ass3} hold. The  extensions $\tb(t,x,\xi,u)$ and $\ts(t,x,\xi,u)$ are uniformly Lipschitz continuous in $(x,\xi)\in\R\times\Pc_2(\R\times \Pc(U))$ such that there is $M>0$ independent of $(t,u)\in [0,T]\times U$ such that, for all $(x,\xi),(x',\xi')\in\R\times\Pc_2(\R\times\Pc(U))$,
\begin{align*}
\left|\tilde{b}(t,x',\xi',u)-\tilde{b}(t,x,\xi,u)\right|+\left|\tilde{\sigma}(t,x',\xi',u)-\tilde{\sigma}(t,x,\xi,u)\right|\leq M\left(|x-x'|+\mathcal{W}_{2,\R\times \Pc(U)}(\xi,\xi')\right).
\end{align*}
\end{lemma}
	
Now, we introduce the admissible $N$-player relaxed control.
\begin{definition}[$N$-player relaxed control]\label{relaxed_strategy}
We call a probability measure $P^N\in\Pc_2(\Omega^N)$ on $(\Omega^N,\F^N)$ an admissible $N$-player relaxed control rule (denoted by $R_N$) if it holds that {\rm(i)} $P^N\circ (Y_0^1,\cdots,Y_0^N) ^{-1}=\Law^{\Pb}(\eta^1,\cdots,\eta^N)$; {\rm(ii)} $P^N(Y_0^i\geq A_0^i)=1$, $\forall i=1,\ldots,N$; {\rm(iii)} the restriction of $P^N$ to $(\C^W\times\tilde\Omega)^N$, $P^N|_{(\C^W\times\tilde\Omega)^N}$ agrees with the joint law $\Pb\circ \left((W^1,A^1),\cdots,(W^N,A^N)\right)^{-1}$ (in the strong sense); {\rm(iv)} for all $\phi\in C_b^2(\R^N\times\R^N)$, the process
\begin{align*}
{\tt M}^N\phi(t)&:=\phi(Y_t^1,\cdots,Y_t^N,W_t^1,\cdots,W_t^N)\nonumber\\
&\quad-\sum_{i=1}^N\left(\int_0^t\int_U\bar{\mathbb{L}}^N_i\phi(s,X_s^1,\cdots,X_s^N,Y_s^1,\cdots,Y_s^N,W_s^1,\cdots,W_s^N,\xi_s^N,u)\Lambda^i_s(\d u)\d s\right)
\end{align*}
is a $(P^N,\Fb^N)$-martingale. Here, the infinitesimal generator acting on $\phi\in C_b^2(\R^N\times\R^N)$ is defined by, for $(t,\boldsymbol{x},\boldsymbol{y},\boldsymbol{w},\xi,u)\in [0,T]\times\R^N\times\R^N\times\R^N\times\Pc_2(\R\times\Pc(U))\times U$ with $\boldsymbol{x}=(x_1,\cdots,x_N)$ $\boldsymbol{y}=(y_1,\cdots,y_N)$ and $\boldsymbol{w}=(w_1,\cdots,w_N)$,
\begin{align*}
\bar{\mathbb{L}}^N_i\phi(t,\boldsymbol{x},\boldsymbol{y},\boldsymbol{w},\xi,u):=\bar b\left(t,x_i,\xi,u\right)^{\T}\nabla_i\phi(\boldsymbol{y},\boldsymbol{w})+\frac12\tr\left(\bar\sigma\bar\sigma^{\T}\left(t,x_i,\xi,u\right)\nabla_i^2\phi(\boldsymbol{y},\boldsymbol{w})\right),
\end{align*}
where the coefficients
\begin{align*}
\nabla_i\phi(\boldsymbol{y},\boldsymbol{w})=\begin{pmatrix}
\pa_{y_i}\phi(\boldsymbol{y},\boldsymbol{w})\\
\pa_{w_i}\phi(\boldsymbol{y},\boldsymbol{w})
\end{pmatrix},\qquad \nabla_i^2\phi(\boldsymbol{y},\boldsymbol{w})=\begin{pmatrix} \pa_{y_iy_i}^2\phi(\boldsymbol{y},\boldsymbol{w}), &\pa_{y_iw_i}^2\phi(\boldsymbol{y},\boldsymbol{w})\\
\pa_{w_iy_i}^2\phi(\boldsymbol{y},\boldsymbol{w}),
&\pa_{w_iw_i}^2\phi(\boldsymbol{y},\boldsymbol{w})
\end{pmatrix},
\end{align*}
$(\bar b,\bar\sigma)$ is defined in \defref{relaxed_control}
and $\boldsymbol{\xi}^N=\frac1N\sum_{i=1}^N\delta_{(X^i,\Lambda^i)}$, $X^i(\omega)=\Gamma(A^i(\omega),Y^i(\omega)))$ with $\xi_t^{N}$ being the $t$-marginal distribution and the mapping $\Gamma:\mathcal{D}\mapsto\mathcal{C}$ being defined by \equref{eq:Skorohodmap}.
		Furthermore, if there exists an $\mathbb{F}^N$-progressively measurable $U^N$-valued process $\boldsymbol{\alpha}=(\alpha^1,\ldots,\alpha^N)=\left((\alpha^1_t)_{t\in [0,T]},\ldots,(\alpha_t^N)_{t\in [0,T]}\right)$ on $\Omega^N$ such that $P^N(\Lambda^i_t(\d t)\d t=\delta_{\alpha_t^i}(\d u)\d t,i=1,\ldots,N)=1$, we say that $P^N$ corresponds to a strict control $\boldsymbol{\alpha}$. The set of all $N$-player strict controls is denoted by $R^{\rm s}_N$.
	\end{definition}
	
	\begin{remark}
		The conditions (i) and (iii) of \defref{relaxed_strategy} ensure that $(Y_0^i)_{i=1}^N$, $ (A^i)_{i=1}^N$ are i.i.d. under $P^N$, respectively.
    \end{remark}
	Similarly, we have the following standard martingale measure representation for the $N$-player relaxed control in the weak formulation whose proof is omitted. 
	\begin{lemma}\label{moment_p_N}
		We have that $P^N\in R_N$ iff there exists a filtered probability space $(\Omega',\F',\Fb'=(\F'_t)_{t\in[0,T]},P')$ supporting a $\Pc(U)^N$-valued $\Fb'$-progressively measurable process $\boldsymbol{\Lambda}=(\Lambda^i)_{i=1}^N=((\Lambda_t^i)_{t\in [0,T]})_{i=1}^N$, N scalar $\Fb'$-adapted processes $(Y^{\boldsymbol{\Lambda},i})_{i=1}^N=((Y_t^{\boldsymbol{\Lambda},i})_{t\in[0,T]})_{i=1}^N$, $N$ i.i.d. one dimensional standard $\F_t'$-Brownian motion $(W^i)_{i=1}^N=((W_t^i)_{t\in [0,T]})_{i=1}^N$, N independent $\Fb'$-martingale measures $({\cal M}^i)_{i=1}^N$ on $U\times[0,T]$, with intensity $(\Lambda_t^i(\d u)\d t)_{i=1}^N$ and N independent real-valued $\mathbb{F}'$-adapted continuous processes $(A^i)_{i=1}^N=((A_t^i)_{t\in[0,T]})_{i=1}^N$ such that $P^N=P'\circ ((Y^{\boldsymbol{\Lambda},i},\Lambda^i,W^i,A^i)_{i=1}^N)^{-1}$,  and it holds that {\rm(i)} $P'\circ(Y_0^{\boldsymbol{\Lambda},1},\cdots,Y_0^{\boldsymbol{\Lambda},N})^{-1}=\Law^{\mathbb{P}}(\eta^1,\cdots,\eta^N)$; {\rm(ii)} $P'(Y_0^{\boldsymbol{\Lambda},i}\geq A_0^i)=1$,~$\forall i=1,\ldots,N$; {\rm(iii)} $W_t^i=\int_0^t\int_U\mathcal{M}^i(\d u,\d t),~\forall t\in [0,T],~i=1,\ldots,N$,~$P'$-$\as$;
        {\rm(iv)} the state process obeys that, $P'$-a.s.,
\begin{align}\label{SDE_N}
\d Y^{\boldsymbol{\Lambda},i}_t&=\int_U\tb(t,X^{\boldsymbol{\Lambda},i}_t,\xi_t^N,u)\Lambda_t^i(\d u)\d t+\int_U\ts(t,X^{\boldsymbol{\Lambda},i}_t,\xi_t^N,u){\cal M}^i(\d u,\d t),\nonumber\\ X^{\boldsymbol{\Lambda},i}_t&=\Gamma(A^i,Y^{\boldsymbol{\Lambda},i})_t,\quad \boldsymbol{\xi}^N:=\frac{1}{N}\sum_{i=1}^N\delta_{(X^{\boldsymbol{\Lambda},i},\Lambda^i)}
\end{align}         
with $\xi_t^N$ being the $t$-marginal distribution of $\boldsymbol{\xi}^N$. Moreover, there exists a constant $C>0$ depending on $(M,\Law^{\mathbb{P}}(\eta),T)$ such that, for all $i=1,\ldots,N$,	$\E^{P'}\left[\sup_{t\in[0,T]}|Y_t^{\Lambda,i}|^p\right]\leq C$,  where $M$ is given in \assref{ass1}.
\end{lemma}
	
\begin{remark}
Using Lemma 3.2 in \cite{Lacker2}, Eq.~\equref{SDE_N} admits a unique strong solution by \assref{ass1} and \lemref{extension_ass_N}.
\end{remark}
	
	For $P\in R_N^{\rm s}$, we introduce the subset $R^{\rm s}_{i,N}(P)$ of $R_N^{\rm s}$ in order to describe the solution in the  weak sense, the scenario when the player varies his control while the $N-1$ remaining players keep their strategies unchanged.
	\begin{definition}\label{RsiN}
For $P^N\in R_N^{\rm s}$, we say $Q^N\in R^{\rm s}_{i,N}(P^N)$ if there exists a filtered probability space $(\Omega',\F',\Fb'=(\F'_t)_{t\in[0,T]},P')$ supporting $N+1$ $U$-valued $\Fb'$-progressively measurable process $((\alpha^j)_{j=1}^N,\beta)=(((\alpha_t^j)_{t\in [0,T]})_{j=1}^N,(\beta_t)_{t\in [0,T]})$, $N$ i.i.d. real-valued $\F_0'$-measurable r.v. $(\eta^j)_{j=1}^N$ with law $\Law^{\Pb}(\eta)$, $N$ independent scalar $\Fb'$-Brownian motions $((W^j)_{t\in [0,T]})_{j=1}^N$ and $N$ independent real-valued $\mathbb{F}'$-adapted continuous process $(A^j)_{j=1}^N=((A_t^j)_{t\in[0,T]})_{i=1}^N$ and $((Y_t^j)_{t\in [0,T]})_{j=1}^N,((Z_t^j)_{t\in [0,T]})_{j=1}^N$ (strongly) solving the SDE system:
		\begin{align*}
			Y_t^j&=\eta^j+\int_0^t b(s,\Gamma(A^j,Y^j)_s,\rho^N_s,\alpha^j_s)\d s+\int_0^t\sigma(s,\Gamma(A^j,Y^j)_s,\rho^N_s,\alpha^j_s)\d W_s^j,~j=1,\ldots,N,\\
			Z_t^j&=\eta^j+\int_0^t b(s,\Gamma(A^j,Z^j)_s,\theta^N_s,\alpha^j_s)\d s+\int_0^t\sigma(s,\Gamma(A^j,Z^j)_s,\theta^N_s,\alpha^j_s)\d W_s^j,~j\neq i,\\
			Z_t^i&=\eta^i+\int_0^t b(s,\Gamma(A^i,Z^i)_s,\theta^N_s,\beta_s)\d s+\int_0^t\sigma(s,\Gamma(A^i,Z^i)_s,\theta^N_s,\beta_s)\d W_s^i,
		\end{align*}
		with $\rho_t^N=\frac1N\sum_{j=1}^N\delta_{(\Gamma(A^j,Y^j)_t,\alpha_t^j)}$ and $ \theta_t^N=\frac1N\left(\sum_{j\neq i}\delta_{(\Gamma(A^j,Z^j)_t,\alpha_t^j)}+\delta_{(\Gamma(A^i,Z^i)_t,\beta_t)}\right)$ such that
		\begin{align*}
			P^N&=P'\circ \left((Y^j)_{j=1}^N,(\delta_{\alpha^j_t}(\d u)\d t)_{j=1}^N,(W^j)_{j=1}^N,(A^j)_{j=1}^N\right)^{-1},\nonumber\\ 
			Q^N&=P'\circ \left((Z^j)_{j=1}^N,(\delta_{\boldsymbol{\alpha}_t^{-i}}(\d u)\d t,\delta_{\beta_t}(\d u)\d t),(W^j)_{j=1}^N,(A^j)_{j=1}^N\right)^{-1},  
		\end{align*}
		where we denote
		\begin{align*}
			\left(\delta_{\boldsymbol{\alpha}_t^{-i}}(\d u),\delta_{\beta_t}(\d u)\right):=\left(\delta_{\alpha^1_t}(\d u),\ldots,\delta_{\alpha^{i-1}_t}(\d u),\delta_{\beta_t}(\d u),\delta_{\alpha^{i+1}_t}(\d u),\ldots, \delta_{\alpha^N_t}(\d u)\right).    
		\end{align*}
		
	\end{definition}

\begin{remark}
It follows from \lemref{moment_p_N} that $P^N\in R_{i,N}^{\rm s}(P^N)$ for any $P^N\in R_N^{\rm s}$, and hence $R_{i,N}^{\rm s}(P^N)$ is nonempty.
\end{remark}
	
We next introduce the $\boldsymbol{\epsilon}$-Nash equilibrium in the weak sense in the $N$-player game. Let us first consider the cost functional for player $i$ that
\begin{align}\label{N_cost_relaxed}
\mathcal{J}_i(P^N):=\E^{P^N}\left[\Delta_i(\boldsymbol{\xi}^N)\right],
\end{align}
where $\boldsymbol{\xi}^N=\frac1N\sum_{i=1}^N\delta_{(X^i,\Lambda^i)}$ and $\Delta_i(\boldsymbol{\xi})$ is defined by
\begin{align}\label{Delta_i}
\Delta_i(\boldsymbol{\xi}):=\int_0^T\int_U\tf(t,X_t^i,\xi_t,u)\Lambda_t^i(\d u)\d t+\int_0^Tc(t,X_t^i)\d R^i_t+g(X_T^i,\mu_T),~~\forall\boldsymbol{\xi}\in \Pc_2(\C\times\mathcal{Q}),
\end{align}
where $\mu_T$ is the first marginal distribution of $\xi_T$ and $R^i=X^i-Y^i$.
\begin{definition}[$\boldsymbol{\epsilon}$-Nash equilibrium in the weak sense for $N$-player game]\label{NE_weak}
For $\boldsymbol{\epsilon}=(\epsilon^1,\ldots,\epsilon^N)\in\R_+^N$, a strict control $P^N\in R_N^{\rm s}$ is an $\boldsymbol{\epsilon}$-(open loop) Nash equilibrium in the weak sense for the $N$-player game in \defref{relaxed_strategy} and \equref{N_cost_relaxed} if it holds that
\begin{align}\label{N_NE}
\mathcal{J}_i(P^N)\leq \inf_{Q^N\in R^{\rm s}_{i,N}(P^N)}\mathcal{J}_i(Q^N)+\epsilon^i,\quad\forall i=1,\ldots,N,
\end{align}
where the set $R_{i,N}^{\rm s}(P^N)$ is given in \defref{RsiN}.
\end{definition}
	
\begin{remark}
We introduce the $\boldsymbol{\epsilon}$-Nash equilibrium in the weak sense to cope with the weak formulation that we employed. This definition serves to generalize the notion of Nash equilibrium in the strong sense. Moreover, $\boldsymbol{\epsilon}$-Nash equilibrium in the weak sense (\defref{NE_weak}) implies $\boldsymbol{\epsilon}$-Nash equilibrium in the strong sense (\defref{NE}) because if $\boldsymbol{\alpha}=(\alpha^1,\ldots,\alpha^N)$ is an admissible $N$-player control in the strong sense, with $(X^1,\ldots,X^N)$ being the corresponding state process and if $P^N=\Pb\circ((Y^{\boldsymbol{\alpha},i},\delta_{\alpha_t^i}(\d u)\d t,W^i,A^i)_{i=1}^N)^{-1}$ constitutes an $\boldsymbol{\epsilon}$-Nash equilibrium in the weak sense, then by definition, $\boldsymbol{\alpha}$ also yields an $\boldsymbol{\epsilon}$-Nash equilibrium in the strong sense.
\end{remark}
For the rest of the paper, we only consider the $\boldsymbol{\epsilon}$-Nash equilibrium in the weak sense in \defref{NE_weak} for the $N$-player game.

	\section{Main Results}\label{sec:existence}

\subsection{Existence of MFEs}
	
This subsection presents our main results on the existence of an (R)-MFE in \defref{relaxed_MFE} as well as the existence of a Markovian (SW)-MFE in \defref{markovian}.
	
\begin{theorem}\label{existence_RMFE}
Let \assref{ass1} hold. There exists an {\rm(R)-MFE}.	
\end{theorem}
	
The proof of \thmref{existence_RMFE} can be sketched as three main steps. The detailed and lengthy proof is given in subsection \ref{proof:thm1}. According to the discussion below Definition~\ref{relaxed_MFE}, it suffices to identify the fixed point of the mapping $\mathcal{R}^*$ defined in \eqref{eq:Kstar}. For an arbitrary  $\Pc_2(\C\times\mathcal{Q})$-valued probability measure	$\boldsymbol{\xi}$, the roadmap for achieving this goal is shown as below:	\ \\
\indent\textbf{Step-(i)} Show that $R_{\rm opt}(\boldsymbol{\xi})$ defined by \eqref{optimal_rule} is a (nonempty) compact and convex subset of the set$R(\boldsymbol{\xi})$ of relaxed controls (\lemref{compactness} in subsection \ref{proof:thm1}); \\
\indent\textbf{Step-(ii)} Prove that the set-valued mapping $\boldsymbol{\xi} \mapsto R_{\rm opt}(\boldsymbol{\xi})$ is upper semi-continuous by first showing that the set-valued mapping $\boldsymbol{\xi} \mapsto R(\boldsymbol{\xi})$ is continuous (\lemref{closed}, \lemref{continuity} and \lemref{u.s.c.} in subsection \ref{proof:thm1});\\
\indent\textbf{Step-(iii)} Construct a compact subset $\mathcal{M}$ to which the restriction of the set-valued mapping $\mathcal{R}^*$ is a self-mapping and then apply Kakutani's fixed point theorem to $\mathcal{R}^*|_{\mathcal{M}}$ to establish the existence of the fixed point (\lemref{compact_set} in subsection \ref{proof:thm1}).
\begin{remark}\label{existence_Lip}
We can establish the same existence result in \thmref{existence_RMFE} under \assref{ass1} and \assref{ass3} without imposing the separation condition in {\rm (A1)} and {\rm (A5)} of \assref{ass1} because \lemref{closed} can still be derived (see \remref{closed_Lip}), and the remaining proofs follow the same arguments as in \thmref{existence_RMFE}.
\end{remark}

To further establish the existence of a Markovian (SW)-MFE in \defref{markovian}, we also need some additional assumption stated as below.      
\begin{ass}\label{ass2}
\begin{itemize}
\item [{\rm (B1)}] There exist measurable mappings $b_0,\sigma_0,f_0:[0,T]\times\R\times\Pc_2(\R)\times U$ and $\lambda_1,\lambda_2,\lambda_3\in\R$  such that, for all $(t,x,\rho,u)\in [0,T]\times\R\times\Pc_2(\R\times U)\times U$,
\begin{align*}
b(t,x,\rho,u)&=b_0(t,x,\mu,u)+\lambda_1\int_{\R\times U}b_0(t,y,\mu,u)\rho(\d y, \d u),\\
\sigma^2(t,x,\rho,u)&=\sigma_0^2(t,x,\mu,u)+\lambda_2\int_{\R\times U}\sigma^2_0(t,y,\mu,u)\rho(\d y, \d u),\\	f(t,x,\rho,u)&=f_0(t,x,\mu,u)+\lambda_3\int_{\R\times U}f_0(t,y,\mu,u)\rho(\d y, \d u),
\end{align*}
where $\mu\in\Pc_2(\R)$ is the first marginal of $\rho$.
\item [{\rm (B2)}] For any $(t,x,\mu)\in[0,T]\times\R\times\Pc_2(\R)$, the following set is convex:
\begin{align*}
K(t,x,\mu):=\left\{(b_0(t,x,\mu,u),\sigma_0^2(t,x,\mu,u),z);~ z\geq f_0(t,x,\mu,u),~u\in U\right\}.    
\end{align*}
\end{itemize}
\end{ass}
	
       
First, note that, for any $\boldsymbol{\xi}\in\Pc_2(\C\times\mathcal{Q})$, every admissible strict control (in the weak sense) $P\in R^{\rm s}(\boldsymbol{\xi})$ is an admissible relaxed control by construction, i.e., $P\in R(\boldsymbol{\xi})$, and hence 
 \begin{align*}
 \inf_{P\in R^{\rm s}(\boldsymbol{\xi})}\mathcal{J}(P,\boldsymbol{\xi})\geq \inf_{P\in R(\boldsymbol{\xi})}\mathcal{J}(P,\boldsymbol{\xi}).
\end{align*}
The next proposition shows the reverse inequality.
\begin{prop}\label{measurable}
Let \assref{ass2} hold. For any $P\in R(\boldsymbol{\xi})$, there exists some $Q\in R^{\rm s}(\boldsymbol{\xi})$ such that $\mathcal{J}(Q;\boldsymbol{\xi})\leq \mathcal{J}(P;\boldsymbol{\xi})$. Moreover, if $(\boldsymbol{\xi}^*,P^*)$ is an {\rm (R)-MFE}, then there exist $\boldsymbol{\rho}^*\in\Pc_2(\C\times\mathcal{B})$ with $\tilde{\boldsymbol{\xi}}=\mathscr{R}(\boldsymbol{\rho}^*)$ and $Q^*\in R^{\rm s}(\tilde{\boldsymbol{\xi}})$ such that $(\boldsymbol{\rho}^*,Q^*)$ is an (SW)-MFE and the first marginal distributions of $\boldsymbol{\rho}^*$ and $\tilde{\boldsymbol{\xi}}$ coincide, i.e., $\boldsymbol{\rho}^*|_{\C}=\tilde{\boldsymbol{\xi}}|_{\C}$.
\end{prop}
As a direct consequence of \thmref{existence_RMFE} and \propref{measurable}, we can get the existence of an (SW)-MFE in \defref{strict_MFE}.

	\begin{corollary}\label{existence_SMFE}
		Let \assref{ass1} and \assref{ass2} hold. There exists an {\rm(SW)-MFE}. 
	\end{corollary}
	Furthermore, under the same assumptions, we can also derive the existence of a Markovian (SW)-MFE in \defref{markovian}  in the next result, whose proof is delegated to subsection \ref{appendix}.
	\begin{corollary}\label{existence_Markovian}
		Let \assref{ass1} and \assref{ass2} hold. There exists a Markovian {\rm (SW)-MFE}.
	\end{corollary}

\subsection{Limit theory}\label{sec:limit}
	
The goal of this subsection is to establish some connections between the MFG problem and the $N$-player game problem with joint state-control interactions and state reflections. 
	
First, the following result establishes the convergence of $\boldsymbol{\epsilon}^N$-Nash equilibria as $N$ tends to infinity, closely resembling the propagation of chaos result in the mean-field control theory (Theorem 2.11 in Lacker \cite{Lacker1}).
	
\begin{theorem}\label{propagation}
Let \assref{ass1} and \assref{ass3} hold. Suppose that $P^N\in\Pc_2(\Omega^N)$ is an $\boldsymbol{\epsilon}^N$-Nash equilibrium in \defref{NE_weak} for $N\geq1$ with $\lim_{N\to\infty}\frac1N\sum_{i=1}^N\epsilon^{i,N}=0$. Then, $(P^N\circ (\boldsymbol{\xi}^N,\boldsymbol{\delta}^N)^{-1})_{N\geq 1}$ is precompact in $\Pc_2(\Pc_2(\C\times\mathcal{Q})\times\Pc_2(\Omega))$, and each of its limit point $\Xi\in\Pc_2(\Pc_2(\C\times\mathcal{Q})\times\Pc_2(\Omega))$ is supported on $\{(\boldsymbol{\xi},P)\in\Pc_2(\C\times\mathcal{Q})\times\Pc_2(\Omega);~(\boldsymbol{\xi},P) \text{ is an {\rm (R)-MFE}}\}$ with $\boldsymbol{\delta}^N:=\frac1N\sum_{i=1}^N\delta_{(X^i,\Lambda^i, W^i,A^i)}$.
\end{theorem}
	
The proof of the above theorem is technical, which is given in subsection \ref{proof:thm2}. We can outline the arguments by the following three main steps: \\
\indent\textbf{Step-(i)} Show the precompactness of {$(\Xi^N)_{N\geq 1}$} with $\Xi^N:=P^N\circ (\boldsymbol{\xi}^N,\boldsymbol{\delta}^N)^{-1}$ (\lemref{precompactness} in subsection \ref{proof:thm2});\\
\indent
\textbf{Step-(ii)} Identify the limit points of $P^N\circ (\boldsymbol{\xi}^N,\boldsymbol{\delta}^N)^{-1}$ as $N\to\infty$, and prove that every limit point is supported on the set $\{(\boldsymbol{\xi},P)\in\Pc_2(\C\times\mathcal{Q})\times\Pc_2(\Omega);~P\in R(\boldsymbol{\xi})~\text{and}~P\circ(X,\Lambda)^{-1}=\boldsymbol{\xi}\}$ (\lemref{pre_MFE_support} in subsection \ref{proof:thm2});\\
	\indent
	\textbf{Step-(iii)} Complete the proof by verifying the optimality (\lemref{relation}, \lemref{relation_value} and the proof of \thmref{propagation} in subsection \ref{proof:thm2}).
	
\begin{remark}\label{uniqueness}
In the above convergence result, if the uniqueness of {\rm (R)-MFE} also holds, every limit point   $\Xi$ of $(P^N\circ (\boldsymbol{\xi}^N,\boldsymbol{\delta}^N)^{-1})_{N\geq 1}$ concentrates its entire mass on the unique {\rm (R)-MFE} $(\boldsymbol{\xi}^*,P^*)$. More precisely, $\Xi=\delta_{(\boldsymbol{\xi}^*,P^*)}$ for every limit point $\Xi$, and hence the sequence $(P^N\circ (\boldsymbol{\xi}^N,\boldsymbol{\delta}^N)^{-1})_{N\geq 1}$ converges to $\delta_{(\boldsymbol{\xi}^*,P^*)}$ as $N\to\infty$, or equivalently $\lim\limits_{N\to\infty}\E^{P^N}[\mathcal{W}_{2,\Omega}(\boldsymbol{\delta}^N,P^*)]=0$.
	\end{remark}
The converse limit result can also be established. Specifically, we can construct a sequence of approximate $\boldsymbol{\epsilon}^N$-Nash equilibria that converges to a given Markovian {\rm (SW)-MFE} in the mean field model, as stated in the next main result of this section. The detailed proof is given in subsection \ref{proof:thm3}.
\begin{theorem}\label{convergence_theorem}
Let \assref{ass1} and \assref{ass3} hold. If $(\boldsymbol{\rho}^*,P^*)$ is a Markovian {\rm (SW)-MFE} for the MFG problem, then there exist $P^N\in R_N^{\rm s}$ and $\boldsymbol{\epsilon}^N=(\epsilon^{1,N},\ldots,\epsilon^{N,N})\in\R_+^N$ for $N\geq1$ satisfying $\lim_{N\to\infty}\vee_{i=1}^N\epsilon^{i,N}=0$ such that $P^N$ is an $\boldsymbol{\epsilon}^N$-Nash equilibrium as in \defref{NE_weak}, and the following convergence holds:
		\begin{align}\label{convergence}
			\lim_{N\to\infty}\E^{P^N}\left[\mathcal{W}_{2,\Omega}\left(\frac1N\sum_{i=1}^N\delta_{(Y^i,\Lambda^i,W^i,A^i)},P^*\right)\right]=0.
		\end{align}
		As a consequence, we have
		\begin{align}\label{convergence_weak}
			\lim_{N\to\infty}\mathcal{W}_{2,\Pc_2(\Omega)}\left(P^N\circ\left(\frac1N\sum_{i=1}^N\delta_{(Y^i,\Lambda^i,W^i,A^i)}\right)^{-1},\delta_{P^*}\right)=0.
		\end{align}
	\end{theorem}
    
\begin{remark}
Note that the constructed sequence of $\boldsymbol{\epsilon}^N$-Nash equilibrium in the above theorem with $\lim_{N\to\infty}\vee_{i=1}^N\epsilon^{i,N}=0$ is a stronger condition than the one in \thmref{propagation}, in which, we only require that $\lim_{N\to\infty}\frac1N\sum_{i=1}^N\epsilon^{i,N}=0$. Moreover, the convergence we established in \equref{convergence} is also generally stronger than the convergence result in \thmref{propagation} and the two convergence claims become equivalent when the uniqueness of {\rm (R)-MFE} is validated (see \remref{uniqueness}).
\end{remark}

	\section{Proofs}\label{sec:proofs}
	\subsection{Proofs of Main Theorems}\label{sec:proof-main-result}
	This subsection presents the technical preparations and lengthy proofs of the main theorems, namely \thmref{existence_RMFE},  \thmref{propagation} and \thmref{convergence_theorem}, in Section \ref{sec:existence}. 
	
\subsubsection{Proof of \thmref{existence_RMFE}}\label{proof:thm1}
Firstly, we prove the claimed result in \textbf{Step-(i)}: 
    
\begin{lemma}\label{compactness} For any probability measure $\boldsymbol{\xi}\in\Pc_2(\C\times\mathcal{Q})$ with $(\xi_t)_{t\in [0,T]}$ being the corresponding $\Pc_2(\R\times\Pc(U))$-valued measure flow, it holds that {\rm(i)} $R(\boldsymbol{\xi})$ is compact and convex in $\Pc_2(\Omega)$; {\rm(ii)} $R_{\rm opt}(\boldsymbol{\xi})$ is a (nonempty) compact and convex set in $\Pc_2(\Omega)$.
\end{lemma}
	
\begin{proof}
(i) We first show that $\{P\circ Y^{-1};~P\in R(\boldsymbol{\xi})\}$ is tight. Note that $\sup_{P\in R(\boldsymbol{\xi})}\E^ P\left[|Y_0|^2\right]=\int_{\R}|y|^2\eta(\d y)<\infty$. By \lemref{moment_p} and H\"{o}lder's inequality, we have $\sup_{P\in R(\boldsymbol{\xi})}\E^P\left[\left|Y_t-Y_s\right|^p\right]\leq C|t-s|^{1+\frac{p}{2}}$, 
where $C>0$ is a constant depending on $(M,\Law^P(\eta),T)$. In view of Kolmogorov's criterion, the set $\{P\circ Y^{-1};~P\in R(\boldsymbol{\xi})\}$ is tight. On the other hand, $\mathcal{Q}$ is compact, and hence $\{P\circ \Lambda^{-1};~P\in R(\boldsymbol{\xi})\}$ is tight. Moreover, $P\circ W^{-1}=\mathcal{W}$ for all $P\in R(\boldsymbol{\xi})$, where $\mathcal{W}$ denotes the Wiener measure on $\C^W$. Consequently, $R(\boldsymbol{\xi})$ is also tight. The $L^p$-boundedness in \lemref{moment_p}, together with the compactness of the policy space $U$, allows us to upgrade the tightness to precompactness in $\Pc_2(\Omega)$ (c.f. Proposition B.3 in Lacker \cite{Lacker}).
		
We next prove that $R(\boldsymbol{\xi})$ is closed. Let $(P_n)_{n\geq1}\subset R(\boldsymbol{\xi})$ such that $P_n\to P$ in $\Pc_2(\Omega)$ as $n\to\infty$. One can easily verify that $P$ satisfies (i)-(iii) of \defref{relaxed_control}, thus it suffices to show that (iv) of \defref{relaxed_control} holds. For any $\phi\in C_b^2(\R)$, $0\leq s\leq t\leq T$ and $\F_s$-measurable bounded $\RV$ $h$, we have
		\begin{align*}
			\E^P\left[\left({\tt M}^{\boldsymbol{\xi}}\phi(t)-{\tt M}^{\boldsymbol{\xi}}\phi(s)\right)h\right]=\lim_{n\to\infty}\E^{P_n}\left[\left({\tt M}^{\boldsymbol{\xi}}\phi(t)-{\tt M}^{\boldsymbol{\xi}}\phi(s)\right)h\right]=0, \end{align*}
		where the convergence follows from the fact that ${\tt M}^{\boldsymbol{\xi}}\phi(t)$ has at most linear growth in $\omega\in\Omega$. This implies that $P\in R(\boldsymbol{\xi})$. Finally, the convexity of $R(\boldsymbol{\xi})$ can be easily checked, which concludes the first assertion.
		
		(ii) Note that $\Delta(\boldsymbol{\xi})$ is continuous in $\omega\in\Omega$ due to Proposition 3.4 of Zalinescu \cite{Zalinescu}, and so is $\mathcal{J}(P,\boldsymbol{\xi})$ in $P\in R(\boldsymbol{\xi})$. Recall that a continuous mapping has a minimum point over a compact set, and therefore $R_{\rm opt}(\boldsymbol{\xi})$ is nonempty. It can be easily verified that $R_{\rm opt}(\boldsymbol{\xi})$ is closed and convex. As a closed subset of the compact set $R(\boldsymbol{\xi})$, $R_{\rm opt}(\boldsymbol{\xi})$ is also compact. Thus, the desired claim holds.
	\end{proof}
	
We next proceed the argument in \textbf{Step-(ii)}. Let $\comp(\Pc_2(\Omega))$ be the collection of compact subsets in $\Pc_2(\Omega)$ and define a metric $\dist(\cdot,\cdot)$ on $\comp(\Pc_2(\Omega))$ as follows:
\begin{align*}
\dist(K_1,K_2):=\inf\left\{\epsilon>0;~K_1\subset B_{\epsilon}(K_2) \text{ and } K_2\subset B_{\epsilon}(K_1) \right\},~~\forall K_1,K_2\in\comp(\Pc_2(\Omega))
\end{align*}
with $B_{\epsilon}(K):=\{Q\in\Pc_2(\Omega);~\mathcal{W}_{2,\Omega}(Q,P)<\epsilon~\exists P\in K\}$ for $K\in{\rm comp}({\cal P}_2(\Omega))$. Then,  $(\comp(\Pc_2(\Omega)),\dist(\cdot,\cdot))$ is a Polish space (c.f. Chapter 12 in \cite{Stroock}).  We next construct the following two mappings:
	\begin{itemize}
		\item[(i)] $\textsf{R}:\Pc_2(\C\times\mathcal{Q})\mapsto \comp(\Pc_2(\Omega))$, which is defined by $\boldsymbol{\xi}\mapsto R(\boldsymbol{\xi})$;
		\item[(ii)] $\textsf{R}_{\rm opt}:\Pc_2(\C\times\mathcal{Q})\mapsto \comp(\Pc_2(\Omega))$, which is defined by $ \boldsymbol{\xi}\mapsto R_{\rm opt}(\boldsymbol{\xi})$.
	\end{itemize}
	We also introduce the graph $\Gr(\textsf{R})$ of the mapping $\textsf{R}$ by
	\begin{align*}
		\Gr(\textsf{R}):=\left\{(\boldsymbol{\xi},P)\in \Pc_2(\C\times\mathcal{Q})\times\Pc_2(\Omega);~\boldsymbol{\xi}\in\Pc_2(\C\times\mathcal{Q}),P\in R(\boldsymbol{\xi})\right\}.    
	\end{align*}
	
	We then show that $\textsf{R}$ is a closed mapping in the next result.
	\begin{lemma}\label{closed}
		The graph $\Gr(\textsf{R})$ of the mapping $\textsf{R}$ is closed in $\Pc_2(\C\times\mathcal{Q})\times\Pc_2(\Omega)$.
	\end{lemma}
	
\begin{proof}
Let $(\boldsymbol{\xi}_n)_{n\geq1}\subset \Pc_2(\C\times\mathcal{Q})$ and $P_n\in R(\boldsymbol{\xi}_n)$ for $n\geq1$ such that $(\boldsymbol{\xi}_n,P_n)\to(\boldsymbol{\xi},P)$ in $\Pc_2(\C\times\mathcal{Q})\times\Pc_2(\Omega)$ as $n\to\infty$. It suffices to show that $P\in R(\boldsymbol{\xi})$. To this end, it is sufficient to verify (iv) of \defref{relaxed_control}. Let $\mu_t^n$ be the first marginal of $\xi_t^n$. Recall the coefficients $\bar{b}$ and $\bar{\sigma}$ defined in \defref{relaxed_control}. Then, for any $\phi\in C_b^2(\R^n\times\R^n)$,
\begin{align*}
&\int_0^t\int_U\left(\bar b(s,X_s,\xi_s^n,u)-\bar b(s,X_s,\xi_s,u)\right)^{\T}\nabla\phi(Y_s,W_s)\Lambda_s(\d u)\d s\\
&\quad=\int_0^t\int_U\left(b_1(s,X_s,\mu_s^n,u)-b_1(s,X_s,\mu_s,u)\right)\nabla_y\phi(Y_s,W_s)\Lambda_s(\d u)\d s\\
&\qquad+k_1\int_0^t\int_U\left(\tb_2(s,X_s,\xi_s^n,u)-\tb_2(s,X_s,\xi_s,u)\right)\nabla_y\phi(Y_s,W_s)\Lambda_s(\d u)\d s\\
&\quad=\int_0^t\int_U\left(b_1(s,X_s,\mu_s^n,u)-b_1(s,X_s,\mu_s,u)\right)\nabla_y\phi(Y_s,W_s)\Lambda_s(\d u)\d s\\
&\qquad+k_1\int_0^t\int_U\int_{\R\times\mathcal{M}(U)}\int_U\left(b_3(s,X_s,\mu_s^n,u,x',u')-b_3(s,X_s,\mu_s,u,x',u')\right)\\
&\qquad\qquad\qquad\times\nabla_y\phi(Y_s,W_s)q'(\d u')\xi_s(\d x',\d q')\Lambda_s(\d u)\d s\\
&\quad=:I_1+k_1 I_2.
\end{align*}
For $(X',\Lambda')\in\C\times\mathcal{Q}$, let us define
\begin{align*}
\Xi^n(t,X',\Lambda',\omega)&:=\int_0^t\int_U\int_Ub_3(s,X_s(\omega),\mu_s^n,u,X_s',u')\nabla_y\phi(Y_s(\omega),W_s(\omega))\Lambda_s'(\d u')\Lambda_s(\omega,\d u)\d s,\\
\Xi(t,X',\Lambda',\omega)&:=\int_0^t\int_U\int_Ub_3(s,X_s(\omega),\mu_s,u,X_s',u')\nabla_y\phi(Y_s(\omega),W_s(\omega))\Lambda_s'(\d u')\Lambda_s(\omega,\d u)\d s.
\end{align*}
In light of the uniform continuity of $b_1,b_3$ imposed in \assref{ass1}, we have
		\begin{align}\label{uniform_convergence_pre}
			\lim_{n\to\infty}\sup_{(t,\omega)\in[0,T]\times\Omega}\left(|I_1|+\sup_{(X',\Lambda')\in\C\times\mathcal{Q}}\left|\Xi^n(t,X',\Lambda',\omega)-\Xi(t,X',\Lambda',\omega)\right|\right)=0.
		\end{align}
		On the other hand, we also have that, for $(t,\omega)\in[0,T]\times\Omega$,
		{\scriptsize\begin{align*}
				&|\Xi(t,X',\Lambda',\omega)-\Xi(t,Y',\Upsilon',\omega)|\\
				\leq&\left|\int_0^t\int_U\left(\int_Ub_3(s,X_s(\omega),\mu_s,u,X_s',u')\Lambda_s'(\d u')-\int_Ub_3(s,X_s(\omega),\mu_s,u,Y_s',u')\Upsilon_s'(\d u')\right)\nabla_y\phi(Y_s(\omega),W_s(\omega))\Lambda_s(\omega,\d u)\d s\right|\\
				\leq&\left|\int_0^t\int_U\int_U\left(b_3(s,X_s(\omega),\mu_s,u,X_s',u')-b_3(s,X_s(\omega),\mu_s,u,Y_s',u')\right)\nabla_y\phi(Y_s(\omega),W_s(\omega))\Lambda_s'(\d u)\Lambda_s(\omega,\d u)\d s\right|\\
				&+\left|\int_0^t\int_U\left(\int_Ub_3(s,X_s(\omega),\mu_s,u,Y_s',u')\Lambda_s'(\d u')-\int_Ub_3(s,X_s(\omega),\mu_s,u,Y_s',u')\Upsilon_s'(\d u')\right)\nabla_y\phi(Y_s(\omega),W_s(\omega))\Lambda_s(\omega,\d u)\d s\right|\\
				=&:J_1+J_2.
		\end{align*}}
From (A5) of \assref{ass1}, it follows that $|J_1|\leq MT\|X'-Y'\|_{\infty}^2$. Let us set
\begin{align*}
\iota(t,\omega,u'):=\int_Ub_3\left(t,X_t(\omega),\mu_t,u,Y_t',u'\right)\nabla_y\phi(Y_t(\omega),W_t(\omega))\Lambda_t(\omega,\d u).    
\end{align*}
Then, by (A5) of \assref{ass1} again,  $|\iota(\omega,u')-\iota(\omega,v')|\leq M'|u'-v'|^2$ for all $\omega\in\Omega$ and $u',v'\in U$ with $M'$ depending only on $(\phi,M)$ in \assref{ass1}. Thanks to Kantorovich duality, for all $(t,\omega)\in [0,T]\times\Omega$, it holds that
\begin{align*}
J_2&=\left|\int_0^t\int_U\iota(s,\omega,u')\Lambda_s'(\d u')\d s-\int_0^t\int_U\iota(s,\omega,u')\Upsilon'_s(\d u')\d s\right|\\
&\leq M'\sup_{h:|h(u')-h(v')|\leq |u'-v'|^2}\left(\int_0^t\int_Uh(s,\omega,u')\Lambda'_s(\d u')\d s-\int_0^t\int_Uh(s,\omega,u')\Upsilon'_s(\d u')\d s\right)\\
&\leq M'd_{\mathcal{Q}}(\Lambda',\Upsilon')^2,
\end{align*}
which yields that, for some constant $C>0$ independent of $(t,\omega)$,
\begin{align*}
\sup_{(t,\omega)\in[0,T]\times\Omega}\left|\Xi(t,X',\Lambda',\omega)-\Xi(t,Y',\Upsilon',\omega)\right|\leq C\left(\|X'-Y'\|_{\infty}^2+d_{\mathcal{Q}}(\Lambda',\Upsilon')^2\right).    
\end{align*}
As a result, together with Kantorovich duality, we can derive that
\begin{align}\label{convergence_pre}
\lim_{n\to\infty}\sup_{(t,\omega)\in[0,T]\times\Omega}\left|\int_{\C\times\mathcal{Q}}\Xi(t,X',\Lambda',\omega)(\xi^n(\d X',\d\Lambda')-\xi(\d X',\d\Lambda'))\right|=0.
\end{align}
Combing \equref{uniform_convergence_pre} and \equref{convergence_pre}, we have $\lim_{n\to\infty}(|I_1|+|I_2|)=0$. Thus, it holds that
\begin{align*}
\lim_{n\to\infty}\int_0^t\int_U\bar b(s,X_s,\xi_s^n,u)^{\T}\nabla\phi(Y_s,W_s)\Lambda_s(\d u)\d s=\int_0^t\int_U\bar b(s,X_s,\xi_s,u)^{\T}\nabla\phi(Y_s,W_s)\Lambda_s(\d u)\d s  
\end{align*}
uniformly with respect to $(t,\omega)\in [0,T]\times\Omega$. In the same way, we can also show that\begin{align*}
\lim_{n\to\infty} \sup_{(t,\omega)\in [0,T]\times\Omega}&\left|\int_0^t\int_U\left[\tr\left(\frac12\bar\sigma\bar\sigma^{\T}(t,X_s,\xi_s^n,u)\nabla^2\phi(Y_s,W_s)\right)\right.\right.\\
&\left.\left.-\tr\left(\frac12\bar\sigma\bar\sigma^{\T}(t,X_s,\xi_s,u)\nabla^2\phi(Y_s,W_s)\right)\right]\Lambda_s(\d u)\d s\right|=0.
\end{align*}
Thus, we can derive that
		\begin{align}\label{uniform_convergence}
			\lim_{n\to\infty}\sup_{(t,\omega)\in[0,T]\times\Omega}\left|{\tt M}^{\boldsymbol{\xi}_n}\phi(t,\omega)-{\tt M}^{\boldsymbol{\xi}}\phi(t,\omega)\right|=0.
		\end{align}
Then, for any bounded $\F_s$-measurable $\RV$ $h$ with $s<t$, it holds that
\begin{align*}
&\varlimsup_{n\to\infty}\left|\E^{P_n}\left[{\tt M}^{\boldsymbol{\xi}_n}(t)h\right]-\E^P\left[{\tt M}^{\boldsymbol{\xi}}(t)h\right]\right|\leq \varlimsup_{n\to\infty}\E^{P_n}\left[\left|\left({\tt M}^{\boldsymbol{\xi}_n}(t)-{\tt M}^{\boldsymbol{\xi}}(t)\right)h\right|\right]\nonumber\\
&\quad+\varlimsup_{n\to\infty}\left|\E^{P_n}\left[{\tt M}^{\boldsymbol{\xi}}(t)h\right]-\E^P\left[{\tt M}^{\boldsymbol{\xi}}(t)h\right]\right|=0,
\end{align*}
where the last equality follows from the fact that $\mathcal{W}_{2,\Omega}(P_n,P)\to 0$ as $n\to\infty$ and \equref{uniform_convergence}. By virtue of the martingale property of ${\tt M}^{\boldsymbol{\xi}_n}=({\tt M}^{\boldsymbol{\xi}_n}(t))_{t\in[0,T]}$ under $P_n$, the conclusion holds.
	\end{proof}
        \begin{remark}\label{closed_Lip}
            Note that here we only require the Lipschitz condition of $(b,\sigma)$ with respect to $x\in\R$. If we further assume the Lipschitz condition with respect to $\rho\in\Pc_2(\R\times\R^l)$ (\assref{ass3}), then the separation conditions in {\rm (A1)} and {\rm (A5)} of \assref{ass1} are redundant and can be omitted, and \lemref{closed} follows directly from \lemref{CB_property}
        \end{remark}

With the aid of the closed graph theorem (c.f. Proposition 17.11 in Aliprantis and Border \cite{Aliprantis}), we can prove the continuity of the mapping $\textsf{R}$ defined before \lemref{closed}.
\begin{lemma}\label{continuity}
The set-valued mapping $\textsf{R}:\Pc_2(\C\times\mathcal{Q})\mapsto \comp(\Pc_2(\Omega))$ is continuous. 
\end{lemma}
	
\begin{proof}
The upper semi-continuity of $\textsf{R}$ follows from the closed graph theorem, and it suffices to show the lower semi-continuity of $\textsf{R}$. Let $\boldsymbol{\xi}^n\to\boldsymbol{\xi}$  in $\Pc_2(\C\times\mathcal{Q})$ as $n\to\infty$ and $P\in R(\boldsymbol{\xi})$. In light of \lemref{moment_p}, there exists a filtered probability space $(\Omega',\F',\Fb',P')$ supporting a scalar $\Fb'$-adapted process $Y=(Y_t)_{t\in[0,T]}$, a standard scalar $\Fb'$- Brownian motion $W=(W_t)_{t\in [0,T]}$, an $\F_0'$-measurable $\RV$ $\eta'$ with law $\Law^{\Pb}(\eta)$,  an $\Fb'$-martingale measure ${\cal M}$ on $U\times[0,T]$, with $\mathbb{F}'$-progressively measurable intensity $\Lambda_t(\d u)\d t$ and a real-valued $\mathbb{F}'$-adapted continuous process $A=(A_t)_{t\in[0,T]}$ such that $P=P'\circ (Y^{\Lambda,\boldsymbol{\xi}},\Lambda,W,A)^{-1}$  and (i)-(iv) of \lemref{moment_p} hold. On the other hand, \lemref{extension_ass} ensures that, for each $n\geq1$, we can strongly solve the SDE:
\begin{align*}
\d Y_t^n=\int_U\tb (t,X_t^n,\xi_t^n,u)\Lambda_t(\d u)\d t+\int_U\ts (t,X_t^n,\xi_t^n,u){\cal M}(\d u,\d t),~~Y_0^n=\eta',    
\end{align*}
where $X_t^n=\Gamma(A,Y^n)_t$ and $\xi_t^n$ denotes the $t$-marginal distribution of $\boldsymbol{\xi}^n$. The Lipschitz condition in \lemref{extension_ass} with the proof in \lemref{closed} and the BDG inequality yield that
\begin{align*}
\lim_{n\to\infty}\E^{P'}\left[\sup_{t\in [0,T]}\left|Y_t^n-Y_t\right|^2\right]= 0.    
\end{align*}
Set $P_n:=P'\circ(Y^n,\Lambda,W,A)^{-1}$. Then, we have $P_n\in R(\boldsymbol{\xi}_n)$ by construction and $P_n\to P$ in $\Pc_2(\Omega)$ as $n\to\infty$. The proof is thus complete.
\end{proof}  
	
We next follow Theorem 5.7 in Karoui et. al. \cite{Karoui1} to establish the upper semi-continuity of the mapping $\textsf{R}_{\rm opt}$ defined before \lemref{closed}:
\begin{lemma}\label{u.s.c.}
The value function $\inf_{P\in R(\boldsymbol{\xi})}\mathcal{J}(\boldsymbol{\xi},P)$ is continuous in $\boldsymbol{\xi}\in\Pc_2(\C\times\mathcal{Q})$, and moreover, the set-valued mapping $\textsf{R}_{\rm opt}:\Pc_2(\C\times\mathcal{Q})\mapsto \comp(\Pc_2(\Omega))$ is upper semi-continuous. 
\end{lemma}

\begin{proof} Following the proof of \lemref{closed}, we are able to conclude that
\begin{align}\label{f_uniform_convergence}
\lim_{n\to\infty}\sup_{(t,\omega)\in[0,t]\times\Omega}\left|\int_0^T\int_U\tf(t,X_t,\xi_t^n,u)\Lambda_t(\d u)\d t-\int_0^T\int_U\tf(t,X_t,\xi_t,u)\Lambda_t(\d u)\d t\right|=0.       
\end{align}
Let $(\boldsymbol{\xi}^n,P^n)\to(\boldsymbol{\xi},P)$ in $\Pc_2(\C\times\mathcal{Q})\times\Pc_2(\Omega)$ as $n\to\infty$. Then, a simple calculation leads to that
\begin{align*}
\left|\mathcal{J}(P^n,\boldsymbol{\xi}^n)-\mathcal{J}(P,\boldsymbol{\xi})\right|&=\left|\E^{P^n}\left[\Delta(\boldsymbol{\xi}^n)\right]-\E^P\left[\Delta(\boldsymbol{\xi})\right]\right|\\
&\leq \E^{P^n}\left[\left|\Delta(\boldsymbol{\xi}^n)-\Delta(\boldsymbol{\xi})\right|\right]+\left|\E^{P^n}\left[\Delta(\boldsymbol{\xi})\right]-\E^P\left[\Delta(\boldsymbol{\xi})\right]\right|\\
&:= I_1+I_2.
\end{align*} 
Note that $I_1$ converges to $0$ due to \equref{f_uniform_convergence} and $I_2$ converges to $0$ due to the continuity of $P\mapsto \E^{P}\left[\Delta(\boldsymbol{\xi}\right]$ for every $\boldsymbol{\xi}\in\Pc_2(\C\times\mathcal{Q})$. Therefore, we derive the joint continuity of $(\boldsymbol{\xi},P)\mapsto\mathcal{J}(P,\boldsymbol{\xi})$.
		The result easily follows from Theorem 5.7 in Karoui \cite{Karoui1} by setting $Y=\Pc_2(\C\times\mathcal{Q})$, $X=\Pc_2(\Omega)$, $K=\textsf{R}$ and $w=\mathcal{J}$.
	\end{proof}
	
Finally, we prove the desired results in  \textbf{Step (iii)}. Recall that, to apply Kakutani's fixed point theorem, we need to restrict $\mathcal{R}^*$ to a compact and convex subset $\mathcal{M}$ of $\Pc_2(\C\times\mathcal{Q})$ such that $\mathcal{R}^*|_{\mathcal{M}}$ is a self-mapping. 
\begin{lemma}\label{compact_set}
Define the subset $\mathcal{M}$ of $\Pc_2(\C\times\mathcal{Q})$ by
\begin{align*}
\mathcal{M}:=\left\{\boldsymbol{\xi}\in\Pc_2(\C\times\mathcal{Q});~\int_{\C\times\mathcal{Q}}\|\boldsymbol{x}\|_{\infty}^p~\boldsymbol{\xi}(\d\boldsymbol{x},\d q)\leq C_1 \right\}.    
\end{align*}
with $C_1$ being the constant given by \equref{eq:C1Lemma36} in \lemref{moment_p}. Then, $\mathcal{M}$ is a compact and convex subset of $\Pc_2(\C\times\mathcal{Q})$. Moreover, $\mathcal{R}^*|_{\mathcal{M}}$ is a self-mapping (recall that $\mathcal{R}^*$ is defined in \equref{eq:Kstar}).
	\end{lemma}
	
\begin{proof}
We first show that $\mathcal{M}$ is closed. For $\boldsymbol{\xi}_n\in \mathcal{M}$ such that $\boldsymbol{\xi}_n\to \boldsymbol{\xi}$ in $\Pc_2(\C\times\mathcal{Q})$ as $n\to\infty$, by Fatou's lemma and the application of Skorokhod representation theorem if necessary, we can conclude that $\int_{\C\times\mathcal{Q}}\|\boldsymbol{x}\|_{\infty}^p~\boldsymbol{\xi}(\d\boldsymbol{x},\d q)\leq C_1$. This implies that $\boldsymbol{\xi}\in \mathcal{M}$. Moreover, we can verify the precompactness of $\mathcal{M}$ by following the proof of \lemref{compactness}. Hence, $\mathcal{M}$ is compact and is obviously convex. Lastly, by \lemref{moment_p} again, we derive that $\mathcal{R}^*$ is a self-mapping.
\end{proof}
	
Given all previous preparations, we can finally prove the main result--\thmref{existence_RMFE}.
\begin{proof}[Proof of \thmref{existence_RMFE}]
It follows from \lemref{compactness} and \lemref{compact_set} that,  $\mathcal{R}^*|_{\mathcal{M}}:\mathcal{M}\to\mathcal{M}$ is compact and convex. By using the continuity of the push-forward mapping 
\begin{align*}
S:\Pc_2(\Omega)\to \Pc_2(\C\times\mathcal{Q}),~P\mapsto P\circ (X,\Lambda)^{-1},    
\end{align*}
we have $\mathcal{R}^*=S\circ\textsf{R}_{\rm opt}$ is also upper semi-continuous. Consequently, by using Kakutani's fixed point theorem, there exists a fixed point $\boldsymbol{\xi}^*\in\mathcal{M}$ of the mapping $\mathcal{R}^*$ and $P^*\in R_{\rm opt}(\boldsymbol{\xi}^*)$ such that $\boldsymbol{\xi}^*= S(P^*)$. According to \defref{relaxed_MFE}, $(\boldsymbol{\xi}^*,P^*)$ is an (R)-MFE, which ends the proof.
\end{proof}
	
	\subsubsection{Proof of \thmref{propagation}}\label{proof:thm2}
	This subsection is dedicated to proving \thmref{propagation} by following the three-step procedure. 
	
	For \textbf{Step (i)}, we first borrow a precompactness criterion from Proposition A.2 in Carmona et. al. \cite{Carmona}, which is given as below.
	
	\begin{lemma}\label{precompactness_criterion}
		Let $(E,\ell)$ be a complete separable metric space. Assume that there exists a subset $K\subset\Pc_p(\Pc_p(E))$ such that $\{{\tt m}P\}_{P\in K}\subset\Pc(E)$ is tight, and for some $p'>p$ and  $x_0\in E$,
		\begin{align}\label{Lp'_boundedness}
			\sup_{P\in K}\int_E{\tt m}P(\d x)\ell(x,x_0)^{p'}<\infty,
		\end{align}
		where ${\tt m}P$ is the mean measure defined by ${\tt m}P(C):=\int_{\Pc_p(E)}\mu(C)P(\d\mu)$ for $C\in\mathcal{B}(E)$ and $P\in\Pc_p(\Pc_p(E))$. Then, the subset $K$ is precompact in $\Pc_p(\Pc_p(E))$.
	\end{lemma}
	
	Then, we have
	\begin{lemma}\label{precompactness}
		It holds that  $(P^N\circ (\boldsymbol{\delta}^N)^{-1})_{N\geq1}$ is precompact in $\Pc_2(\Pc_2(\Omega))$.
	\end{lemma}
	
	\begin{proof}
		Set $\Upsilon^N:=P^N\circ (\delta^N)^{-1}$ for $N\geq 1$. In light of \lemref{precompactness_criterion}, we first show that $({\tt m}\Upsilon^N)_{N\geq1}$ is tight. By construction,  we have ${\tt m}\Upsilon^N=\frac1N\sum_{i=1}^N\Law^{P^N}(Y^i,\Lambda^i,W^i,A^i)$. To show the tightness, we need to show that $({\tt m}\Upsilon^N\circ Y^{-1})_{N\geq1}=(\frac1N\sum_{i=1}^N\Law^{P^N}(Y^i))_{N\geq1}$ is tight due to the compactness of $\mathcal{Q}$ and the fact $\frac1N\sum_{i=1}^N\Law^{P^N}(W^i,A^i)=\Law^{\Pb}(W,A)$ (c.f. \defref{relaxed_strategy}).
		In view of \lemref{moment_p_N}, there exists a constant $C>0$ independent of $N$ such that
		\begin{align*}
			\E^{{\tt m}\Upsilon^N}\left[\sup_{t\in [0,T]}\left|Y_t\right|^p\right]=\frac1N\sum_{i=1}^N\E^{P^N}\left[\sup_{t\in [0,T]}\left|Y_t^i\right|^p\right]<+\infty,    
		\end{align*}
		which verifies \equref{Lp'_boundedness}. As a consequence of the above $L^{P'}$-boundedness, the set $({\tt m}\Upsilon^N\circ Y^{-1})_{N\geq1}$ is tight, and is thus precompact in $\Pc_2(\Pc_2(\Omega))$ thanks to \lemref{precompactness_criterion}.
	\end{proof}

	We next proceed to show the results in 
	\textbf{Step (ii)}.	Define the following mapping that
	\begin{align}\label{eq:mappingF}
\mathscr{T}:\Pc_2(\Omega)\to\Pc_2(\C\times\mathcal{Q})\times\Pc_2(\Omega),~P\mapsto (P\circ (X,\Lambda)^{-1},P), 
	\end{align}
	where we recall that $X=\Gamma(A,Y)$ and $(Y,\Lambda,W,A)$ are the coordinate processes on $\Omega$ introduced in \defref{relaxed_control}. Then, we have $P^N\circ (\boldsymbol{\xi}^N,\boldsymbol{\delta}^N)^{-1}=\left(P^N\circ(\boldsymbol{\delta}^N)^{-1}\right)\circ\mathscr{T}^{-1}$. As a result, $(P^N\circ (\boldsymbol{\xi}^N,\boldsymbol{\delta}^N)^{-1})_{N\geq1}$ is precompact in $\Pc_2(\Pc_2(\C\times\mathcal{Q})\times\Pc_2(\Omega))$. Moreover, the push-forward mapping of $\mathscr{T}$ establishes a one-to-one correspondence between the set of limit points of the sequence $(P^N\circ(\boldsymbol{\delta}^N)^{-1})_{N\geq1}$ and the set of limit points of the sequence $(P^N\circ(\boldsymbol{\xi}^N,\boldsymbol{\delta}^N)^{-1})_{N\geq1}$ due to the continuity of $\mathscr{T}$. More precisely, if $P^{k_N}\circ(\boldsymbol{\delta}^{k_N})$ converges to a limit point $\Upsilon\in\Pc_2(\Pc_2(\Omega))$ as $N\to\infty$, then $P^{k_N}\circ(\boldsymbol{\xi}^{k_N},\boldsymbol{\delta}^{k_N})^{-1}$ converges to $\Upsilon\circ\mathscr{T}^{-1}$ in $\Pc_2(\Pc_2(\C\times\mathcal{Q})\times\Pc_2(\Omega))$ as $N\to\infty$, and the converse follows immediately by simply restricting $\Pc_2(\Pc_2({\cal C}\times\mathcal{Q})\times\Pc_2(\Omega))$ to $\Pc_2(\Pc_2(\Omega))$.
	
	We next show the proof in {\bf Step (ii)}. To start with, we fix a limit point $\Xi\in\Pc_2(\Pc_2(\C\times\mathcal{Q})\times\Pc_2(\Omega))$ of the sequence $(\Xi^N=P^N\circ (\boldsymbol{\xi}^N,\boldsymbol{\delta}^N)^{-1})_{N\geq1}$ in $\Pc_2(\Pc_2(\C\times\mathcal{Q})\times\Pc_2(\Omega))$, and we may assume without loss of generality that $\Xi^N$ converges to $\Xi$ in $\Pc_2(\Pc_2(\C\times\mathcal{Q})\times\Pc_2(\Omega))$ as $N\to\infty$, relabeling if necessary.

	\begin{lemma}\label{pre_MFE_support}
		The limit point	$\Xi$ is supported on the set given by
		\begin{align}\label{pre_MFE}
			\mathscr{M}:=\{(\boldsymbol{\xi},P)\in\Pc_2(\C\times\mathcal{Q})\times\Pc_2(\Omega);~P\in R(\boldsymbol{\xi}),~P\circ(X,\Lambda)^{-1}=\boldsymbol{\xi}\}.
		\end{align}
	\end{lemma}

\begin{proof}
Let us define the following subsets of $\Pc_2(\C\times\mathcal{Q})\times\Pc_2(\Omega)$ that
\begin{align*}
\mathscr{M}_1&:=\left\{(\boldsymbol{\xi},P)\in\Pc_2(\C\times\mathcal{Q})\times\Pc_2(\Omega);~P\circ (X,\Lambda)^{-1}=\boldsymbol{\xi}\right\},\\
\mathscr{M}_2&:=\left\{(\boldsymbol{\xi},P)\in\Pc_2(\C\times\mathcal{Q})\times\Pc_2(\Omega);~P\circ Y_0^{-1}=\Law^{\Pb}(\eta)\right\},\\
\mathscr{M}_3&=\left\{(\boldsymbol{\xi},P)\in\Pc_2(\C\times\mathcal{Q})\times\Pc_2(\Omega);~P(Y_0\geq A_0)=1\right\},\\
\mathscr{M}_4&:=\left\{(\boldsymbol{\xi},P)\in\Pc_2(\C\times\mathcal{Q})\times\Pc_2(\Omega);~P\circ (W,A)^{-1}=\hat P\right\},\\
\mathscr{M}_5&:=\left\{(\boldsymbol{\xi},P)\in\Pc_2(\C\times\mathcal{Q})\times\Pc_2(\Omega);~(\mathtt{M}^{\boldsymbol{\xi}}\phi(t))_{t\in[0,T]}~\text{is a $(P,\Fb)$-martingale}~\text{for all}~\phi\in C_b^2(\R\times\R)\right\},
\end{align*}
where we recall that $\mathtt{M}^{\boldsymbol{\xi}}\phi(t)$ is given in \defref{relaxed_control}. Then, it follows from 
\defref{relaxed_control} that $\mathscr{M}=\bigcap_{i=1}^5\mathscr{M}_i$. Therefore, to prove the desired result, we only need to show $\Xi(\mathscr{M}_i)=1$ for all $i=1,\ldots,5$. Given the discussion below \lemref{precompactness}, it is evident that $\Xi(\mathscr{M}_1)=1$. We define the (continuous) restriction mapping $\mathscr{N}$ by \begin{align*}
\mathscr{N}:\Pc_2(\C\times\mathcal{Q})\times\Pc_2(\Omega)\to\Pc_2(\C^W\times\tilde{\Omega}),\quad (\boldsymbol{\xi},P)\mapsto P|_{C^W\times\tilde \Omega}=P\circ (W,A)^{-1}.
\end{align*}
Then, establishing $\Xi(\mathscr{M}_4)=1$ is equivalent to showing $\Xi\circ\mathscr{N}^{-1}=\delta_{\hat P}$. Observe that $(W^i,A^i)$ for $i=1,\ldots,N$ are i.i.d. under $P^N$ and
		\begin{align*}
			\Xi^N\circ \mathscr{N}^{-1}(\cdot)=P^N\left(\frac1N\sum_{i=1}^N\delta_{(W^i,A^i)}\in\cdot\right).    
		\end{align*}
		From the law of large numbers, it follows that $\Xi^N\circ\mathscr{N}^{-1}\to\delta_{\hat P}$ in $\Pc_2(\Pc_2(\C^W\times\tilde\Omega))$ as $N\to\infty$. On the other hand, we have $\Xi^N\to \Xi$ in $\Pc_2(\Pc_2(\C\times\mathcal{Q})\times\Pc_2(\Omega))$ as $N\to\infty$, which implies $\Xi^N\circ\mathscr{N}^{-1}\to\Xi\circ\mathscr{N}^{-1}$  in $\Pc_2(\Pc_2(\tilde\Omega))$ as $N\to\infty$. Consequently, we obtain $\Xi\circ \mathscr{N}^{-1}=\delta_{\hat P}$, and hence $\Xi(\mathscr{M}_4)=1$. By a similar argument, we can also conclude that $\Xi(\mathscr{M}_2)=1$. Note that, it holds that
		\begin{align*}
			\Xi^N(\mathscr{M}_3)=P^N\left(\boldsymbol{\delta}^N(Y_0\geq A_0)=1\right)=P^N\left(\bigcap_{i=1}^N\{Y^i_0\geq A^i_0\}\right)=1,
		\end{align*}
		where we have used the fact that $P^N(Y_0^i\geq A_0^i)=1$ for $i=1,\ldots,N$. By the closedness of $\mathscr{M}_3$, it follows from the Portmaneau theorem that $\Xi(\mathscr{M}_3)\geq\limsup_{N\to\infty}\Xi^N(\mathscr{M}_3)=1$. For the set $\mathscr{M}_5$, we first show that, for any $\phi\in C_b^2(\R\times\R)$, $0\leq s<t\leq T$ and $\F_s$-measurable bounded $\RV$ $h$, \begin{align}\label{pre_martingale}
			\Xi\left(\left\{(\boldsymbol{\xi},P);~\E^P\left[(\mathtt{M}^{\boldsymbol{\xi}}\phi(t)-\mathtt{M}^{\boldsymbol{\xi}}\phi(s))h\right]=0\right\}\right)=1.
		\end{align}
We then have
{\small\begin{align*}
&\E^{\Xi^N}\left[\left|\E^P\left[\left(\mathtt{M}^{\boldsymbol{\xi}}\phi(t)-\mathtt{M}^{\boldsymbol{\xi}}\phi(s)\right)h\right]\right|^2\right]\nonumber\\
=&\E^{P^N}\left[\left|\frac1N\sum_{i=1}^N\left[\left(\phi(Y_t^i,W_t^i)-\phi(Y_s^i,W_s^i)-\int_s^t\int_U\bar{\mathbb{L}}\phi(r,X_r^i,Y_r^i,W_r^i,\xi_r^N,u)\Lambda_r^i(\d u)\d r\right)h(Y^i,\Lambda^i,A^i)\right]\right|^2\right]\\
=&\frac{1}{N^2}\E^{P^N}\left[\sum_{i=1}^N\int_s^t\int_U\left|\nabla\phi(Y_r^i,W_r^i)^{\T}\bar\sigma(r,X_r^i,\xi_r^N,u)\right|^2\Lambda_r^i(\d u)\d r\left|h(Y^i,\Lambda^i,A^i)\right|^2\right],
\end{align*}}which tends to $0$ as $N\to\infty$ in view of \lemref{moment_p_N} and growth condition on $\sigma$ imposed in \assref{ass1}. Thus, by the at most quadratic growth condition of $\left|\E^P\left[\left(\mathtt{M}^{\boldsymbol{\xi}}\phi(t)-\mathtt{M}^{\boldsymbol{\xi}}\phi(s)\right)h\right]\right|^2$ w.r.t. $(\boldsymbol{\xi},P)$ and the convergence $\Xi^N\to\Xi$ in $\Pc_2(\Pc_2(\C\times\mathcal{Q})\times\Pc_2(\Omega))$ as $N\to\infty$, one has
		\begin{align*}
			\E^{\Xi}\left[\left|\E^P\left[\left(\mathtt{M}^{\boldsymbol{\xi}}\phi(t)-\mathtt{M}^{\boldsymbol{\xi}}\phi(s)\right)h\right]\right|^2\right]=0.   
		\end{align*}
		This yields the desired \equref{pre_martingale}. Applying the above argument to a suitably dense countable set of $(s,t,,\phi,h)$, we upgrade \equref{pre_martingale} to $\Xi(\mathscr{M}_5)=1$, which completes the proof.
	\end{proof}
	
	Finally, we turn to the challenging part of the proof, i.e., the task in {\bf Step (iii)}, to establish the optimality. To this end, we first extend the canonical space introduced in Subsection \ref{MFE} by incorporating the canonical space $\Pc_2(\C\times\mathcal{Q})$ of $\boldsymbol{\xi}$.
	Equip $\Omega^1=\Pc_2(\C\times\mathcal{Q})$ with the $2$-Wasserstein metric $d_{\Omega^1}=\mathcal{W}_{2,\C\times\mathcal{Q}}$ and let $\F^1$ be the corresponding Borel $\sigma$-algebra. We also define the filtration $\Fb^1=(\F_t^1)_{t\in [0,T]}$ on $(\Omega^1,\F^1)$ by setting $\F_t^1=\sigma(\xi_s;~s\leq t)$ with $\xi_s$ being the $s$-marginal distribution of $\boldsymbol{\xi}\in\Pc_2(\C\times\mathcal{Q})$.
	
	Now, we set the product space $\bar{\Omega}=\Omega\times\Omega^1$ endowed with the product metric:
	\begin{align*}
		d_{\bar\Omega}(\bar{\omega}^1,\bar{\omega}^2)=d_{\Omega}(\omega^1,\omega^2)+d_{\Omega^1}(\boldsymbol{\xi}^1,\boldsymbol{\xi}^2),\quad\bar{\omega}^i=(\omega^i,\boldsymbol{\xi}^i)\in\bar{\Omega},~i=1,2.    
	\end{align*}
	Define the corresponding product $\sigma$-algebra as $\bar{\F}=\F\otimes\F^1$ and the product filtration as $\bar{\Fb}=(\bar{\F}_t)_{t\in [0,T]}$ with $\bar{\F}_t=\F_t\otimes\F_t^1$, respectively.
	
We next extend the definition of admissible relaxed control in \defref{relaxed_control}. More precisely, the coordinate mapping on $\bar{\Omega}$ is denoted by $((\bar{Y},\bar{\Lambda},\bar{W},\bar{A}),\bar{\boldsymbol{\xi}})$. This means that $\bar{Y}(\bar{\omega})=y$, $ \bar{\Lambda}(\bar{\omega})=q$, $\bar{W}(\bar{\omega})=w$, $\bar{A}(\bar{\omega})=\tilde{\omega}$ and $\bar{\boldsymbol{\xi}}(\bar{\omega})=\boldsymbol{\xi}$ for any $\bar{\omega}=(\omega,\boldsymbol{\xi})=((y,q,w,\tilde{\omega}),\boldsymbol{\xi})\in\bar{\Omega}$.
\begin{definition}\label{extended_relaxed_control}
Let $\Theta\in\Pc_2(\Omega^1)$. We call a probability measure $\bar{P}\in\Pc_2(\bar{\Omega})$ on $(\bar{\Omega},\bar{\F})$ an extended admissible relaxed control (denoted by $\bar{P}\in R_{\rm e}(\Theta)$) if it holds that {\rm(i)} $\bar{P}\circ\bar{Y}_0^{-1}=\Law^{\Pb}(\eta)$; {\rm(ii)} $\bar{P}(\bar{Y}_0\geq\bar{A}_0)=1$; {\rm(iii)}  $\bar{P}\circ(\bar{W},\bar{A})^{-1}=\hat{P}$; {\rm(iv)} $\bar{P}\circ\bar{\boldsymbol{\xi}}^{-1}=\Theta$; {\rm(v)} 
for all $\phi\in C_b^2(\R\times\R)$, the process
\begin{align*}
\bar{\mathtt{M}}^{\Theta}\phi(t):=\phi(\bar{Y}_t,\bar{W}_t)-\int_0^t\int_U\bar{\mathbb{L}}\phi(s,\bar{X}_s,\bar{Y}_s,\bar{W}_s,\bar{\xi}_s,u)\bar{\Lambda}_s(\d u)\d s,~~t\in[0,T]   
\end{align*}
is a $(\bar{P},\bar{\Fb})$-martingale. Here, the generator $\bar{\mathbb{L}}$ is defined in \defref{relaxed_control}, $\bar{\xi}_s$ denotes the $s$-marginal distribution of $\bar{\boldsymbol{\xi}}$ and $\bar{X}(\bar{\omega})=\Gamma(\bar{A}(\bar{\omega}),\bar{Y}(\bar{\omega}))$ for $\bar{\omega}\in\bar{\Omega}$. Furthermore, if there exists an $\bar{\Fb}$-progressively measurable $U$-valued process $\bar{\alpha}=(\bar{\alpha}_t)_{t\in [0,T]}$ on $\bar{\Omega}$ such that $\bar{P}(\bar{\Lambda}_t(\d u)\d t=\delta_{\bar{\alpha}_t}(\d u)\d t)=1$, we say that $\bar{P}$ corresponds to an extended strict control ($\bar{\alpha}$) or we call it an extended strict control rule. Correspondingly, the set of all extended strict controls is denoted by $R_{\rm e}^{\rm s}(\boldsymbol{\xi})$.
\end{definition}
	
We also have the following equivalent characterization of the set $R_{\rm e}(\Theta)$ for extended relaxed control rules. The proof of the following lemma is also standard and omitted.
\begin{lemma}\label{extended_moment_p}
Let	$\Theta\in\Pc_2(\Omega^1)$. Then $\bar P\in\mathcal{R}_{\rm e}(\Theta)$ iff there exists a filtered probability space $(\Omega',\F',\Fb'=(\F'_t)_{t\in[0,T]},P')$ supporting a $\Pc(U)$-valued $\Fb'$-progressively measurable process $\Lambda=(\Lambda_t)_{t\in [0,T]}$, an $\Omega^1$-valued $\RV$ $\boldsymbol{\xi}'$ with $P'\circ (\boldsymbol{\xi}')^{-1}=\Theta$, a scalar $\Fb'$-adapted process $Y^{\Lambda,\boldsymbol{\xi}'}=(Y_t^{\Lambda,\boldsymbol{\xi}'})_{t\in[0,T]}$, a standard scalar $\Fb'$-Brownian motion $W=(W_t)_{t\in [0,T]}$, an $\Fb'$-martingale measure ${\cal M}$ on $U\times[0,T]$ with intensity $\Lambda_t(\d u)\d t$ and a real-valued $\mathbb{F}'$-adapted continuous process $A=(A_t)_{t\in[0,T]}$ such that $\bar P=P'\circ (Y^{\Lambda,\boldsymbol{\xi}'},\Lambda,W,A,\boldsymbol{\xi}')^{-1}$, and it holds that {\rm(i)} $P'\circ (Y_0^{\Lambda',\boldsymbol{\xi}})^{-1}=\Law^{\mathbb{P}}(\eta)$; {\rm(ii)} $P'(Y_0^{\Lambda,\boldsymbol{\xi}'}\geq A_0)=1$; {\rm(iii)} the restriction of $P$ to $\C^W\times\tilde\Omega$, $P|_{\C^W\times\tilde\Omega}$, agrees with $\Pb\circ (W,A)^{-1}$, $\ie$ $P\circ (W,A)^{-1}=\Pb\circ (W,A)^{-1}=\hat P$ ; {\rm(iv)} the dynamics of state process obeys that, $P'$-a.s.
\begin{align*}
\d Y_t^{\Lambda,\boldsymbol{\xi}'}=\int_U\tb(t,X_t^{\Lambda,\boldsymbol{\xi}'},\xi'_t,u)\Lambda_t(\d u)\d t+\int_U\ts(t,X_t^{\Lambda,\boldsymbol{\xi}'},\xi'_t,u){\cal M}(\d u,\d t),~ X_t^{\Lambda,\boldsymbol{\xi}'}=\Gamma(A,Y^{\Lambda,\boldsymbol{\xi}'})_t,
\end{align*}
where $\xi_t'$ denotes the $t$-marginal distribution of $\boldsymbol{\xi}'$. Moreover, there is a constant $C>0$ depending on $(M,\Law^{\mathbb{P}}(\eta),T)$ such that $\E^{P'}[\sup_{t\in[0,T]}|Y_t^{\Lambda,\boldsymbol{\xi}'}|^p]\leq C$ with $M$ being stated in \assref{ass1}.
\end{lemma}
	
For any $\bar{P}\in \Pc_2(\bar{\Omega})$ such that $\bar{P}|_{\Omega^1}=\Theta$,  it can be expressed in the following disintegration form:
\begin{align}\label{disintegration}
\bar{P}(\d\omega,\d\boldsymbol{\xi})=P^{\boldsymbol{\xi}}(\d\omega)\Theta(\d\boldsymbol{\xi}),
\end{align}
where $P^{\boldsymbol{\xi}}\in\Pc_2(\Omega)$ is measurable with respect to $\boldsymbol{\xi}$. 
\begin{lemma}\label{relation}
Let $\bar{P}\in\Pc_2(\bar{\Omega})$ admit the disintegration form \equref{disintegration}. If for $\Theta$-a.s. $\boldsymbol{\xi}\in\Omega^1$, the probability kernel $P^{\boldsymbol{\xi}}\in{\cal P}_2(\Omega)$ given in \equref{disintegration} belongs to $R(\boldsymbol{\xi})$, and hence $\bar{P}\in R_{\rm e}(\Theta)$.
\end{lemma}
	
\begin{proof}
By construction, one can easily verify that (i)-(iv) of \defref{extended_relaxed_control} holds for $\bar P$. It is left to prove the martingale property for $\bar{P}$. Note that, for any $\phi\in C_b^2(\R\times\R)$, $0\leq s<t\leq T$ and $\bar{\F}_s$-measurable bounded $\RV$ $\bar{h}$, it follows from \defref{extended_relaxed_control}-(v) that
\begin{align}\label{equivalence_martingale}
0&=\int_{\Omega^1}\E^{P^{\boldsymbol{\xi}}}\left[\left(\mathtt{M}^{\boldsymbol{\xi}}\phi(t)-\mathtt{M}^{\boldsymbol{\xi}}\phi(s))\bar{h}(\cdot,\boldsymbol{\xi})\right)\right]\Theta(\d\boldsymbol{\xi})\nonumber\\
&=\int_{\Omega^1}\left(\int_{\Omega}\Big(\phi(Y_t(\omega),W_t(\omega))-\phi(Y_s(\omega),W_s(\omega))\right.\nonumber\\
&\quad-\left.\int_s^t\int_U\bar{\mathbb{L}}\phi(r,X_r(\omega),Y_r(\omega),W_r(\omega),\boldsymbol{\xi}_r,u)\Lambda_r(\omega,\d u)\d r\Big)\bar{h}(\omega,\boldsymbol{\xi})P^{\boldsymbol{\xi}}(\d\omega)\right)\Theta(\d\boldsymbol{\xi})\nonumber\\
&=\int_{\Omega^1}\int_{\Omega}\Big(\phi(\bar{Y}_t(\bar{\omega}),\bar{W}_t(\bar\omega))-\phi(\bar{Y}_s(\bar{\omega}),\bar{Y}_s(\bar\omega))\nonumber\\
&\quad-\int_s^t\int_U\bar{\mathbb{L}}\phi(r,\bar{X}_r(\bar{\omega}),\bar{Y}_r(\bar\omega),\bar{W}_r(\bar\omega),\bar{\boldsymbol{\xi}}_r(\bar{\omega}),u)\bar{\Lambda}_r(\bar{\omega},\d u)\d r\Big)\bar{h}(\bar{\omega})P^{\boldsymbol{\xi}}(\d\omega)\Theta(\d\boldsymbol{\xi})\nonumber\\
&=\E^{\bar{P}}\left[\left(\bar{\mathtt{M}}^{\Theta}\phi(t)-\bar{\mathtt{M}}^{\Theta}\phi(s)\right)h\right],
\end{align}
where in the first equality, we have exploited the fact that $\bar{h}(\cdot,\boldsymbol{\xi})$ is $\F_s$-measurable for every $\boldsymbol{\xi}\in\Omega^1$ and that $\mathtt{M}^{\boldsymbol{\xi}}\phi$ is a $(P^{\boldsymbol{\xi}},\Fb)$-martingale for $\Theta$-$\as$ $\boldsymbol{\xi}$. Therefore, we can conclude that $\bar{P}\in R_{\rm e}(\Theta)$ after validating (i)-(v) of \defref{extended_relaxed_control}. 
\end{proof}
	
For $\Theta\in\Pc_2(\Omega^1)$ and $\bar{P}\in R_{\rm e}(\Theta)$, we also define the cost functional $\bar{\mathcal{J}}(\bar{P},\Theta)$ by
\begin{align}\label{extended_cost}
\bar{\mathcal{J}}(\bar{P},\Theta):=\E^{\bar{P}}\left[\Delta(\bar{\boldsymbol{\xi}})\right]
\end{align}
with $\Delta(\cdot)$ being defined in \equref{costDel}.
	The following lemma reveals the relationship between the cost functionals $\mathcal{J}(P,\boldsymbol{\xi})$ and $\bar{\mathcal{J}}(\bar{P},\Theta)$.
	\begin{lemma}\label{relation_value}
		For $\Theta\in\Pc_2(\Omega^1)$, it holds that
		\begin{align*}
			\inf_{\bar{P}\in R_{\rm e}(\Theta)}\bar{\mathcal{J}}(\bar{P},\Theta)\leq\int_{\Omega^1}\inf_{P\in R(\boldsymbol{\xi})}\mathcal{J}(P,\boldsymbol{\xi})\Theta(\d\boldsymbol{\xi}).    
		\end{align*}
	\end{lemma}
	
\begin{proof}
From \lemref{u.s.c.}, Proposition 17.11 in Aliprantis and Border \cite{Aliprantis} and Lemma 12.1.8 in Stroock and Varadhan \cite{Stroock}, it follows that the set-valued mapping $\mathsf{R}_{\rm opt}$ is Borel measurable. In view of Theorem 12.1.10 of \cite{Stroock}, there exists a measurable selection $\mathscr{J}:\Pc_2(\C\times\mathcal{Q})\to\Pc_2(\Omega)$ such that $\mathscr{J}(\boldsymbol{\xi})\in R_{\rm opt}(\boldsymbol{\xi})$ for all $\boldsymbol{\xi}\in\Pc_2(\C\times\mathcal{Q})$. Hence, we can define the kernel $P^{\boldsymbol{\xi}}$ by $P^{\boldsymbol{\xi}}=\mathscr{J}(\boldsymbol{\xi})$ for every $\boldsymbol{\xi}\in\Omega^1$, and then construct $\bar{P}$ by $\bar{P}(\d\omega,\d\boldsymbol{\xi})=P^{\boldsymbol{\xi}}(\d\omega)\Theta(\d\boldsymbol{\xi})$. Thus, in light of \lemref{relation}, we have $\bar{P}\in R_{\rm e}(\Theta)$. As a result, it holds that
\begin{align*}
\inf_{\bar{P}\in R_{\rm e}(\Theta)}\bar{\mathcal{J}}(\bar{P},\Theta)\leq\bar{\mathcal{J}}(\bar{P},\Theta)=\int_{\Omega^1}\mathcal{J}(P^{\boldsymbol{\xi}},\boldsymbol{\xi})\Theta(\d\boldsymbol{\xi})=\int_{\Omega^1}\inf_{P\in R(\boldsymbol{\xi})}\mathcal{J}(P,\boldsymbol{\xi})\Theta(\d\boldsymbol{\xi}),
\end{align*}
where we  used the fact that $P^{\boldsymbol{\xi}}=\mathscr{J}(\boldsymbol{\xi})\in R_{\rm opt}(\boldsymbol{\xi})$. Thus, the proof is complete.
\end{proof}
	
Now, we are at the position to give a proof of \thmref{propagation}.
\begin{proof}[Proof of \thmref{propagation}]\label{proof_final}
We are left to show the optimality, i.e.,
\begin{align}\label{optimality}
\Xi\left(\left\{(\boldsymbol{\xi},P)\in{\cal P}_2(\C\times{\cal Q})\times{\cal P}_2(\Omega);~\mathcal{J}(P,\boldsymbol{\xi})=\inf_{Q\in R(\boldsymbol{\xi})}\mathcal{J}(Q,\boldsymbol{\xi})\right\}\right)=1.
\end{align}
We will prove \equref{optimality} in the following equivalent form given by
\begin{align}\label{optimality1}
\E^{\Xi}\left[\mathcal{J}(P,\boldsymbol{\xi})-\inf_{Q\in R(\boldsymbol{\xi})}\mathcal{J}(Q,\boldsymbol{\xi})\right]=0.
\end{align}
Note that the equivalence follows from the fact that $\mathcal{J}(P,\boldsymbol{\xi})-\inf_{Q\in R(\boldsymbol{\xi})}\mathcal{J}(Q,\boldsymbol{\xi})\geq 0$. To proceed with, set $\Theta^N:=\Xi^N\circ\boldsymbol{\xi}^{-1}=P^N\circ(\boldsymbol{\xi}^N)^{-1}$. Then, we have from a similar chattering lemma argument in \cite{Djete}  and \lemref{relation_value} that
\begin{align}
\E^{\Xi^N}\left[\mathcal{J}(P,\boldsymbol{\xi})-\inf_{Q\in R(\boldsymbol{\xi})}\mathcal{J}(Q,\boldsymbol{\xi})\right]\leq \E^{\Xi^N}\left[\mathcal{J}(P,\boldsymbol{\xi})\right]-\inf_{\bar{Q}\in R_{\rm e}^{\rm s}(\Theta^N)}\bar{\mathcal{J}}(\bar{Q},\Theta^N).
\end{align}
Due to the at most quadratic growth of $\mathcal{J}(P,\boldsymbol{\xi})-\inf_{Q\in R(\boldsymbol{\xi})}\mathcal{J}(Q,\boldsymbol{\xi})$ in $(\boldsymbol{\xi},P)$, we only need to show that
\begin{align}\label{goal}
\liminf_{N\to\infty}\left\{\E^{\Xi^N}\left[\mathcal{J}(P,\boldsymbol{\xi})\right]-\inf_{\bar{Q}\in R_{\rm e}^{\rm s}(\Theta^N)}\bar{\mathcal{J}}(\bar{Q},\Theta^N)\right\}\leq 0.
\end{align}
In fact, if \equref{goal} holds, we can derive that
\begin{align*}
\E^{\Xi}\left[\mathcal{J}(P,\boldsymbol{\xi})-\inf_{Q\in R(\boldsymbol{\xi})}\mathcal{J}(Q,\boldsymbol{\xi})\right]&=\lim_{N\to\infty}	\E^{\Xi^N}\left[\mathcal{J}(P,\boldsymbol{\xi})-\inf_{Q\in R(\boldsymbol{\xi})}\mathcal{J}(Q,\boldsymbol{\xi})\right]\\
&\leq 	\liminf_{N\to\infty}\left\{\E^{\Xi^N}\left[\mathcal{J}(P,\boldsymbol{\xi})\right]-\inf_{\bar{Q}\in R_e^{\rm s}(\Theta^N)}\bar{\mathcal{J}}(\bar{Q},\Theta^N)\right\}\leq 0,
\end{align*}
which implies \equref{optimality1}. We next turn to the proof of the inequality \equref{goal}. By definition of $\Xi^{N}$ and $\Theta^N$, it holds that
\begin{align*}
\E^{\Xi^N}\left[\mathcal{J}(P,\boldsymbol{\xi})\right]-\inf_{\bar{Q}\in R_e^{\rm s}(\Theta^N)}\bar{\mathcal{J}}(\bar{Q},\Theta^N)=\E^{P^N}\left[\mathcal{J}(\boldsymbol{\delta}^N,\boldsymbol{\xi}^N)\right]-\inf_{\bar{Q}\in R_e^{\rm s}(\Theta^N)}\bar{\mathcal{J}}(\bar{Q},\Theta^N).
\end{align*}
Take $\bar{P}^N\in R_{\rm e}^{\rm s}(\Theta^N)$ such that
\begin{align*}
\E^{P^N}\left[\mathcal{J}(\boldsymbol{\delta}^N,\boldsymbol{\xi}^N)\right]-\inf_{\bar{Q}\in R_{\rm e}^{\rm s}(\Theta^N)}\bar{\mathcal{J}}(\bar{Q},\Theta^N)<\E^{P^N}\left[\mathcal{J}(\boldsymbol{\delta}^N,\boldsymbol{\xi}^N)\right]-\bar{\mathcal{J}}(\bar{P}^N,\Theta^N)+\frac1N.    
\end{align*}
Then, by \lemref{moment_p_N} and \lemref{extended_moment_p}, there exists  a filtered probability space $(\Omega',\F',\Fb',P')$ supporting $N+1$ $U$-valued $\Fb'$-progressively measurable processes $\alpha^j=(\alpha_t^j)_{t\in[0,T]}$ for $j=1,\ldots,N$, $\beta=(\beta_t)_{t\in[0,T]}$, $N$ i.i.d. $\R$-valued $\F_0'$-measurable $\RV$ $(\eta^j)_{j=1}^N$ with law $\Law^{\Pb}(\eta)$, independent scalar $\Fb'$-Brownian motions $W^j=(W_t^j)_{t\in[0,T]}$ for $j=1,\ldots,N$, and $N$ independent real-valued $\mathbb{F}'$-adapted continuous process $(A^j)_{j=1}^N=((A_t^j)_{t\in[0,T]})_{i=1}^N$ and $((\mathtt{Y}_t^i)_{t\in [0,T]})_{i=1}^N$, $((Y_t^i[\boldsymbol{\alpha}])_{t\in [0,T]})_{i=1}^N$ (strongly) solving the following SDE system:
\begin{align*}
\d\mathtt{Y}_t^i&=b(t,\mathtt{X}_t^i,\rho_t^N,\beta_t)\d t+\sigma(t,\mathtt{X}_t^i,\rho_t^N,\beta_t)\d W_t^i,~~\mathtt{Y}_0^i=\eta^i,\quad i=1,\ldots,N,\\
\d Y_t^i[\boldsymbol{\alpha}]&=b(t,X_t^i[\boldsymbol{\alpha}],\rho_t^N,\alpha^i_t)\d t+\sigma(t,X_t^i[\boldsymbol{\alpha}],\rho_t^N,\alpha^i_t)\d W_t^i,~~Y_0^i[\boldsymbol{\alpha}]=\eta^i,\quad i=1,\ldots,N
\end{align*}
with $X^i=\Gamma(Y^i,A^i)$, $X^i[\boldsymbol{\alpha}]=\Gamma(Y^i[\boldsymbol{\alpha}],A^i)$, $\boldsymbol{\rho}^N=\frac1N\sum_{i=1}^N\delta_{(X^i[\boldsymbol{\alpha}],\delta_{\alpha^i_t}(\d u)\d t)}$ and $\rho_t^N$ being the $t$-marginal of $\boldsymbol{\rho}^N$ such that
\begin{align*}
P^N=P'\circ\left((Y^i[\boldsymbol{\alpha}],\delta_{\alpha^i_t}(\d u)\d t,W^i,A^i)_{i=1}^N\right)^{-1},\quad \bar{P}^N=P'\circ \left({\tt Y}^1,\delta_{\beta_t}(\d u)\d t,W^1,A^1,\mathscr{R}(\boldsymbol{\rho}^N)\right)^{-1}.    
\end{align*}
It can be easily verified that $\bar{P}^N=P'\circ (Y^i,\delta_{\beta_t}(\d u)\d t,W^i,A^i,\mathscr{R}(\boldsymbol{\rho}^N))^{-1}$ for $i=1,\ldots,N$. Let $(Y_t^j[\boldsymbol{\alpha}^{-i},\beta])_{i,j=1}^N$ for $t\in[0,T]$ solve the following SDE systems, for $i=1,\ldots,N$,
\begin{align*}
\d Y_t^i[\boldsymbol{\alpha}^{-i},\beta]&=b(t,X_t^i[\boldsymbol{\alpha}^{-i},\beta],\theta_t^{i,N},\beta_t)\d t+\sigma(t,X_t^i[\boldsymbol{\alpha}^{-i},\beta],\theta_t^{i,N},\beta_t)\d W_t^i,\\
\d Y_t^j[\boldsymbol{\alpha}^{-i},\beta]&=b(t,X_t^j[\boldsymbol{\alpha}^{-i},\beta],\theta_t^{i,N},\alpha_t^j)\d t+\sigma(t,X_t^j[\boldsymbol{\alpha}^{-i},\beta],\theta_t^{i,N},\alpha^j_t)\d W_t^j,
\end{align*}
with $Y_0^i[\boldsymbol{\alpha}^{-i},\beta]=\eta^i$, $Y_0^j[\boldsymbol{\alpha}^{-i},\beta]=\eta^j$, $\forall j\neq i$, $X^j[\boldsymbol{\alpha}^{-i},\beta]=\Gamma(Y^j[\boldsymbol{\alpha}^{-i},\beta],A^j)$ for $j=1,\ldots,N$, and
\begin{align*}
\boldsymbol{\theta}^{i,N}=\frac1N\left(\sum_{j\neq i}\delta_{(X^j[\boldsymbol{\alpha}^{-i},\beta],\alpha^j)}+\delta_{(X^i[\boldsymbol{\alpha}^{-i},\beta],\beta)}\right),   
\end{align*}
where $\theta_t^{i,N}$ denotes the $t$-marginal distribution of $\boldsymbol{\theta}^{i,N}$. We also set that
\begin{align*}
P^{i,N}&=P'\circ \bigg( \left(Y^1[\boldsymbol{\alpha}^{-i},\beta],\delta_{\alpha_t^1}(\d u)\d t,W^1,A^1\right),\dots,\left(Y^i[\boldsymbol{\alpha}^{-i},\beta],\delta_{\beta_t}(\d u)\d t,W^i,A^i\right),\dots,\nonumber\\
&\qquad\qquad\ldots,\left(Y^N[\boldsymbol{\alpha}^{-i},\beta],\delta_{\alpha_t^N}(\d u)\d t,W^N,A^N\right) \bigg)^{-1}.    
\end{align*}
Then, by construction, we have $P^{i,N}\in R_{i,N}^{\rm s}(P^N)$,  where we recall that the set $R_{i,N}^{\rm s}(P^N)$ is given in \defref{RsiN} for $i=1,\ldots,N$. Since $P^N$ forms an $\boldsymbol{\epsilon}^N$-Nash equilibrium (recall \defref{NE_weak}), it holds that
\begin{align}\label{NE_inequa}
\mathcal{J}_i(P^N)\leq\mathcal{J}_i(P^{i,N})+\epsilon^{i,N}.
\end{align}
By applying Ito's rule to $|Y_t^j[\boldsymbol{\alpha}]-Y_t^j[\boldsymbol{\alpha}^{-i},\beta]|^2$ for $j\neq i$, and using BDG inequality to derive that, there is a constant $C>0$ independent of $(i,N)$ such that
\begin{align}\label{difference_estimate}
&\E^{P'}\left[\sup_{s\in[0,t]}\left|Y_s^j[\boldsymbol{\alpha}]-Y_s^j[\boldsymbol{\alpha}^{-i},\beta]\right|^2\right]\nonumber\\
&\qquad\leq C\E^{P'}\left[\int_0^t\left|Y_s^j[\boldsymbol{\alpha}]-Y_s^j[\boldsymbol{\alpha}^{-i},\beta]\right|^2\d s+\int_0^t\mathcal{W}_{2,\R\times U}(\rho_s^N,\theta_s^{i,N})^2\d s\right],
\end{align}
where we have taken advantage of the Lipschitz property of $(b,\sigma)$ in \assref{ass3}  and the mapping $\Gamma:\mathcal{D}\to\C$ in \equref{Lipschitz_Gamma}. By definition, we have
\begin{align}\label{mean_measure_difference}
&\E^{P'}\left[\sup_{s\in [0,t]}\mathcal{W}_{2,\R\times U}\left(\rho_s^N,\theta_s^{i,N}\right)^2\right]\nonumber\\
&\quad\leq \frac1N\E^{P'}\left[\sup_{s\in [0,t]}\left(\sum_{j\neq i}\left|Y_s^j[\boldsymbol{\alpha}]-Y_s^j[\boldsymbol{\alpha}^{-i},\beta]\right|^2+\left|Y_s^i[\boldsymbol{\alpha}]-Y_s^i[\boldsymbol{\alpha}^{-i},\beta]\right|^2+|\alpha_s^i-\beta_s|^2\right)\right].
\end{align}
The compactness of the policy space $U$ with basic moment estimation yields that, for some constant $C>0$ independent of $(i,N)$,
\begin{align}\label{moment_i}
\E^{P'}\left[\sup_{t\in [0,T]}\left|Y_s^i[\boldsymbol{\alpha}]-Y_s^i[\boldsymbol{\alpha}^{-i},\beta]\right|^2+\sup_{t\in [0,T]}|\alpha_t^i-\beta_t|^2\right]\leq C.
\end{align}
Summing \equref{difference_estimate} over $j\neq i$ and dividing by $N$, we deduce by Gronwall's inequality that
\begin{align}
\lim_{N\to\infty}\max_{i=1,\ldots,N}\frac1N\sum_{j\neq i}\E^{P'}\left[\sup_{t\in [0,T]}\left|Y_t^j[\boldsymbol{\alpha}]-Y_t^j[\boldsymbol{\alpha}^{-i},\beta]\right|^2\right]=0.
\end{align}
Consequently, by \equref{mean_measure_difference}, it holds that 
		\begin{align}\label{mean_measure_convergence}
			\lim_{N\to\infty}\max_{i=1,\ldots,N}\E^{P'}\left[\sup_{t\in [0,T]}\mathcal{W}_{2,\R\times U}(\rho_t^N,\theta_t^{i,N})^2\right]=0.
		\end{align}
        Note that by definition, we may easily verify that\begin{align}\label{CB_metric}
            \mathcal{W}_{2,\C\times\mathcal{B}}(\boldsymbol{\rho}^N,\boldsymbol{\theta}^{i,N})\leq (T\vee 1)\sup_{t\in [0,T]}\mathcal{W}_{2,\R\times U}(\rho_t^N,\theta^{i,N}),
        \end{align}
		which yields that $\lim_{N\to\infty}\max_{i=1,\ldots,N}\E^{P'}\left[\mathcal{W}_{2,\C\times\mathcal{B}}(\boldsymbol{\rho}^N,\boldsymbol{\theta}^{i,N})^2\right]=0$. Moreover, by using \lemref{extension_continuous}, it holds that 
		\begin{align}\label{measure_convergence_mean}
			\lim_{N\to\infty}\max_{i=1,\ldots,N}\E^{P'}\left[\mathcal{W}_{2,\C\times\mathcal{Q}}(\mathscr{R}(\boldsymbol{\rho}^N),\mathscr{R}(\boldsymbol{\theta}^{i,N}))^2\right]=0.
		\end{align}
		Applying Ito's formula to $|Y_t^i[\boldsymbol{\alpha}^{-i},\beta]-\mathtt{Y}_t^i|^2$, and using BDG inequality, we can similarly conclude that
		\begin{align*}
			\E^{P'}\left[\sup_{s\in[0,t]}\left|Y_s^i[\boldsymbol{\alpha}^{-i},\beta]-\mathtt{Y}_s^i\right|^2\right]\leq C\E^{P'}\left[\int_0^t\left|Y_s^i[\boldsymbol{\alpha}^{-i},\beta]-\mathtt{Y}_s^i\right|^2\d s+\int_0^t\mathcal{W}_{2,\R\times U}(\rho_s^N,\theta_s^{i,N})^2\d s\right],
		\end{align*}
		where $C>0$ is a constant independent of $(i,N)$. Using Gronwall's inequality to the above inequality and exploiting \equref{mean_measure_difference}, we can then derive that
		\begin{align}\label{moment_convergence}
			\lim_{N\to\infty}\max_{i=1,\ldots,N}\E^{P'}\left[\sup_{t\in[0,T]}\left|Y_t^i[\boldsymbol{\alpha}^{-i},\beta]-\mathtt{Y}_t^i\right|^2\right]=0.
		\end{align} 
		Moreover, if we set $\bar{P}^{i,N}=P'\circ\left(Y^i[\boldsymbol{\alpha}^{-i},\beta],\delta_{\beta_t}(\d u)\d t,W^i,A^i,\mathscr{R}(\boldsymbol{\theta}^{i,N})\right)^{-1}$, it holds from  \equref{measure_convergence_mean} and \equref{moment_convergence} that
		\begin{align}\label{extended_measure_convergence}
			\lim_{N\to\infty}\max_{i=1,\ldots,N}\mathcal{W}_{2,\bar{\Omega}}(\bar{P}^{i,N},\bar{P}^N)=0.
		\end{align}
		As a result of \equref{extended_measure_convergence} and the proof in \lemref{compactness}, we have
		\begin{align}\label{value_convergence}
			\lim_{N\to\infty}\max_{i=1,\ldots,N}\left|\bar{\mathcal{J}}\left(\bar{P}^{i,N},P'\circ(\mathscr{R}(\boldsymbol{\theta}^{i,N}))^{-1}\right)-\bar{\mathcal{J}}(\bar{P}^N,\Theta^N)\right|=0.
		\end{align}
		For $i=1,\ldots,N$, it holds that
		\begin{align}\label{value_equivalence}
			\bar{\mathcal{J}}\left(\bar{P}^{i,N},P'\circ(\mathscr{R}(\boldsymbol{\theta}^{i,N}))^{-1}\right)=\mathcal{J}_i(P^{i,N}).
		\end{align}
		Then, by combining \equref{NE_inequa}, \equref{value_convergence} and \equref{value_equivalence}, and noting that
		\begin{align*}
			&\E^{P^N}\left[\mathcal{J}(\boldsymbol{\delta}^N,\boldsymbol{\xi}^N)\right]-\bar{\mathcal{J}}(\bar{P}^N,\Theta^N)=\frac1N\sum_{i=1}^N\mathcal{J}_i(P^N)-\bar{\mathcal{J}}(\bar{P}^N,\Theta^N)\nonumber\\
			&\quad\leq \frac1N\sum_{i=1}^N\left(\mathcal{J}_i(P^N)-\mathcal{J}_i(P^{i,N})\right)+\frac1N\sum_{i=1}^N\left|\bar{\mathcal{J}}\left(\bar{P}^{i,N},P'\circ(\mathscr{R}(\boldsymbol{\theta}^{i,N}))^{-1}\right)-\bar{\mathcal{J}}(\bar{P}^N,\Theta^N)\right|\\
			&\quad\leq\frac1N\sum_{i=1}^N\epsilon^{i,N}+\max_{i=1,\ldots,N}\left|\bar{\mathcal{J}}\left(\bar{P}^{i,N},P'\circ(\mathscr{R}(\boldsymbol{\theta}^{i,N}))^{-1}\right)-\bar{\mathcal{J}}(\bar{P}^N,\Theta^N)\right|,
		\end{align*}
		and we can finally conclude that
		\begin{align*}
			&\liminf_{N\to\infty}\left(\E^{\Xi^N}\left[\mathcal{J}(P,\boldsymbol{\xi})\right]-\inf_{\bar{Q}\in R_{\rm e}^{\rm s}(\Theta^N)}\bar{\mathcal{J}}(\bar{Q},\Theta^N)\right)\leq\liminf_{N\to\infty}\left(\E^{P^N}\left[\mathcal{J}(\boldsymbol{\delta}^N,\boldsymbol{\xi}^N)\right]-\bar{\mathcal{J}}(\bar{P}^N,\Theta^N)+\frac1N\right)\\
			&\quad\leq\lim_{N\to\infty}\left(\frac1N\sum_{i=1}^N\epsilon^{i,N}+\max_{i=1,\ldots,N}\left|\bar{\mathcal{J}}\left(\bar{P}^{i,N},P'\circ(\mathscr{R}(\boldsymbol{\theta}^{i,N}))^{-1}\right)-\bar{\mathcal{J}}(\bar{P}^N,\Theta^N)\right|+\frac1N\right)=0,
		\end{align*}
		which confirms the desired result \equref{goal}. The proof is thus complete.
	\end{proof}

	\subsubsection{Proof of \thmref{convergence_theorem}}\label{proof:thm3}
	
\begin{proof}[Proof of \thmref{convergence_theorem}]
Let $\psi:[0,T]\times\R\to U$ be the corresponding measurable mapping to $(\boldsymbol{\rho}^*,P^*)$ stated in \defref{markovian}, and $(\Omega',\F',\Fb',P')$ be a filtered probability space supporting $N$ i.i.d. $\R^2$-valued continuous $\Fb'$-adapted processes $(W^i,A^i)~,i=1,\dots,N$ with the same law as $\hat P$ and $W^i$ being standard $\Fb'$-Brownian motion for every $i=1,\dots,N$ and i.i.d. real-valued $\F_0'$-measurable $\RV$s $(\eta^1,\ldots,\eta^N)$ with the same law as $\eta$ under $\Pb$. Moreover, let $\mathtt{Y}^i=(\mathtt{Y}^i_t)_{t\in [0,T]}$ be the unique strong solution to the SDE, for $i=1,\ldots,N$,
\begin{align}\label{copy_MKV}
\d \mathtt{Y}^i_t=b(t,\mathtt{X}^i_t,\rho_t^*,\alpha_t^i)\d t+\sigma(t,\mathtt{X}^i_t,\rho_t^*,\alpha_t^i)\d W_t^i,\quad\mathtt{Y}_0^i=\eta^i
\end{align}
with $\mathtt{X}^i_t=\Gamma(A^i,\mathtt{Y}^i)_t$, $\alpha_t^i=\psi(t,\mathtt{X}^i_t)$ for $t\in[0,T]$ and $\boldsymbol{\rho}^*=P'\circ (X,\alpha)^{-1}$. Then, we introduce that $\hat{\boldsymbol{\rho}}^N:=\frac1N\sum_{i=1}^N\delta_{(\mathtt{X}^i,\alpha^i)}$ and $\hat{\rho}^N_t:=\frac1N\sum_{i=1}^N\delta_{(\mathtt{X}^i_t,\alpha^i_t)}$. By the classical theory of McKean-Vlasov SDE (see Djete et. al. \cite{Djete}), it follows that $\mathtt{Y}^i=(\mathtt{Y}_t^i)_{t\in[0,T]}$ for $i=1,\ldots,N$ are i.i.d., and hence by the law of large numbers, we have  $\hat{\boldsymbol{\rho}}^N$ converges weakly to $\boldsymbol{\rho}^*$ as $N\to\infty$. 
		
The finite moment estimates in \lemref{moment_p} can upgrade this convergence to $\Pc_2(\C\times \mathcal{B})$ (c.f. Proposition 5.2 in Lacker \cite{Lacker1}). That is, it holds that
\begin{align}\label{LLN}
\lim_{N\to\infty}\mathcal{W}_{2,\C\times\mathcal{B}}(\boldsymbol{\rho}^*,\hat{\boldsymbol{\rho}}^N)=0.
\end{align}  
Moreover, we have $P^*=P'\circ (\mathtt{Y}^1,\alpha^1,W^1,A^1)^{-1}$, and $P'$-$\as$,
\begin{align}\label{LLN_strong}
\lim_{N\to\infty}\mathcal{W}_{2,\Omega}\left(\frac1N\sum_{i=1}^N\delta_{(\mathtt{Y}^i,\delta_{\alpha^i_t}(\d u)\d t,W^i,A^i)},P^*\right)=0.
\end{align} 
Let $\mathtt{Y}^{i,N}=(\mathtt{Y}_t^{i,N})_{t\in[0,T]}$ be the unique solution to the following SDE, for $i=1,\ldots,N$,
\begin{align*}	\d\mathtt{Y}_t^{i,N}=b(t,\mathtt{X}_t^{i,N},\rho_t^N,\alpha_t^i)\d t+\sigma(t,\mathtt{X}_t^{i,N},\rho_t^N,\alpha_t^i)\d W_t^i,\quad \mathtt{Y}_0^{i,N}=\eta^i  
\end{align*}
with $\mathtt{X}^{i,N}_t:=\Gamma(A^i,\mathtt{Y}^{i,N})_t$ and $\rho_t^N:=\frac{1}{N}\sum_{i=1}^N\delta_{(\mathtt{X}^{i,N}_t,\alpha^i_t)}$ for $t\in[0,T]$. By applying It\^{o} lemma to $|\mathtt{Y}^i_t-\mathtt{Y}_t^{i,N}|^2$ and BDG inequality, one can derive that
\begin{align}\label{Y_estimation}	\E^{P'}\left[\sup_{s\in[0,t]}\left|\mathtt{Y}_s^i-\mathtt{Y}_s^{i,N}\right|^2\right]\leq C\E^{P'}\left[\int_0^t\left|\mathtt{Y}_s^i-\mathtt{Y}_s^{i,N}\right|^2\d s+\int_0^t\mathcal{W}_{2,\R\times U}(\rho_s^*,\rho_s^N)^2\d s\right],
\end{align}
where $C>0$ is a constant independent of $(i,N)$ and we have exploited the Lipschitz property of $b,\sigma$ imposed in \assref{ass3}  and the mapping $\Gamma:\mathcal{D}\to\C$ defined by \equref{Lipschitz_Gamma}. By Gronwall's inequality and the triangular inequality, we have
		\begin{align}\label{triangular}
			\mathcal{W}_{2,\R\times U}(\rho_s^*,\rho_s^N)^2&\leq 2\mathcal{W}_{2,\R\times U}(\rho_s^*,\hat{\rho}_s^N)^2+2\mathcal{W}_{2,\R\times U}(\hat{\rho}_s^N,\rho_s^N)^2\nonumber\\
			&\leq  2\mathcal{W}_{2,\R\times U}(\rho_s^*,\hat{\rho}_s^N)^2+\frac2N\sum_{i=1}^N|(\mathtt{X}_s^{i,N},\alpha^i_s)-(\mathtt{X}_s^i,\alpha^i_s)|^2\nonumber\\
			&\leq 2\mathcal{W}_{2,\R\times U}(\rho_s^*,\hat{\rho}_s^N)^2+\frac2N\sum_{i=1}^N|\mathtt{Y}_s^{i,N}-\mathtt{Y}_s^i|^2.
		\end{align}
		By averaging \equref{Y_estimation} from $i=1$ to $N$, and applying Gronwall's inequality, it holds that
		\begin{align*}
			\frac1N\sum_{i=1}^N\E^{P'}\left[\sup_{s\in[0,t]}\left|\mathtt{Y}_s^i-\mathtt{Y}_s^{i,N}\right|^2\right]\leq C\E^{P'}\left[\int_0^t\mathcal{W}_{2,\R\times U}(\rho_s^*,\rho_s^N)^2\d s\right]
		\end{align*}
		with $C>0$ being a generic constant independent of $(i,N)$. Hence, we have from \lemref{CB_property}, \equref{LLN} and \equref{triangular} that
		\begin{align}\label{Y_convergence_pre}
			\lim_{N\to\infty}\frac1N\sum_{i=1}^N\E^{P'}\left[\sup_{t\in[0,T]}\left|\mathtt{Y}^i_t-\mathtt{Y}_t^{i,N}\right|^2\right]=0.
		\end{align}
		As a direct consequence of \equref{LLN}, \equref{triangular} and \equref{Y_convergence_pre}, we get that
		\begin{align}\label{mean_convergence}
			\lim_{N\to\infty}\E^{P'}\left[\mathcal{W}_{2,\C\times\mathcal{B}}(\boldsymbol{\rho}^*,\boldsymbol{\rho}^N)\right]=0.
		\end{align}
Inserting \equref{mean_convergence} into \equref{Y_estimation}, applying \lemref{CB_property} and Gronwall's inequality, we arrive at
\begin{align}\label{Y_convergence}
\lim_{N\to\infty}\E^{P'}\left[\sup_{t\in[0,T]}\left|\mathtt{Y}^i_t-\mathtt{Y}_t^{i,N}\right|^2\right]=0.
\end{align}
We then set $P^N:=P'\circ((\mathtt{Y}^{i,N},\delta_{\alpha_t^i}(\d u)\d t,W^i,A^i)_{i=1}^N)^{-1}$, and thus $P^N\in R_N^{\rm s}$ by construction. It follows from \equref{Y_convergence} that, for $i=1,\ldots,N$,
\begin{align}\label{P_convergence}
\lim_{N\to\infty}\mathcal{W}_{2,\Omega}\left(P^N\circ(Y^i,\Lambda^i,W^i,A^i)^{-1},P^*\right)=0.
\end{align}
On the other hand, we have by the triangular inequality that
{\small\begin{align*}
&\quad\E^{P^N}\left[\mathcal{W}_{2,\Omega}\left(\frac1N\sum_{i=1}^N\delta_{(Y^i,\delta_{\alpha_t^{i,N}}(\d u)\d t,W^i,A^i)},P^*\right)\right]=\E^{P'}\left[\mathcal{W}_{2,\Omega}\left(\frac1N\sum_{i=1}^N\delta_{(\mathtt{Y}^{i,N},\delta_{\alpha_t^{i}}(\d u)\d t,W^i,A^i)},P^*\right)\right]\\
&\qquad\leq \E^{P'}\left[\mathcal{W}_{2,\Omega}\left(\frac1N\sum_{i=1}^N\delta_{\left(\mathtt{Y}^{i,N},\delta_{\alpha_t^{i}}(\d u)\d t,W^i,A^i\right)},\frac1N\sum_{i=1}^N\delta_{\left(\mathtt{Y}^i,\delta_{\alpha_t^i}(\d u)\d t,W^i,A^i\right)}\right)\right.\\
&\qquad\quad\left.+\mathcal{W}_{2,\Omega}\left(P^*,\frac1N\sum_{i=1}^N\delta_{\left(\mathtt{Y}^i,\delta_{\alpha_t^i}(\d u)\d t,W^i,A^i\right)}\right)\right]=:I_1+I_2.
\end{align*}}By definition, we have $I_1\leq\frac1N\sum_{i=1}^N\E^{P'}\left[\|\mathtt{Y}^i-\mathtt{Y}^{i,N}\|_{\infty}^2\right]$, which converges to $0$ as $N\to\infty$ by using \equref{Y_convergence}. Applying \equref{LLN_strong} and the dominated convergence theorem, we obtain that $I_2$ converges to $0$ as $N\to\infty$. Thus, it holds that
		\begin{align*}
			\lim_{N\to\infty}\E^{P^N}\left[\mathcal{W}_{2,\Omega}\left(\frac1N\sum_{i=1}^N\delta_{(Y^i,\Lambda_t^i(\d u)\d t,W^i, A^i)},P^*\right)\right]=0.   
		\end{align*}
		The rest is to show that, there exists some $\boldsymbol{\epsilon}^N\in\R_+^N$ for $N\geq 1$ satisfying $\lim\limits_{N\to\infty}\max\limits_{i=1,\dots,N}\epsilon^{i,N}=0$, such that $P^N$ is an $\boldsymbol{\epsilon}^N$-Nash equilibrium.  To this end, let us set\begin{align}
\label{epsiloniN}
\epsilon^{i,N}:=\sup_{Q^N\in R^{\rm s}_{i,N}(P^N)}\left(\mathcal{J}_i(P^N)-\mathcal{J}_i(Q^N)\right).
\end{align} 
By construction, we have $\epsilon^{i,N}\geq 0$, and hence it only suffices to show that $\lim\limits_{N\to\infty}\max\limits_{i=1,\dots,N}\epsilon^{i,N}=0$. Note that, for any $N\geq1$, there exists some $Q^N\in R_{i,N}^{\rm s}(P^N)$ such that 
\begin{align}\label{Q_construction}
\mathcal{J}_i(P^N)-\mathcal{J}_i(Q^N)>\epsilon^{i,N}-2^{-N}.
\end{align} 
Without loss of generality, we assume that the filtered probability space with $(W^i,A^i,\eta^i,\mathtt{\alpha}^i)_{i=1}^N$ is the basic stochastic element corresponding to $(P^N,Q^N)$ introduced in \defref{RsiN} (we can similarly define $\mathtt{Y}^i$ on a different space and so is $\alpha^i$) and there exists a $\beta=(\beta_t)_{t\in [0,T]}$ being a $U$-valued $\Fb'$-progressively measurable process such that 
\begin{align*}
Q^N=P'\circ\left((\mathscr{Z}^{j,N})_{j=1}^N,(\delta_{\boldsymbol{\alpha}_t^{-i}}(\d u)\d t,\delta_{\beta_t}(\d u)\d t),(W^j)_{j=1}^N,(A^j)_{j=1}^N\right)^{-1},    
\end{align*}
where $(\mathscr{Z}^{j,N})_{j=1}^N=((\mathscr{Z}_t^{j,N})_{t\in [0,T]})_{j=1}^N$ is the unique solution to the system of SDEs:
		\begin{align*}
			\mathscr{Z}_t^{j,N}&=b(t,\mathscr{X}_t^{j,N},\theta^N_t,\alpha^j_t)\d t+\sigma(t,\mathscr{X}_t^{j,N},\theta^N_t,\alpha^j_t)\d W_t^j,~\mathscr{Z}_0^j=\eta^j,~j\neq i,\\
			\mathscr{Z}_t^{i,N}&=b(t,\mathscr{X}_t^{i,N},\theta^N_t,\beta_t)\d t+\sigma(t,\mathscr{X}_t^{i,N},\theta^N_t,\beta_t)\d W_t^i,~\mathscr{Z}_0^i=\eta^i
		\end{align*}
		with $\mathscr{X}_t^{j,N}=\Gamma(A^j,\mathscr{Y}^{j,N})_t$ for $j=1,\ldots,N$ and
		\begin{align*}
			\boldsymbol{\theta}^N:=\frac1N\left(\sum_{j\neq i}\delta_{(\mathscr{X}^{j,N},\alpha^j)}+\delta_{(\mathscr{X}^{i,N},\beta)}\right),\quad \theta_t^N=\frac1N\left(\sum_{j\neq i}\delta_{(\mathscr{X}^{j,N}_t,\alpha^j_t)}+\delta_{(\mathscr{X}^{i,N}_t,\beta_t)}\right).    
		\end{align*}
		Applying It\^{o} formula to $|\mathscr{Z}_t^{j,N}-\mathtt{Y}^{j,N}|^2$ for $j\neq i$ and using BDG inequality, we have
		\begin{align}\label{ZY}
			\E^{P'}\left[\sup_{s\in[0,t]}\left|\mathscr{Z}_s^{j,N}-\mathtt{Y}^j_s\right|^2\right]\leq C\E^{P'}\left[\int_0^t\left|\mathtt{Y}_s^j-\mathscr{Z}_s^{j,N}\right|^2\d s+\int_0^t\mathcal{W}_{2,\R\times U}(\rho_s^*,\theta_s^N)^2\d s\right],
		\end{align} 
		where $C>0$ is a constant independent of $(j,N)$ and we have taken advantage of the Lipschitz property of $b,\sigma$ imposed in \assref{ass3} and the mapping $\Gamma:\mathcal{D}\to\C$ in \equref{Lipschitz_Gamma}. Moreover, we have
		\begin{align}\label{difference}
			\E^{P'}\left[\mathcal{W}_2\left(\hat{\rho}_t^N,\theta_t^N\right)^2\right]&\leq\frac1N\E^{P'}\left[\sum_{j\neq i}\left|\mathscr{Z}_t^{j,N}-\mathtt{Y}_t^j\right|^2+\left|\mathscr{Z}_t^{i,N}-\mathtt{Y}_t^i\right|^2+\left|\alpha_t^i-\beta_t\right|^2\right].
		\end{align}
		A standard estimate (similar as \lemref{moment_p_N}), together with the growth condition provided in \assref{ass1}, results in the existence of a constant $C>0$ independent of $(i,N,t)$ such that
		\begin{align*}
			\E^{P'}\left[\sup_{t\in [0,T]}\left(\left|\mathscr{Z}_t^{i,N}-\mathtt{Y}_t^i\right|^2+\left|\alpha_t^i-\beta_t\right|^2+\mathcal{W}_2\left(\rho_t^*,\hat{\rho}_t^N\right)^2\right)\right]\leq C.   
		\end{align*}
		Consequently, the estimate \equref{difference} turns to be, for all $t\in[0,T]$,
		\begin{align}\label{difference1}
			\E^{P'}\left[\mathcal{W}_2\left(\hat{\rho}_s^N,\theta_s^N\right)^2\right]\leq\frac1N\E^{P'}\left[\sum_{j\neq i}\left|\mathscr{Z}_s^{j,N}-\mathtt{Y}_s^j\right|^2\right]+\frac{C}{N}.
		\end{align}
		Then, the estimate \equref{difference1}, together with \equref{triangular}, yields that, for all $t\in[0,T]$,
		\begin{align}\label{difference2}
			\E^{P'}\left[\mathcal{W}_2\left(\rho_t^*,\theta_t^N\right)^2\right]\leq\frac2N\E^{P'}\left[\sum_{j\neq i}\left|\mathscr{Z}_t^{j,N}-\mathtt{Y}_t^j\right|^2\right]+\frac{2C}{N}+	2\E^{P'}\left[\mathcal{W}_2\left(\hat{\rho}_t^N,\rho_t^*\right)^2\right].
		\end{align}
		Combining \equref{LLN} and \equref{difference2}, summing \equref{ZY} over $j\neq i$, and then dividing by $N$, we conclude from the Gronwall's inequality that
		\begin{align}\label{Z_convergence}
			\lim_{N\to\infty}\frac1N\sum_{j\neq i}\E^{P'}\left[\sup_{t\in[0,T]}\left|\mathscr{Z}_t^{j,N}-\mathtt{Y}^j_t\right|^2\right]=0.
		\end{align}
		Furthermore, as a result of \equref{Z_convergence}, it holds that
		\begin{align}\label{measure_convergence}
			\lim_{N\to\infty}\E^{P'}\left[\mathcal{W}_{2,\C\times\mathcal{B}}\left(\boldsymbol{\rho}^*,\boldsymbol{\theta}^N\right)\right]=0.
		\end{align}
		Let $\mathscr{Y}^i=(\mathscr{Y}_t^i)_{t\in [0,T]}$ be the unique solution to the following SDE:
		\begin{align*}
			\d\mathscr{Y}_t^i=b(t,\Gamma(A^i,\mathscr{Y}^i)_t,\rho_t^*,\beta_t)\d t+\sigma(t,\Gamma(A^i,\mathscr{Y}^i)_t,\rho_t^*,\beta_t)\d W_t^i,~ \mathscr{Y}_0^i=\eta^i.    
		\end{align*}
Then, it follows from It\^{o}'s lemma to $|\mathscr{Y}_t^i-\mathscr{Z}_t^i|^2$ and BDG inequality that
\begin{align*}	\E^{P'}\left[\sup_{s\in[0,t]}\left|\mathscr{Z}_s^{i,N}-\mathscr{Y}^i_s\right|^2\right]\leq C\E^{P'}\left[\int_0^t\left|\mathscr{Y}_s^i-\mathscr{Z}_s^{i,N}\right|^2\d s+\int_0^t\mathcal{W}_{2,\R\times U}\left(\rho_s^*,\theta_s^N\right)^2\d s\right].   
\end{align*}
We derive from \equref{measure_convergence} that
\begin{align}\label{YZ_convergence}
\lim_{N\to\infty}\E^{P'}\left[\sup_{s\in[0,t]}\left|\mathscr{Z}_s^{i,N}-\mathscr{Y}^i_s\right|^2\right]=0.
\end{align}
Set $Q:=P'\circ(\mathscr{Y}^i,\delta_{\beta_t}(\d u)\d t,W^i,A^i)$. Thus, $Q\in R^{\rm s}(\boldsymbol{\xi}^*)$ with $\boldsymbol{\xi}^*=\mathscr{R}(\boldsymbol{\rho}^*)$. Then, it can be proved in the same fashion of \equref{convergence} that
\begin{align}\label{Q_convergence}	\lim_{N\to\infty}\mathcal{W}_{2,\Omega}\left(Q^N\circ (Y^i,\Lambda^i,W^i,A^i)^{-1},Q\right)=0.
\end{align}
Note also that (recall \equref{Delta_i}): 
\begin{align*}	\mathcal{J}_i(Q^N)&=\E^{Q^N}\left[\Delta_i(\mathscr{R}(\boldsymbol{\theta}^N))\right]=\E^{Q^N}\left[\Delta_i(\boldsymbol{\xi}^*)\right]+\E^{Q^N}\left[\Delta_i(\mathscr{R}(\boldsymbol{\theta}^N))-\Delta_i(\boldsymbol{\xi}^*)\right]=I_1^N+I_2^N.
\end{align*}
It follows from \equref{Q_convergence} and and the proof in \lemref{compactness} that $\lim_{N\to\infty}I_1^N=\E^Q\left[\Delta(\boldsymbol{\xi}^*)\right]=\mathcal{J}(\boldsymbol{\xi}^*,Q)$. In view of the uniform convergence in (A1) of \assref{ass1} and \equref{measure_convergence}, we obtain that $\lim_{N\to\infty}I_2^N=0$. As a result, it holds that
\begin{align}\label{value_convergence1}
\lim_{N\to\infty}\mathcal{J}_i(Q^N)=\mathcal{J}(\boldsymbol{\xi}^*,Q).
\end{align}
Similarly, we can also show that
\begin{align}\label{value_convergence2}
\lim_{N\to\infty}\mathcal{J}_i(P^N)=\mathcal{J}(\boldsymbol{\xi}^*,P).
\end{align}  
In light of the optimality of $P^*$, we have $\mathcal{J}(\boldsymbol{\xi}^*,P^*)\leq\mathcal{J}(\boldsymbol{\xi}^*,Q)$. It then follows that\begin{align*}
			\epsilon^{i,N}-2^{-N}&<\mathcal{J}_i(P^N)-\mathcal{J}_i(Q^N)\\
			&=(\mathcal{J}_i(P^N)-\mathcal{J}(\boldsymbol{\xi}^*,P^*))-(\mathcal{J}_i(Q^N)-\mathcal{J}(\boldsymbol{\xi}^*,Q))+(\mathcal{J}(\boldsymbol{\xi}^*,P^*)-\mathcal{J}(\boldsymbol{\xi}^*,Q))\\
			&\leq |\mathcal{J}_i(P^N)-\mathcal{J}(\boldsymbol{\xi}^*,P^*)|+|\mathcal{J}_i(Q^N)-\mathcal{J}(\boldsymbol{\xi}^*,Q)|,
		\end{align*}
		which indicates that $\lim_{N\to\infty}\epsilon^{i,N}=0$ in view of \equref{value_convergence1}, \equref{value_convergence2} and the construction of $Q^N$ in \equref{Q_construction}. Note that the convergence is uniform in $i$, and hence $\lim_{N\to\infty}\max_{i=1,\dots,N}\epsilon^{i,N}=0$. The proof is thus complete.
	\end{proof}
	
\subsection{Proofs of auxiliary results}\label{appendix}

In this subsection, we collect proofs of auxiliary results in the previous sections.
\begin{proof}[Proof of \lemref{CB_property}]
Let $(X^1,\alpha^1)$ and $(X^2,\alpha^2)$ be $\C\times\mathcal{B}$-valued $\RV$s such that $\Law(X^1,\alpha^1)=\boldsymbol{\rho}^1$,  $\Law(X^2,\alpha^2)=\boldsymbol{\rho}^2$ and $\mathcal{W}_{2,\C\times\mathcal{B}}(\boldsymbol{\rho}^1,\boldsymbol{\rho}^2)^2=\E\left[d_{\C\times\mathcal{B}}((X^1,\alpha^1),(X^2,\alpha^2))^2\right]$. Note that $\Law(X_t^1,\alpha_t^1)=\rho_t^1$ and $\Law(X_t^2,\alpha_t^2)=\rho_t^2$. Then, we have 
\begin{align*}
&\int_0^T\mathcal{W}_{2,\R\times U}(\rho_t^1,\rho_t^2)^2\d t\leq \int_0^T\E\left[\left|X_t^1-X_t^2\right|^2+\left|\alpha_t^1-\alpha_t^2\right|^2\right]\d t\nonumber\\
&\quad\leq \int_0^T\E\left[\left\|X^1-X^2\right\|_{\infty}^2\right]\d t+\E\left[\int_0^T\left|\alpha_t^1-\alpha_t^2\right|^2\d t\right]=\E\left[T\left\|X^1-X^2\right\|_{\infty}^2+d_{\mathcal{B}}(\alpha^1,\alpha^2)^2\right]\\
&\quad\leq C\E\left[d_{\C\times\mathcal{B}}((X^1,\alpha^1),(X^2,\alpha^2))^2\right]=C\mathcal{W}_{2,\C\times\mathcal{B}}(\boldsymbol{\rho}^1,\boldsymbol{\rho}^2)^2
\end{align*}
		with $C=T\vee 1$. Above, we also applied Fubini's theorem in the second inequality. Thus, the desired result follows.
	\end{proof}

\begin{proof}[Proof of \lemref{Lipschitz_P}]
Let $(X_i,q_i)$ for $i=1,2$ be two random variables taking values in $\R\times\Pc(U)$ on some probability space $(\Omega',\F',\Pb')$ such that $\Law^{P'}(X_i,q_i)=\xi_i$ for $i=1,2$ and
\begin{align*}
\mathcal{W}_{2,\R\times\Pc(U)}(\xi_1,\xi_2)^2=\E^{P'}\left[d_{\R\times\Pc(U)}((X_1,q_1),(X_2,q_2))^2\right],   
\end{align*}
In terms of definition $\rho_i:=\mathscr{P}(\xi_i)$ for $i=1,2$, we have $\rho_i|_{U}(\d u)=\int_{\R}\int_{\mathcal{M}(U)}q(\d u)\xi_i(\d x,\d q)=\E^{P'}[q_i(\d u)]$ for $i=1,2$. We also note that
\begin{align*}
\mathcal{W}^2_{2,\R\times U}(\rho_1,\rho_2)=\inf_{\substack{\Law^{\mathbb{P}}(Y_1,\alpha_1)=\rho_1\\\Law^{\mathbb{P}}(Y_2,\alpha_2)=\rho_2}}
\E\left[|Y_1-Y_2|^2+|\alpha_1-\alpha_2|^2\right]=\mathcal{W}^2_{2,\R}(\rho_1|_{\R},\rho_2|_{\R})+\mathcal{W}^2_{2,U}(\rho_1|_{U},\rho_2|_U).    
\end{align*}

		Let us claim that
		\begin{align}\label{commute}
			\mathcal{W}^2_{2,U}\left(\E^{P'}[q_1(\d u)],~\E^{P'}[q_2(\d u)]\right)\leq\E^{P'}\left[\mathcal{W}^2_{2,U}(q_1,q_2)\right].
		\end{align}
		Indeed, using the Kantorovich's duality, we deduce that
		\begin{align*}
			&\mathcal{W}^2_{2,U}\left(\E^{P'}[q_1(\d u)],~\E^{P'}[q_2(\d u)]\right)=\sup_{\substack{f(x)+g(y)\leq |x-y|^2}}\int_Uf(u)\E^{P'}[q_1(\d u)]+\int_Ug(u)\E^{P'}[q_2(\d u)]\\
			&\qquad=\sup_{\substack{f(x)+g(y)\leq |x-y|^2}}\int_Uf(u)\int_{\Omega'}q_1(\omega',\d u)P'(\d \omega)+\int_Ug(u)\int_{\Omega'}q_2(\omega',\d u)P'(\d \omega)\\
			&\qquad=\sup_{\substack{f(x)+g(y)\leq |x-y|^2}}\int_{\Omega'}\left(\int_Uf(u)q_1(\omega',\d u)+\int_Ug(u)q_2(\omega',\d u)\right)P'(\d\omega)\\
			&\qquad\leq \int_{\Omega'}\sup_{\substack{f(x)+g(y)\leq |x-y|^2}}\left(\int_Uf(u)q_1(\omega',\d u)+\int_Ug(u)q_2(\omega',\d u)\right)P'(\d\omega)\\
			&\qquad=\E^{P'}\left[\mathcal{W}^2_{2,U}(q_1,q_2)\right],
		\end{align*}
		where we have applied Fubini' theorem in the third equality. Given the result \equref{commute}, we then derive from the fact $\rho_i|_{\R}=\Law^{P'}(X_i)$ for $i=1,2$ that
		\begin{align*}
			\mathcal{W}^2_{2,\R\times U}(\rho_1,\rho_2)\leq\E^{P'}\left[|X_1-X_2|^2+\mathcal{W}^2_{2,U}(q_1,q_2)\right]\leq \E\left[d^2_{\R\times\Pc(U)}((X_1,q_1),(X_2,q_2))\right].  
		\end{align*}
		This implies the desired result. 
	\end{proof}

	\begin{proof}[Proof of \lemref{extension_continuous}] 
		Choose $(X,\alpha),(Y,\beta)\in\C\times\mathcal{B}$ such that $\Law(X,\alpha)=\boldsymbol{\rho}^1,\Law(Y,\beta)=\boldsymbol{\rho}^2$ and \[
		\mathcal{W}_{2,\C\times\mathcal{B}}(\boldsymbol{\rho}^1,\boldsymbol{\rho}^2)^2=\E\left[\|X-Y\|_{\infty}^2+d_{\mathcal{B}}(\alpha,\beta)^2\right].
		\]
		Then, by Kantorovich duality, it holds that
		\begin{align*}
			\mathcal{W}_{2,\C\times\mathcal{Q}}(\mathscr{R}(\boldsymbol{\rho}^1),\mathscr{R}(\boldsymbol{\rho}^2))^2&=\mathcal{W}_{2,\C\times\mathcal{Q}}(\Psi(X,\alpha),\Psi(Y,\beta))^2\\
			&\leq \E\left[\|X-Y\|^2+d_{\mathcal{Q}}(\delta_{\alpha_t}(\d u)\d t,\delta_{\beta_t}(\d u)\d t)^2\right]\\
			&= \E\left[\|X-Y\|_{\infty}^2+\sup_{(h_1,h_2):h_1(u)+h_2(v)\leq|u-v|^2}\int_0^T(h_1(\alpha_t)+h_2(\beta_t))\d t\right]\\
			&\leq\E\left[\|X-Y\|_{\infty}^2+\int_0^T|\alpha_t-\beta_t|^2\d t\right]=\mathcal{W}_{2,\C\times\mathcal{B}}(\boldsymbol{\rho}^1,\boldsymbol{\rho}^2)^2.
		\end{align*}
		The proof is thus complete.
	\end{proof}

\begin{proof}[Proof of \propref{measurable}]
Let $\mu_t$ be the first marginal law of $\xi_t$ for $t\in [0,T]$. In light of \assref{ass2}, for $(t,\omega)\in [0,T]\times\Omega$, it holds that
\begin{align*}
&\int_U\tb(t,X_t(\omega),\xi_t,u)\Lambda_t(\omega,\d u)=\int_Ub_0(t,X_t(\omega),\mu_t,u)\Lambda_t(\omega,\d u)+\lambda_1\int_{\R\times U}b_0(t,x,\mu_t,u)\mathscr{P}(\xi_t)(\d x,\d u)\\
&\quad=\int_Ub_0(t,X_t(\omega),\mu_t,u)\Lambda_t(\omega,\d u)+\lambda_1\int_{\R\times \Pc(U)}\int_Ub_0(t,x,\mu_t,u)q(\d u)\xi_t(\d x,\d q),
\end{align*}
where we have used the fact that $\mathscr{P}(\xi_t)|_{\R}=\xi_t|_{\R}=\mu_t$. The similar result also holds for the extensions $\ts$ and $\tf$, respectively.
		
Let us introduce the set $K(t,\omega):=K(t,X_t(\omega),\xi_t)$ for $(t,\omega)\in[0,T]\times\Omega$. Note that $K(t,\omega)$ is convex and closed, we have
\begin{align*}
			&\left(\int_Ub_0(t,X_t(\omega),\mu_t,u)q_t(\omega,\d u),\int_U\sigma_0^2(t,X_t(\omega),\mu_t,u)q_t(\omega,\d u),\int_Uf_0(t,X_t(\omega),\mu_t,u)q_t(\omega,\d u)\right)\nonumber\\
			&\qquad\in K(t,\omega),
		\end{align*}
		which is equivalent to the set
		\begin{align*}
			U(t,\omega):=\Big\{u\in U;~ &\int_Ub_0(t,X_t(\omega),\mu_t,u)q_t(\omega,\d u)=b_0(t,X_t(\omega),\mu_t,u),\\	 		&\int_U\sigma_0^2(t,X_t(\omega),\mu_t,u)q_t(\omega,\d u)=\sigma_0^2(t,X_t(\omega),\mu_t,u),\\
			&\int_Uf_0(t,X_t(\omega),\mu_t,u)q_t(\omega,\d u)\geq f_0(t,X_t(\omega),\mu_t,u)\Big\}\neq\varnothing.
		\end{align*}
As the policy space $U$ is compact, there exists a measurable selection $\alpha:[0,T]\times\Omega\to U$ such that $\alpha_t(\omega)\in U(t,\omega)$ for $(t,\omega)\in [0,T]\times\Omega$, and  $\alpha$ is $\mathcal{B}([0,t])\otimes\mathcal{F}_t/\mathcal{B}(U)$ measurable according to Leese \cite{Leese}. Now, let us set $Q=P\circ (Y,\delta_{\alpha_t}(\d u)\d t, W,A)^{-1}\in\Pc_2(\Omega)$ and we will show that $Q\in R(\boldsymbol{\xi})$. To this end, it suffices to prove (iv) of \defref{relaxed_control}. By the definition and \assref{ass2}, we have 
\begin{align*}
&M^{\boldsymbol{\xi}}\phi(t,\omega)=\phi(Y_t(\omega),W_t(\omega))-\int_0^t\int_U\bar{\mathcal{L}}\phi(s,X_s(\omega),Y_s(\omega),W_s(\omega),\xi_s,u)\Lambda_s(\omega,\d u)\d s\\
&\quad =\phi(Y_t(\omega),W_t(\omega))-\int_0^t\int_U\hat b_0(s,X_s(\omega),\mu_s,u)^{\T}\nabla\phi(Y_s(\omega),W_s(\omega))\Lambda_t(\omega,\d u)\d s\\
&\qquad-\int_0^t\int_U\frac12\tr\left(\hat\sigma_0\hat\sigma_0^{\T}(s,X_s(\omega),\mu_s,u)\nabla^2\phi(Y_s(\omega),W_s(\omega))\right)\Lambda_t(\omega,\d u)\d s\\
&\qquad-\int_0^t\int_{\R\times\Pc(U)}\int_U\lambda_1 \hat b_0(s,x,\mu_s,u)^{\T}\nabla\phi(Y_s(\omega),W_s(\omega))q(\d u)\xi_s(\d x,\d q)\d s\\
&\qquad-\int_0^t\int_{\R\times \Pc(U)}\int_U\frac{\lambda_2}{2}\tr\left(\hat\sigma_0\hat\sigma_0^2(s,x,\mu_s,u)\nabla^2\phi(Y_s(\omega),W_s(\omega))\right)q(\d u)\xi_s(\d x,\d q)\d s\\
&\quad =\phi(Y_t(\omega),W_t(\omega))-\int_0^t\hat b_0(s,X_s(\omega),\mu_s,\alpha_s)^{\T}\nabla\phi(Y_s(\omega),W_s(\omega))\d s\\
&\qquad-\int_0^t\frac12\tr\left(\hat\sigma_0\hat\sigma_0^{\T}(s,X_s(\omega),\mu_s,\alpha_s)\nabla^2\phi(Y_s(\omega),W_s(\omega))\right)\d s\\
&\qquad-\int_0^t\int_{\R\times\Pc(U)}\int_U\lambda_1 \hat b_0(s,x,\mu_s,u)^{\T}\nabla\phi(Y_s(\omega),W_s(\omega))q(\d u)\xi_s(\d x,\d q)\d s\\
&\qquad-\int_0^t\int_{\R\times \Pc(U)}\int_U\frac{\lambda_2}{2}\tr\left(\hat\sigma_0\hat\sigma_0^2(s,x,\mu_s,u)\nabla^2\phi(Y_s(\omega),W_s(\omega))\right)q(\d u)\xi_s(\d x,\d q)\d s\\
&\quad:=N^{\boldsymbol{\xi}}\phi(t,\omega),
\end{align*}
where, for $(t,x,\mu,u)\in [0,T]\times\R\times\Pc(\R)\times U$,
        \begin{align*}
            \hat b_0(t,x,\mu,u)=\begin{pmatrix}
                b_0(t,x,\mu,u)\\
                0
            \end{pmatrix},~~
            \hat\sigma_0(t,x,\mu,u)=\begin{pmatrix}
                \sigma_0(t,x,\mu,u)\\
                1
            \end{pmatrix},
        \end{align*} 
        and in the third equality we have exploited the fact that $\alpha_t(\omega)\in U(t,\omega)$. 
		Because $P\in R(\boldsymbol{\xi})$, $N^{\boldsymbol{\xi}}$ is then a $(P,\Fb)$-martingale. For any $0\leq s<t\leq T$ and $\F_s$-measurable $\RV$ $H$, we have
		\begin{align*}
			\E^Q[(M^{\boldsymbol{\xi}}\phi(t)-M^{\boldsymbol{\xi}}\phi(s))H]=\E^P[(N^{\boldsymbol{\xi}}\phi(t)-N^{\boldsymbol{\xi}}\phi(s))\mathcal{H}]=0,  
		\end{align*}
where $\mathcal{H}(\omega):=H(Y(\omega),\delta_{\alpha_t}(\omega)(\d u)\d t, A(\omega))$ is also $\F_s$-measurable due to the progressive measurability of $\alpha$. Hence, we conclude that $Q\in R(\boldsymbol{\xi})$. Moreover, it holds that
\begin{align*}
\mathcal{J}(Q;\boldsymbol{\xi})&=\E^Q\left[\int_0^T\int_U\tf(t,X_t,\xi_t,u)\Lambda_t(\d u)\d t+\int_0^Tc(t,X_t)\d R^A_t+g(X_T,\mu_T)\right]\\
&=\E^Q\left[\int_0^T\int_Uf_0(t,X_t,\mu_t,u)\Lambda_t(\d u)\d t+\int_0^Tc(t,X_t)\d R^A_t+g(X_T,\mu_T)\right]\\
&\quad+\lambda_3\int_0^T\int_{\R\times\Pc(U)}\int_Uf_0(t,x,\mu_t,u)q(\d u)\xi_t(\d x,\d q)\d t\\
&=\E^P\left[\int_0^T\int_Uf_0(t,X_t,\mu_t,u)\delta_{\alpha_t}(\d u)\d t+\int_0^Tc(t,X_t)\d R^A_t+g(X_T,\mu_T)\right]\\
&\quad+\lambda_3\int_0^T\int_{\R\times\Pc(U)}\int_Uf_0(t,x,\mu_t,u)q(\d u)\xi_t(\d x,\d q)\d t\\
&\leq\E^P\left[\int_0^T\int_U\tf(t,X_t,\xi_t,u)\Lambda_t(\d u)\d t+\int_0^Tc(t,X_t)\d R^A_t+g(X_T,\mu_T)\right]=\mathcal{J}(P;\boldsymbol{\xi}).
\end{align*}
Next, we turn to the confirm the second assertion. Let $\alpha=(\alpha_t)_{t\in [0,T]}$ be the measurable selection constructed above. Set $\boldsymbol{\rho}^*=P^*\circ (X,\alpha)^{-1}$ and $Q^*=P^*\circ (Y,\delta_{\alpha_t}(\d u)\d t, W, A)^{-1}$. It is easy to see that $\boldsymbol{\rho}^*|_{\C}=\boldsymbol{\xi}^*|_{\C}$ for $t\in [0,T]$ as $\boldsymbol{\xi}^*=P^*\circ(X,\Lambda)^{-1}$.
Then, we claim 
\begin{align}\label{equality}
\mathcal{J}(P;\boldsymbol{\xi}^*)=\mathcal{J}(P;\tilde{\boldsymbol{\xi}}),\quad \forall P\in\Pc_2(\Omega)
\end{align}
with $\tilde{\boldsymbol{\xi}}=\mathscr{R}(\boldsymbol{\rho}^*)$ for $t\in [0,T]$.
We first note that
\begin{align*}
&\int_{\R\times \Pc(U)}\int_Ub_0(t,x,\mu_t^*,u)q(\d u)\xi_t^*(\d x,\d q)=\E^{P^*}\left[\int_Ub_0(t,X_t,\mu_t^*,u)\Lambda_t(\d u)\right]=\E^{P^*}\left[b_0(t,X_t,\mu_t^*,\alpha_t)\right]\\
&\qquad=\int_{\R\times U}b_0(t,x,\mu_t^*,u)\rho_t^*(\d x,\d u)=\int_{\R\times \Pc(U)}\int_Ub_0(t,x,\mu_t^*,u)q(\d u)\tilde{\xi}_t(\d x,\d q),
\end{align*}
and similar results hold for $\sigma_0^2$ and $f_0$. Consequently, we have
\begin{align*}
\mathcal{J}(P;\boldsymbol{\xi}^*)&=\E^P\left[\int_0^T\int_Uf_0(t,X_t,\mu_t^*,u)\Lambda_t(\d u)\d t+\int_0^Tc(t,X_t)\d R^A_t+g(X_T,\mu_T)\right]\\
&\qquad+\lambda_3\int_0^T\int_{\R\times\Pc(U)}\int_Uf_0(t,x,\mu_t^*,u)q(\d u)\xi_t^*(\d x,\d q)\d t\\
&=\E^P\left[\int_0^T\int_Uf_0(t,X_t,\mu_t^*,u)\Lambda_t(\d u)\d t+\int_0^Tc(t,X_t)\d R^A_t+g(X_T,\mu_T)\right]\\
&\qquad+\lambda_3\int_0^T\int_{\R\times\Pc(U)}\int_Uf_0(t,x,\mu_t^*,u)q(\d u)\tilde{\xi}_t(\d x,\d q)\d t\nonumber\\
&=\mathcal{J}(P;\tilde{\boldsymbol{\xi}}).
\end{align*}
We next check that $(\boldsymbol{\rho}^*,Q^*)$ is indeed an (SW)-MFE. Similar to the proof of \equref{equality}, one can verify that, for all $\phi\in C_b^2(\R)$,  $M^{\boldsymbol{\xi}^*}\phi(t,\omega)=M^{\tilde{\boldsymbol{\xi}}}\phi(t,\omega)$ for $(t,\omega)\in [0,T]\times\Omega$. Therefore, we obtain $R(\boldsymbol{\xi}^*)=R(\tilde{\boldsymbol{\xi}})$ and $R^{\rm s}(\boldsymbol{\xi}^*)=R^{\rm s}(\tilde{\boldsymbol{\xi}})$, thus $Q^*\in R^{\rm s}(\tilde{\boldsymbol{\xi}})$, which verifies item (i) of \defref{strict_MFE}. Using the construction of $Q^*$ and $\boldsymbol{\rho}^*$, we can validate item (iii) of \defref{strict_MFE}. Finally, item (ii) of \defref{strict_MFE} follows from that
{\small\begin{align}\label{equivalence}
\mathcal{J}(Q^*,\tilde{\boldsymbol{\xi}})=\mathcal{J}(Q^*,\boldsymbol{\xi}^*)=\mathcal{J}(P^*,\boldsymbol{\xi}^*)&=\inf_{P\in R(\boldsymbol{\xi}^*)}\mathcal{J}(P,\boldsymbol{\xi}^*)
=\inf_{Q\in R(\tilde{\boldsymbol{\xi}})}\mathcal{J}(Q,\tilde{\boldsymbol{\xi}})\leq \inf_{Q\in R^{\rm s}(\tilde{\boldsymbol{\xi}})}\mathcal{J}(Q,\tilde{\boldsymbol{\xi}}),
\end{align}}where we have used \equref{equality} and the second equality can be easily checked. Hence, $Q^*$ is optimal among $R^{\rm s}(\tilde{\boldsymbol{\xi}})$ as $Q^*\in R^{\rm s}(\tilde{\boldsymbol{\xi}})$. Thus, the proof of the lemma is complete.
\end{proof}
	
\begin{proof}[Proof of \corref{existence_Markovian}]
We can follow and modify some arguments of Corollary 3.8 in Lacker \cite{Lacker} to give a proof in our context. In light of \thmref{existence_RMFE}, there exists an (R)-MFE $(\boldsymbol{\xi}^*,P^*)$. By theorem 2.5 of Karoui \cite{Karoui}, there exists $m\in\mathbb{N}$ and a measurable function $\bar{\sigma}:[0,T]\times\R\times\Pc_2(\R\times U)\times\Pc(U)$ such that $\bar{\sigma}$ is jointly continuous, uniformly continuous in $\rho\in\Pc_2(\R\times U)$ with respect to $(t,x,q)\in[0,T]\times\R\times\Pc(U)$ and $\bar{\sigma}^2(t,x,\rho,q)=\int_U\sigma^2(t,x,\rho,u)q(\d u)$ with $\bar{\sigma}(t,x,\rho,\delta_u)=\sigma(t,x,\rho,u)$. Moreover, we can find  a filtered probability space $(\Omega',\F',\Fb',P')$ supporting a scalar standard $\Fb'$-Brownian motion $W^*=(W_t^*)_{t\in [0,T]}$, a continuous $\Fb'$-adapted real valued process $A^*=(A^*_t)_{t\in [0,T]}$, a $\Pc(U)$-valued $\Fb'$-progressively measurable process $\Lambda^*=(\Lambda^*_t)_{t\in [0,T]}$ and an $\F_0'$-measurable $\R$-valued $\RV$ $\eta^*$ with $\Law^{P'}(\eta^*)=\Law^{\Pb}(\eta)$ such that $P^*=P'\circ(Y^*,\Lambda_t^*(\d u)\d t,W,^*,A^*)^{-1}$, where $Y^*=(Y_t^*)_{t\in [0,T]}$ is the unique strong solution to the SDE:
		\begin{align*}
			\d Y_t^*=\int_U\tb(t,X_t^*,\xi_t^*,u)\Lambda_t^*(\d u)\d t+\tilde{\bar{\sigma}}(t,X_t^*,\xi_t^*,\Lambda_t^*)\d W_t^*,~~Y_0^*=\eta^*    
		\end{align*}
with $\tilde{\bar{\sigma}}$ being the extension of $\bar{\sigma}$ and $X_t^*=\Gamma(A^*,Y^*)_t$.  Same as Corollary 3.8 in Lacker \cite{Lacker}, there exists a measurable function $q:[0,T]\times\R\mapsto\Pc(U)$ such that $q(t,X_t^*)(\d u)=\E^{P'}\left[\Lambda_t^*(\d u)\middle|X_t^*\right]$, ${\tt m}\times P'\mbox{-}\as$. 
In particular, it holds that, $\d t\times\d P'$-$\as$,
\begin{align}\label{invariance}
\int_U\tb(t,X_t^*,\xi_t^*,u)q(t,X_t^*)(\d u)&=\E^{P'}\left[\int_U\tb(t,X_t^*,\xi_t^*,u)\Lambda_t^*(\d u)\middle|X_t^*\right]\nonumber,\\ 
\tilde{\bar{\sigma}}^2(t,X_t^*,\xi_t^*,q(t,X_t^*))&=\E^{P'}\left[\tilde{\bar{\sigma}}^2(t,X_t^*,\xi_t^*,\Lambda_t^*)\middle|X_t^*\right]=\E^{P'}\left[\int_U\ts^2(t,X_t^*,\xi_t^*,u)\Lambda_t^*(\d u)\middle|X_t^*\right]\nonumber,\\
\int_U\tf(t,X_t^*,\xi_t^*,u)q(t,X_t^*)(\d u)&=\E^{P'}\left[\int_U\tf(t,X_t^*,\xi_t^*,u)\Lambda_t^*(\d u)\middle|X_t^*\right].
\end{align}
Applying the mimicking theorem in Corollary 3.7 of Brunick and Shreve \cite{Brunick}, we find another probability space $(\Omega^1,\F^1,\Fb^1,P^1)$ supporting a scalar standard $\Fb^1$-Brownian motion $W^1=(W^1_t)_{t\in [0,T]}$, a continuous $\Fb^1$-adapted real valued process $A^1=(A^1_t)_{t\in [0,T]}$, and an $\F_0^1$-measurable $\R$-valued $\RV$ $\eta^1$ with $\Law^{P'}(\eta^1)=\Law^{\Pb}(\eta)$ such that $\Law^{P^1}(X_t^1)=\Law^{P'}(X_t^*)=\Law^{P^*}(X_t)$ for all $t\in [0,T]$, where $X_t^1=\Gamma(A^1,Y^1)_t$ and $Y^1=(Y_t^1)_{t\in [0,T]}$ is the unique strong solution to the following SDE:
		\begin{align*}
			\d Y_t^1=\int_U\tb(t,X_t^1,\xi_t^*,u)q(t,X_t^1)(\d u)\d t+\tilde{\bar{\sigma}}(t,X_t^1,\xi_t^*,q(t,X_t^1))\d W_t,\quad Y_0^1=\eta^1.    
		\end{align*}
		Set $\boldsymbol{\xi}^1=P\circ (X^1,q(t,X_t^1)(\d u)\d t)^{-1}$ and $Q^*=P^1\circ(Y^1,q(t,X_t^1)(\d u)\d t,A^1)^{-1}$. By (B1) of \assref{ass2}, \equref{equivalence} and \equref{invariance}, we can follow the similar proof as the one for \lemref{measurable} that $Q^*\in R(\boldsymbol{\xi}^1)$ and $\mathcal{J}(Q^*;\boldsymbol{\xi}^1)=\mathcal{J}(P^*,\boldsymbol{\xi}^*)=\inf_{Q\in R(\boldsymbol{\xi}^1)}\mathcal{J}(Q,\boldsymbol{\xi}^1)$. Hence, $(\boldsymbol{\xi}^1,Q^*)$ is an (R)-MFE. The rest of proof follows from the measurable selection argument in \lemref{measurable} and Corollary 3.8 of \cite{Lacker}.
	\end{proof}

\ \\
\noindent\textbf{Acknowledgements} L. Bo and J. Wang are supported by National Natural Science of Foundation of China (No. 12471451), Natural Science Basic Research Program of Shaanxi (No. 2023-JC-JQ-05), Shaanxi Fundamental Science Research Project for Mathematics and Physics (No. 23JSZ010). X. Yu and J. Wang are supported by the Hong Kong RGC General Research Fund (GRF) under grant no. 15306523 and grant no. 15211524 and the Hong Kong Polytechnic University research grant under no. P0045654.	

\ \\

\end{document}